\documentclass[a4paper]{article}

\usepackage{natbib}
\usepackage{bbm}
\usepackage{comment}
\usepackage[all]{xy}\usepackage[latin1]{inputenc}        %accents
\usepackage[dvips]{graphics,graphicx}
\usepackage{amsfonts,amssymb,amsmath,xcolor,mathrsfs, amstext}
\usepackage{amsbsy, amsopn, amscd, amsxtra, amsthm,authblk, enumerate,dsfont}
\usepackage{upref}
\usepackage{geometry}
\usepackage[displaymath]{lineno}
\usepackage{float}

\usepackage[colorlinks,
            linkcolor=blue,
            anchorcolor=green,
            citecolor=blue
            ]{hyperref}

\numberwithin{equation}{section}

\def\epsilon{\varepsilon}
\def\eps{\varepsilon}

\newcommand{\ol}{\overline}
\newcommand{\wt}{\widetilde}
\def\alb#1\ale{\begin{align*}#1\end{align*}}
\newcommand{\eqb}{\begin{equation}}
\newcommand{\eqe}{\end{equation}}

\newcommand{\bbC}{\mathbb{C}}
\newcommand{\bbD}{\mathbb{D}}
\newcommand{\D}{\mathbb{D}}
\newcommand{\clD}{\overline{\mathbb{D}}}

\newcommand{\E}{\mathbb{E}}

\newcommand{\R}{\mathbb{R}}

\newcommand{\N}{\mathbb{N}}
\newcommand{\Z}{\mathbb{Z}}

\renewcommand{\P}{\mathbb{P}}
\newcommand{\bbR}{\mathbb{R}}

\newcommand{\cA}{\mathcal{A}}
\newcommand{\cB}{\mathcal{B}}
\newcommand{\cC}{\mathcal{C}}
\newcommand{\cD}{\mathcal{D}}

\newcommand{\cL}{\mathcal{L}}

\newcommand{\cN}{\mathcal{N}}

\newcommand{\cS}{\mathcal{S}}

\newcommand{\qduration}{\mathcal{S}}

\newcommand{\QD}{\mathrm{QD}}
\newcommand{\GQD}{\mathrm{GQD}}
\newcommand{\QT}{\mathrm{QT}}
\newcommand{\LF}{\mathrm{LF}}
\newcommand{\SLE}{\mathrm{SLE}}
\newcommand{\QLE}{\mathrm{QLE}}

\newtheorem{theorem}{Theorem}[section]
\newtheorem{lemma}[theorem]{Lemma}
\newtheorem{proposition}[theorem]{Proposition}
\newtheorem*{proposition*}{Proposition}
\newtheorem{corollary}[theorem]{Corollary}
\newtheorem*{corollary*}{Corollary}
\newtheorem{definition}[theorem]{Definition}
\newtheorem*{definitions*}{Definitions}

\newtheorem*{example*}{Example}

\newtheorem{question}[theorem]{Question}

\theoremstyle{remark}
\newtheorem{remark}[theorem]{Remark}

\numberwithin{equation}{section}

\title{Quantum Loewner evolution in quantum natural time: phases and Markov properties}
\author{Morris Ang\thanks{University of California San Diego} \qquad Deven Mithal\thanks{University of Chicago}}
\date{}
\begin{document}

\maketitle

\begin{abstract}
Quantum Loewner evolution (QLE)$(\gamma^2, \eta)$ is a family of growth processes in random environments, introduced by Miller and Sheffield as candidate scaling limits of growth processes (such as diffusion-limited aggregation) on random planar maps. The random environments are Liouville quantum gravity (LQG) surfaces with parameter $\gamma$, and the parameter $\eta$ plays a role analogous to that in dielectric breakdown models. Their construction applies to pairs $(\gamma^2, \eta)$ lying on a curve in parameter space, and the associated time parametrization is independent of the underlying LQG surface. In later work, they defined a \emph{quantum natural time} variant of QLE$(8/3, 0)$ whose time parametrization encodes a notion of distance in the LQG geometry, leading to the identification of $\sqrt{8/3}$-LQG with the Brownian map.

In this paper we construct quantum natural time QLE$(\gamma^2, \eta)$ for a different but overlapping subset of the same parameter curve. Its time parametrization conjecturally corresponds to the scaling limit of time parametrizations of discrete growth processes on random planar maps. We prove that it exhibits three phases, mirroring those of Schramm--Loewner evolution (SLE); this answers a question of Miller and Sheffield for quantum natural time QLE. Moreover, we establish stationarity of the unexplored surface and, in the relevant phases, identify the random surfaces cut out or swallowed by the process as quantum disks. Our construction builds on recent radial LQG--SLE coupling results of Ang and Yu.
\end{abstract}

\tableofcontents

\section{Introduction}\label{sec-intro}
A growth model is a random increasing sequence of clusters on an underlying graph $G$, which is frequently taken to be a lattice such as $\Z^2$. The study of growth models is a central topic in mathematical physics, as they model a wide variety of rich physical phenomena. For example, diffusion-limited aggregation (DLA) \cite{dla-1, dla-2} models the growth of clusters formed by particles undergoing diffusion and attaching upon first contact. Such processes arise, for instance, in electrodeposition and the formation of coagulated aerosols. Another classical example is the Eden model \cite{eden}, in which particles are added uniformly at random to sites adjacent to the existing cluster, and which serves as a model for biological growth. 
In this paper we are primarily concerned with the dielectric breakdown model (DBM) \cite{dbm}, in which growth occurs at boundary sites with probability proportional to a power $\eta$ of harmonic measure. The special cases $\eta = 1$ and $\eta = 0$ correspond to DLA and the Eden model respectively. 

It is also natural to study growth processes in random environments, where the underlying graph $G$ is itself random. In particular, we will take $G$ to be a random planar map. This provides an ideal setting to study the fractal geometry of growth processes, since many kinds of random planar maps enjoy Markov properties when explored by a compatible growth process, endowing the coupling with an exact self-similar structure. For instance, a distinguished unexplored region may satisfy a stationarity property, and if there are multiple unexplored regions, they may be conditionally independent given their boundary lengths. This is the case for the Eden model on uniform random planar maps, and DLA on uniform-spanning-tree-weighted random planar maps \cite{ms-qle}.

The breakthrough work of Miller and Sheffield \cite{ms-qle} proposed a scaling limit of the $\eta$-DBM growth model on random planar maps. Random planar maps are expected (and in some cases proven) to converge to Liouville quantum gravity (LQG), a family of random continuum geometries indexed by a parameter $\gamma \in (0,2]$, each of which carries intrinsic notions of area and length. \cite{ms-qle} constructed a continuum growth process called Quantum Loewner Evolution (QLE)$(\gamma^2, \eta)$ on $\gamma$-LQG, and conjectured this to be the scaling limit of $\eta$-DBM on random planar maps in the $\gamma$-LQG universality class. Their construction only applies to some pairs $(\gamma^2, \eta)$ which lie on a curve in parameter space, see Figure~\ref{fig-quantum-vs-capacity} (right); this is due to a continuum analog of requiring the growth process to be compatible with the random environment. The random environment is a $\gamma$-LQG disk with an interior target point. For each $\delta > 0$, they define an approximation of QLE$(\gamma^2, \eta)$ by repeatedly cutting the $\gamma$-LQG surface with a Schramm-Loewner evolution (SLE) curve grown from a random boundary point for time $\delta$, and finally construct QLE$(\gamma^2, \eta)$ via a subsequential limit as $\delta \to 0$. Because each SLE curve is parametrized in a way that does not depend on the underlying $\gamma$-LQG surface (``capacity time''), the limiting growth process is also parametrized independently of the surface. We will refer to this growth process as \emph{capacity time QLE}.

In this paper we construct a growth process called \emph{quantum natural time QLE$(\gamma^2, \eta)$} for a different (but overlapping) subset of pairs $(\gamma^2, \eta)$ lying on the same curve as for capacity time QLE; see Figure~\ref{fig-quantum-vs-capacity} (left). 
This construction is novel for all parameters except $(\gamma^2, \eta) = (8/3, 0)$, for which quantum natural time QLE was constructed in \cite{ms-equivalence}. Our construction is analogous to that of capacity time QLE, with a key difference being that each SLE curve is parametrized according to {quantum natural time}, a notion of length which depends on the underlying $\gamma$-LQG surface. 
We conjecture that quantum natural time QLE$(\gamma^2, \eta)$ is the scaling limit of $\eta$-DBM on random planar maps, with the time parametrization of quantum natural time QLE$(\gamma^2, \eta)$ being the scaling limit of that of the discrete growth model.

In Theorem~\ref{thm-phase}, we prove that quantum natural time QLE exhibits three phases. Consider the $\gamma$-LQG area of the QLE process as a function of time. In the \emph{dilute} phase, the area function is zero for all time. In the \emph{swallowing} phase, the area function is a strictly increasing pure jump process; each jump corresponds to a time when the QLE process separates some region from the target point and immediately absorbs that region. In the \emph{space-filling phase}, the area function is strictly increasing and continuous, and the process eventually covers the whole surface. These three phases mirror the phases of SLE. Theorem~\ref{thm-phase} answers a question of \cite[Section 8.2]{ms-qle} for quantum natural time QLE; the question remains open for capacity time QLE. 

Finally, we show in Theorems~\ref{thm-stationarity}-\ref{thm-swallowing-markov} that quantum natural time QLE enjoys certain Markov properties. First, the unexplored region containing the target point has a stationarity property (Theorem~\ref{thm-stationarity}). Moreover, in the dilute phase the complement of the whole QLE process is a $\gamma$-LQG disk (Theorem~\ref{thm-simple-markov}), and in the swallowing phase the swallowed regions are conditionally independent $\gamma$-LQG disks given their boundary lengths (Theorem~\ref{thm-swallowing-markov}). In contrast, capacity time QLE only satisfies a stationarity property which is analogous to Theorem~\ref{thm-stationarity} but weaker in that one cannot keep track of natural geometric quantities such as $\gamma$-LQG area and length. 

We expect that quantum natural time QLE has substantial advantages over capacity time QLE in the study of scaling limits of growth processes on random planar maps. Indeed, its time parametrization should be the scaling limit of the discrete time parametrization, and its Markov properties mirror those of the Eden model on uniform random planar maps and DLA on uniform-spanning-tree-weighted random planar maps; see Section~\ref{subsec-intro-discrete}.

As a key component in our argument, we use radial variants of Sheffield's \emph{length quantum zipper} \cite[Theorem 1.8]{shef-zipper}: for parameters $\gamma \in (0,2)$ and $\kappa \in \{\gamma^2, 16/\gamma^2\}$, when one samples a certain $\gamma$-LQG surface with an independent radial $\SLE_\kappa$ curve, under the dynamics of cutting along the curve according to quantum natural time, the surface exhibits a certain stationarity in law. For $\kappa \in (0,8) \backslash \{4\}$ we construct this radial quantum zipper using recent results from \cite{ay-radial}, whereas for $\kappa > 8$ the quantum zipper already appears in \cite{ay-reversibility}.

In Section~\ref{subsec-intro-construction} we outline our construction of quantum natural time QLE. In Section~\ref{subsec-intro-comparison} we compare quantum natural time QLE and the capacity time QLE of \cite{ms-qle}. In Section~\ref{subsec-intro-phases} we discuss the phases of quantum natural time QLE, and in Section~\ref{subsec-intro-markov} state its Markov properties. Finally, in Section~\ref{subsec-intro-continuum} we discuss the quantum zipper results underlying our construction.

\subsection{Construction of quantum natural time QLE}\label{subsec-intro-construction}
Liouville quantum gravity (LQG) is a model of random surfaces originating from Polyakov's seminal work on bosonic string theory \cite{polyakov-qg1}. Heuristically, it is the random geometry described by the Riemannian metric tensor $e^{\gamma \phi} (dx^2 + dy^2)$ where $\gamma \in (0,2]$ and $\phi$ is a variant of the free boundary Gaussian free field (GFF) on the unit disk. This does not make literal sense because $\phi$ is a distribution-valued random variable which is too rough to admit pointwise values. Despite that, one can still make sense of certain geometrical quantities by regularization and renormalization, such as the $\gamma$-LQG area measure $\cA^\gamma_\phi$ on $\D$ and the $\gamma$-LQG boundary length measure $\cL^\gamma_\phi$ on $\partial \D$; see Section~\ref{subsec-lqg-gff} for details. 
It is widely believed that $\gamma$-LQG describes the scaling limits of various kinds of random planar maps, and this has been proved in several cases for different notions of convergence. For instance, uniform-spanning-tree-decorated random planar maps converge in the peanosphere topology to SLE$_\kappa$-loop-decorated $\gamma$-LQG where $(\gamma, \kappa) = (\sqrt2, 8)$ \cite{mullin-maps, shef-burger, DMS14}.

Our first main contribution is the construction of a growth process in a $\gamma$-LQG random environment, called \emph{quantum natural time QLE$(\gamma^2, \eta)$}, for parameters $(\gamma, \eta)$ in 
\eqb\label{eq-gamma-eta}
\{(\gamma, \eta) \: : \: \gamma \in (\sqrt3 -1 , 2), \, \eta = \frac{3}{\gamma^2} - \frac12\} \cup \{(\gamma, \eta) \: : \: \gamma \in (0, 4/3) \cup (2\sqrt3 - 2, 2), \, \eta = \frac{3\gamma^2}{16} - \frac12\}.
\eqe
As we later discuss, our construction is a variant of the \emph{capacity time QLE$(\gamma^2, \eta)$} constructed in \cite{ms-qle}. In this paper, the term QLE without a modifier will always refer to quantum natural time QLE.

A QLE$(\gamma^2, \eta)$ process is a random pair $(\phi, (K_\cdot)_{[0,\qduration]})$, where $\phi$ is a field on $\D$ describing a $\gamma$-LQG surface, and $(K_\cdot)_{[0,\qduration]}$ is a growth process targeting the origin whose evolution depends on the random environment. The process has a random duration $\qduration$. For each time $s < \qduration$ the set $K_s$ is a hull (compact subset of $\ol \D$ whose complement is simply connected and contains 0), the final state $K_\qduration$ is a compact subset of $\ol \D$ containing the origin, and $K_s \subset K_{s'}$ for $s \leq s'$. In our construction, the field $\phi$ will be a free boundary Gaussian free field on $\D$ with an $\alpha$-log singularity added to the origin, weighted by its boundary Gaussian multiplicative chaos with parameter $\beta/2$, and conditioned on having boundary length 1; the parameters $(\alpha, \beta)$ are given by 
        \begin{equation} \label{eq-alpha-beta} (\alpha, \beta) = 
        \begin{cases} 
          (Q-\frac{1}{\gamma},\gamma - \frac2\gamma) & \gamma \in (\sqrt3 - 1,2), \, \, \eta = \frac3{\gamma^2} - \frac12 \\
            (Q-\frac{\gamma}{4},\frac4\gamma - \frac\gamma2) & \gamma \in (2\sqrt3 - 2, 2), \, \, \eta = \frac{3\gamma^2}{16} - \frac12 \\
          (Q-\frac{\gamma}{4},\frac{3\gamma}{2}) & \gamma \in (0, 4/3), \, \,  \eta = \frac{3\gamma^2}{16} - \frac12.
       \end{cases}
    \end{equation} 
The law $P_{\alpha, \beta, 1}$ of $\phi$ is precisely described in Definition~\ref{def-P-weighted}.

First, we construct a $\delta$-approximation of QLE via a \emph{quantum zipper} result wherein a $\gamma$-LQG surface is cut by an independent radial SLE$_\kappa$ curve, with parameters coupled via $\kappa \in \{ \gamma^2, 16/\gamma^2\}$; see Section~\ref{subsec-intro-continuum} for more information. Let $\delta > 0$. For each pair $(\gamma, \eta)$ as in~\eqref{eq-gamma-eta}, we set $\kappa = \frac{6}{2\eta + 1}$ and define the $\delta$-QLE$(\gamma^2, \eta)$ process as follows, see Figure~\ref{fig-delta-qle}.
\begin{itemize}
    \item Sample $\phi^\delta \sim P_{\alpha, \beta, 1}$ as in Definition~\ref{def-P-weighted}, and let $K_0^\delta = \partial \D$. 
    \item Suppose the growth process $(K^\delta_\cdot)$ has been defined for an interval of time $[0,n \delta]$. Let $D^\delta_{n \delta} = \D \backslash K^\delta_{n\delta}$. We first sample a point $p_n^\delta \in \partial D^\delta_{n\delta}$ according to a certain boundary measure defined by $\phi^\delta$, then grow a radial $\SLE_\kappa$ curve in $(D^\delta_{n\delta}, p_n^\delta, 0)$ independent of everything else. This curve is parametrized by its quantum natural time with respect to $\phi^\delta$. The process $(K^\delta_\cdot)$ is then defined on $[n\delta, (n+1)\delta]$ by setting $K^\delta_{n\delta + s}$ to be the union of $K^\delta_{n\delta}$, the curve run until time $s$, and all complementary connected components of the curve not containing 0. 
    \item The previous step is iterated until the process hits $0$, at which time the process terminates. We call this hitting time $\qduration^\delta$ the duration of the process. 
\end{itemize}
We say $(\phi^\delta, (K^\delta_\cdot)_{[0,\qduration^\delta]})$ is a $\delta$-QLE$(\gamma^2, \eta)$ process. 
Precise details of this construction are given in Section~\ref{sec-qle-approximates}. Our analysis of $\delta$-QLE$(\gamma^2, \eta)$ depends on a radial quantum zipper which we discuss in Section~\ref{subsec-intro-continuum}.

\begin{remark}
Each point $p_n^\delta \in \partial D^\delta_{n\delta}$ is sampled from a measure that, at least heuristically, corresponds to $(\frac{d\nu}{d\mu})^\eta \mu$ where $\mu$ is the $\gamma$-LQG boundary length measure on $\partial D^\delta_{n \delta}$ and $\nu$ is the harmonic measure on $\partial D^\delta_{n \delta}$ viewed from $0$; see Section~\ref{sec-scaling-rels}. Therefore, $\delta$-QLE$(\gamma^2, \eta)$ can heuristically be viewed as an $\eta$-DBM process where at each step, a boundary point is sampled from $(\frac{d\nu}{d\mu})^\eta \mu$, and a size-$\delta$ particle (the SLE curve segment) is attached at that point; we emphasize that the size of the particle is measured with respect to the $\gamma$-LQG environment since the SLE curve is parametrized by quantum natural time. 
\end{remark}

\begin{figure}[ht]
	\centering
\includegraphics[scale=0.48]{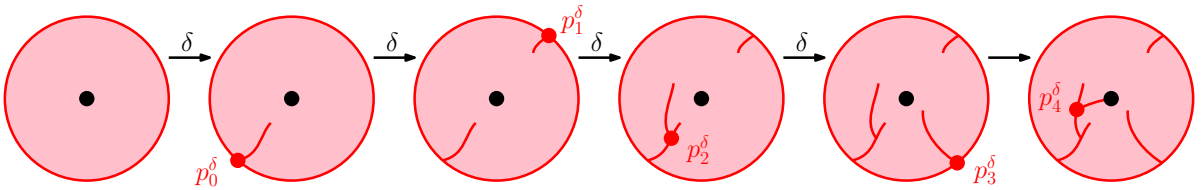}
	\caption{For simplicity we illustrate the $\kappa < 4$ case; the others are similar. The $\delta$-QLE$(\gamma^2, \eta)$ process $(\phi^\delta, (K^\delta_\cdot)_{[0,\qduration^\delta]})$ is depicted at $\delta$ increments of time, with each $K^\delta_\cdot$ shown in red. Starting with $K_0 = \partial \D$, each $K^\delta_{(n+1)\delta}$
    is obtained by iteratively sampling a boundary point $p_{n\delta}^\delta$ of $\D \backslash K_{n\delta}^\delta$ according to a length measure determined by $\phi^\delta$, then growing a radial SLE$_\kappa$ in $\D \backslash K_{n\delta}^\delta$ from $p_{n\delta}^\delta$ targeting 0 for quantum natural time $\delta$. The process terminates when it hits 0, at time $\qduration^\delta$.
    } \label{fig-delta-qle}
\end{figure}

Our first main result is that one can indeed construct a growth process by sending $\delta \to 0$:

\begin{theorem}
There exist subsequential limits of $\delta$-QLE$(\gamma^2, \eta)$ as $\delta \to 0$. 
\end{theorem}
\begin{proof}
    This is immediate from the tightness of the laws of $\delta$-QLE$(\gamma^2, \eta)$ shown in Lemma~\ref{lem-tight-sf}. 
\end{proof}

We define QLE$(\gamma^2, \eta)$ by taking such a subsequential limit in Section~\ref{sec-phases}. As we explain in Section~\ref{subsec-intro-comparison}, QLE$(\gamma^2, \eta)$ exhibits three different phases as $(\gamma, \eta)$ vary in~\eqref{eq-gamma-eta}; our choice of subsequence depends slightly on the parameter range (Definitions~\ref{def-qle-sf},~\ref{def-qle-simple},~\ref{def-qle-swallowing}). It is natural to conjecture, however, that the convergence holds without passing to a subsequence.

The construction of quantum natural time QLE$(\gamma^2, \eta)$ was previously only carried out for $(\gamma^2, \eta) = (8/3, 0)$ \cite{ms-equivalence}; see Appendix~\ref{sec-other-QLEs} for details. For this choice of parameters, it is known that the limit does not depend on the choice of subsequence.

\subsection{Comparison with capacity time QLE}\label{subsec-intro-comparison}

Capacity time QLE$(\gamma^2, \eta)$ was introduced in \cite{ms-qle} as a subsequential limit of a variant of $\delta$-QLE for which each SLE segment is grown until it has a log conformal radius of $\delta$ when viewed from 0 (rather than quantum natural time $\delta$). It is very similar to our quantum natural time QLE$(\gamma^2, \eta)$ in that the parameter $\eta$ must lie in $\{\frac{3}{\gamma^2}-\frac12, \frac{3\gamma^2}{16} - \frac12\}$, and the size $\alpha$ of the log singularity at the origin is the same. 
It is an open problem as to whether capacity time QLE and quantum natural time QLE agree up to time parametrization, see Question~\ref{question-equiv} and  subsequent discussion.

We now discuss the parameter ranges for which quantum natural time and capacity time QLE have been nontrivially constructed\footnote{\cite{ms-equivalence} also defines capacity time QLE$(\gamma^2, \eta)$ for $\gamma \in [0,2]$ and $\eta = -1$ as the process where the hulls are deterministic concentric annuli with no dependence on $\phi$; we omit this regime in our discussion.}, see Figure~\ref{fig-quantum-vs-capacity}. For capacity time QLE, this parameter range is 
\eqb \label{eq-capacity-parameters}
\{(\gamma, \eta) \: : \: \gamma \in (1,2), \, \eta = \frac{3}{\gamma^2} - \frac12\} \cup \{(\gamma, \eta) \: : \: \gamma \in (0, 2], \, \eta = \frac{3\gamma^2}{16} - \frac12\}.
\eqe
In particular,~\eqref{eq-capacity-parameters} omits the values $\{ (\gamma, \eta) \:: \:  \gamma \in (0 , 1], \eta = \frac{3}{\gamma^2} - \frac12\}$; for these parameters the construction of \cite{ms-qle} becomes trivial\footnote{The construction of \cite{ms-qle} in this range gives a simple radial SLE independent of $\phi$. See \cite[Section 1.5]{ms-equivalence} for discussion on whether this trivial construction is related to scaling limits of discrete models.}. In contrast, the parameter range~\eqref{eq-gamma-eta} for which we construct quantum natural time QLE$(\gamma^2, \eta)$ includes a subset of these omitted values, namely $\{ (\gamma, \eta) \:: \:  \gamma \in (\sqrt3-1 , 1], \eta = \frac{3}{\gamma^2} - \frac12\}$. On the other hand, quantum natural time QLE cannot be constructed for the parameters $\{ (\gamma^2, \eta) \: : \: \gamma \in [4/3, 2\sqrt3 - 2], \eta = \frac{3\gamma^2}{16} - \frac12\}$; the obstruction is the same as that of capacity time QLE for the previously-discussed parameters. Finally, our construction does not address the case $(\gamma^2, \eta) = (4, 1/4)$ as it would require a generalization of the input from \cite{ay-radial} to the critical parameter $\gamma = 2$. It is currently unknown whether QLE agrees with capacity time QLE for the parameters $(\gamma^2, \eta)$ for which both are defined (Question~\ref{question-equiv}). 

\begin{figure}[ht]
	\centering
\includegraphics[scale=0.8]{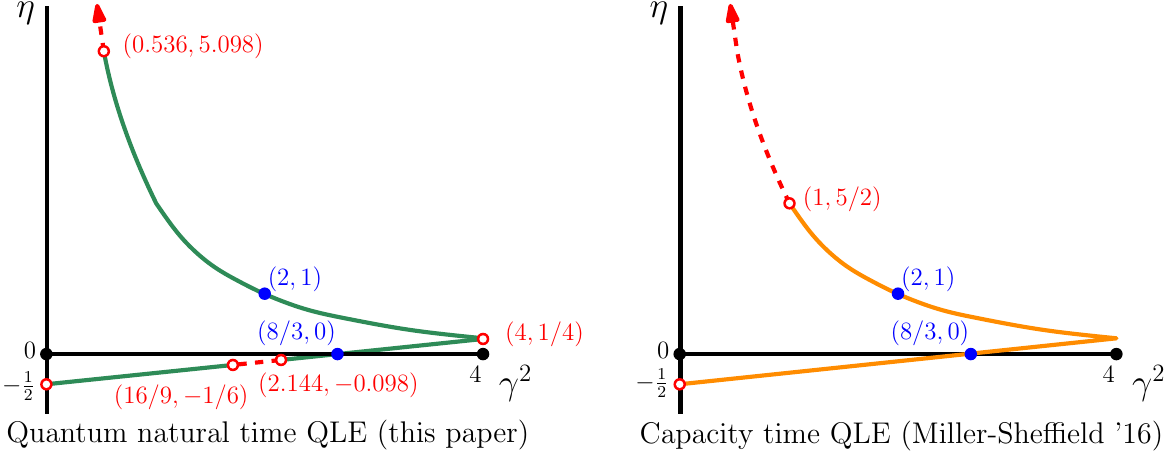}
	\caption{Two graphs depicting the ranges of parameters $(\gamma^2, \eta)$ for which quantum natural time and capacity time QLE$(\gamma^2, \eta)$ have been nontrivially constructed. For both graphs, the top curve is $\eta = \frac{3}{\gamma^2} - \frac12$, and the bottom curve is $\eta = \frac{3\gamma^2}{16} - \frac12$. 
    The red regions on the curves represent parameter values for which the process is not constructed. Capacity time QLE was constructed in~\cite{ms-qle}, and prior to the present work, quantum natural time QLE had only been constructed for the parameters $(\gamma^2, \eta) = (8/3,0)$ \cite{ms-equivalence}.
    Figure adapted from \cite{ms-qle}.
    } \label{fig-quantum-vs-capacity}
\end{figure}

A main result of \cite{ms-qle} is that capacity time QLE solves a stochastic differential equation called the QLE dynamics. Proving this for quantum natural time QLE seems challenging: the SDE describes a growth process evolving in capacity time, whereas our process is parametrized by quantum natural time. 

On the other hand,  quantum natural time QLE is more natural than capacity time QLE in terms of conjectural connections to random planar maps decorated by growth models. Indeed, the quantum natural time parametrization should correspond to the scaling limit of the time parametrization of the discrete model. For instance, a uniform random planar map decorated by metric ball growth should converge, as a time-parametrized growth process, to $(\gamma = \sqrt{8/3})$-LQG decorated by quantum natural time QLE$(8/3,0)$; this is known for many kinds of uniform random planar maps\footnote{Many kinds of (rooted) uniform random planar maps are known to converge as metric spaces to the (pointed) Brownian map with respect to the Gromov-Hausdorff metric; see for instance \cite{legall-uniqueness,miermont-brownian-map}. As an easy consequence, if we decorate each metric space by a metric ball growth process, the decorated metric spaces converge in law with respect to a minor variant of the Gromov-Hausdorff metric where, when we embed two growth-process-decorated metric spaces in a common metric space so their Hausdorff distance is small, we further require the two processes to be uniformly close in Hausdorff distance for all times. One might, more strongly, be able to prove convergence of conformally embedded random planar maps decorated by metric ball growth; see \cite{hs-cardy-embedding}.}. 
Similarly a uniform-spanning-tree-weighted random planar map decorated by a DLA process should converge to $(\gamma =\sqrt2)$-LQG decorated by quantum natural time QLE$(2, 1)$; see Question~\ref{question-scaling}. 

\subsection{Phases}\label{subsec-intro-phases}
In Theorem~\ref{thm-phase} below we prove that QLE$(\gamma^2, \eta)$ has three phases mirroring those of SLE$_\kappa$. 

We first recall the phases of $\SLE_\kappa$; see for instance \cite{lawler2008conformally, rohde-schramm-basic-sle}. For the simple phase $0 < \kappa \leq 4$,  SLE$_\kappa$ is a simple curve. For the swallowing phase $4 < \kappa < 8$, $\SLE_\kappa$ is a curve that intersects itself but has zero Lebesgue measure. For the space-filling phase $\kappa \geq 8$, $\SLE_\kappa$ is  a space-filling curve. 
These three properties are formulated in terms of curves; we now state a variant better adapted to the setting of growth processes. For a radial SLE$_\kappa$ curve in $\D$ from 1 to 0, its hull at time $t$ is the union of the curve at time $t$ and all its complementary connected components not containing 0. Denote the hull process by $(K^{\mathrm{SLE}_\kappa}_\cdot)_{[0,\infty)}$, and denote the two-dimensional Lebesgue measure by $\lambda$.
\begin{itemize}
    \item (Simple phase) For $0 < \kappa \leq 4$, we have $\lambda(K^{\mathrm{SLE}_\kappa}_t) = 0$  for all $t$ a.s. 
    \item (Swallowing phase) For $4 < \kappa < 8$, a.s.\ the function $t \mapsto \lambda(K_t)$ is strictly increasing, has derivative zero almost everywhere, and $K_t$ increases to $\D \backslash \{0\}$ as $t \to \infty$. 
    \item (Space-filling phase) For $\kappa \geq 8$, a.s.\ the function $t \mapsto \lambda(K_t)$ is continuous and strictly increasing, and $K_t$ increases to $\D \backslash \{0\}$ as $t \to \infty$.
\end{itemize}
Note that in the swallowing phase, the jump times of $t \mapsto \lambda(K_t)$ correspond to times where the curve hits itself, hence separating some region from 0.

As explained in \cite[Section 8.2]{ms-qle}, it is natural to conjecture that  QLE$(\gamma^2, \eta)$ also has three phases, mirroring those of the $\SLE_\kappa$ processes used in the definition of QLE$(\gamma^2, \eta)$. Our next theorem verifies that this is indeed the case for quantum natural time QLE$(\gamma^2, \eta)$. Let $\cA^\gamma_\phi$ be the $\gamma$-LQG area measure associated to $\phi$. 

\begin{theorem}\label{thm-phase} QLE$(\gamma^2, \eta)$ exhibits three phases:
\begin{itemize}
    \item (Dilute phase) For $\gamma \in (\sqrt3 - 1,2)$ and $\eta = \frac3{\gamma^2} - \frac12$, we have $\cA^\gamma_\phi(K_\qduration) = 0$ a.s. 
    \item (Swallowing phase) For $\gamma \in (2\sqrt3-2,2)$ and $\eta = \frac{3\gamma^2}{16} - \frac12$, a.s.\ the function $s \mapsto \cA^\gamma_\phi(K_s)$ is strictly increasing, has derivative zero almost everywhere, and is continuous at $\qduration$. Moreover $K_\qduration = \ol \D$.
    \item (Space-filling phase) For $\gamma \in (0,4/3)$ and $\eta = \frac{3\gamma^2}{16} - \frac12$, a.s.\ we have $\cS = \cA^\gamma_\phi(\D)$ and $\cA^\gamma_\phi(K_s) = s$ for all $s \in [0,\qduration]$, and $K_\qduration = \ol \D$.
\end{itemize}
\end{theorem}
The corresponding result has not yet been proven for capacity time QLE$(\gamma^2, \eta)$. Our proof depends on Markov properties of QLE$(\gamma^2, \eta)$; a related Markov property for capacity time QLE$(\gamma^2, \eta)$ exists but is inherently more limited. 

\subsection{Markov properties}\label{subsec-intro-markov}
We now give some Markov properties for QLE, starting with a stationarity property for a sample $(\phi, (K_\cdot)_{[0,\qduration]})$ of QLE$(\gamma^2, \eta)$. Their statements require the notion of $\gamma$-LQG surface, which we now introduce briefly; a more detailed treatment is given in Section~\ref{subsec-lqg-gff}.

Consider two planar domains $D, \wt D \subset \bbC$. Given a conformal map $g: D \to \tilde D$ and a distribution $h$ on $D$, we define a distribution $ g \bullet_\gamma h$ on $\tilde D$ by
\[g \bullet_\gamma h := h \circ g^{-1} + Q \log |(g^{-1})'|, \qquad Q = \frac\gamma2 + \frac2\gamma.\]
A $\gamma$-LQG surface is a $\sim_\gamma$-equivalence class of pairs $(D, h)$, where $(D, h) \sim_\gamma (\tilde D, \tilde h)$ if $\tilde h = g \bullet_\gamma h$ for some conformal map $g$. The $\gamma$-LQG area measure $\cA^\gamma_h$ is compatible with this definition in that $\tilde h = g \bullet_\gamma h$ implies $\cA^\gamma_{\tilde h} = g_* \cA_h$, and an analogous statement holds for the $\gamma$-LQG boundary length measure $\cL^\gamma_h$.

Recall that the law of $\phi$ is $P_{\alpha, \beta, 1}$ where $(\alpha, \beta)$ are as in~\eqref{eq-alpha-beta} and $P_{\alpha,\beta,1}$ is as in Definition~\ref{def-P-weighted}. For each time $s < \qduration$, we can define the field $\phi_s$ on $\D$ by $\phi_s = g \bullet_\gamma \phi$  where $g: \D \backslash K_s \to \D$ is the conformal map fixing $0$ with $g'(0) >0$.

We can define a boundary length process $(L_\cdot)_{[0,\qduration]}$ associated to $(\phi, (K_\cdot)_{[0,\qduration]})$ by measuring the $\gamma$-LQG boundary length $\cL^\gamma_{\phi_s}(\partial \D)$ at rational times, then extending to all times (Corollary~\ref{cor-length-process}). For $s>0$, on $\{ s < \qduration\}$ we identify the conditional law of $\phi_s$ given the boundary length process $(L_\cdot)_{[0,s]}$. 

\begin{theorem}\label{thm-stationarity}
    Suppose $(\gamma, \eta)$ lies in~\eqref{eq-gamma-eta},  let $(\phi, (K_\cdot)_{[0, \qduration]})$ be a QLE$(\gamma^2, \eta)$ process, and let $(L_\cdot)_{[0, \qduration]}$ be its associated boundary length process. 
    For $s>0$, conditioned on $\{s < \qduration\}$ and on $(L_\cdot)_{[0,s]}$, the conditional law of $\phi_s$ is $P_{\alpha, \beta, L_s}$, where $(\alpha, \beta)$ are as in~\eqref{eq-alpha-beta}.
\end{theorem}

A similar stationarity property holds for capacity time QLE$(\gamma^2, \eta)$: at each fixed capacity time $t$, their field at time $t$ agrees in law with their initial field as \emph{distributions modulo additive constant} (an equivalence class of distributions where $\phi_1 \sim \phi_2$ if $\phi_1 = \phi_2+c$ for some constant $c$). Crucially, $\gamma$-LQG areas or lengths do not make sense for distributions modulo additive constant since $\cA^\gamma_{\phi + c} = e^{\gamma c} \cA^\gamma_\phi$ and $\cL^\gamma_{\phi + c} = e^{\gamma c/2} \cL^\gamma_\phi$. Our proof of Theorem~\ref{thm-phase} requires us to keep track of such quantities, and hence does not transfer to the setting of capacity time QLE.

Finally, we identify the laws of $\gamma$-LQG surfaces that are cut out or swallowed by QLE$(\gamma^2, \eta)$ in the dilute and swallowing phases; in the space-filling phase QLE$(\gamma^2, \eta)$ fills the entire disk hence there is no analogous statement. For each $\gamma \in (0,2)$ and $\ell > 0$, there exists a probability measure $\QD(\ell)^\#$ describing the law of the canonical $\gamma$-LQG quantum disk with boundary length $\ell$; see Definition~\ref{def-QD} for details.

\begin{theorem}\label{thm-simple-markov}
    In the dilute phase where $\gamma \in (\sqrt3 - 1,2)$ and $\eta = \frac3{\gamma^2} - \frac12$, conditioned on the boundary length process $(L_\cdot)_{[0,\qduration]}$, the conditional law of the quantum surface $(\D \backslash K_\qduration, \phi)/{\sim_\gamma}$ is $\QD(L_\qduration)^\#$.
\end{theorem}

Let $K_{s-} := \bigcup_{u < s} K_u$. 
We say that a domain $U \subset \D$ is \emph{swallowed} by $(K_\cdot)_{[0,\qduration]}$ if there is a time $s< \qduration$ such that $U$ is one of the connected components of the interior of $K_s \backslash \bigcup_{u < s} K_u$. In the swallowing phase, each jump of $s \mapsto \cA_\phi(K_s)$ represents a domain $U$ being swallowed at time $s$. 
We say a quantum surface is swallowed if it is of the form $(\phi, U)/{\sim_\gamma}$ for some domain $U$ swallowed by $(K_\cdot)_{[0,\qduration]}$. In the swallowing phase, we identify the quantum surfaces swallowed by QLE as quantum disks. 

\begin{theorem}\label{thm-swallowing-markov}
    In the swallowing phase where $\gamma \in (2\sqrt3-2,2)$ and $\eta = \frac{3\gamma^2}{16} - \frac12$, the jumps of $(L_\cdot)_{[0,\qduration]}$ are in bijection with the quantum surfaces swallowed by $(\phi, (K_\cdot)_{[0,\qduration]})$, wherein at each jump time the process $(\phi, (K_\cdot)_{[0,\qduration]})$ swallows a quantum surface whose $\gamma$-LQG boundary length equals the jump size. Moreover, conditioned on $(L_{\cdot})_{[0,\qduration]}$, the swallowed quantum surfaces are conditionally independent, and the conditional law of the quantum surface corresponding to a jump of size $\ell$ is $\QD(\ell)^\#$. 
\end{theorem}

Analogs of the Markov properties in this section are known for certain growth models on random planar maps; see Section~\ref{subsec-intro-discrete} for further discussion.

\subsection{Quantum zippers}\label{subsec-intro-continuum}
The key result underpinning the construction of capacity time QLE in \cite{ms-qle} is a radial variant of the \emph{capacity quantum zipper} of \cite[Corollary 1.5]{shef-zipper}, which goes as follows. Let $\gamma \in (0,2]$ and $\kappa \in \{ \gamma^2, 16/\gamma^2\}$. Let $\phi$ be a certain variant of the free boundary GFF on $\D$ and let $\eta$ be an independent radial $\SLE_\kappa$ curve from $1$ to $0$. Assume $\eta$ is  parametrized according to capacity: if $g_t$ is the conformal map sending the origin-containing connected component of $\D \backslash \eta([0,t])$ to $\D$ such that $g_t(0) = 0$ and $g_t(\eta(t)) = 1$, then $|g_t'(0)| = e^{t}$. The stationarity property of the capacity quantum zipper is that, defining $\phi_t = g_t \bullet_\gamma \phi$, we have $\phi_t \stackrel d= \phi$ when viewed as distributions modulo additive constant.

In order to construct quantum natural time QLE, we instead need a radial variant of the \emph{length quantum zipper} of \cite[Theorem 1.8]{shef-zipper}, in which one cuts along the SLE curve for $t$ units of quantum natural time instead of capacity time, and obtains a stationarity property for the field $\phi_t$ where, conditioned on the $\gamma$-LQG boundary length $L = \cL^{\gamma}_{\phi_t}(\partial \D)$, the conditional law of $\phi_t$ is explicit and does not depend on $t$. We discuss these results in Section~\ref{sec-zippers}. The zipper for $\kappa > 8$ was shown in \cite{ay-reversibility}, and we prove the zippers for $\kappa$ in $(0,4)$ and $(4,8)$ using symmetries of special $\gamma$-LQG surfaces called \emph{quantum disks} and \emph{generalized quantum disks} together with inputs from \cite{ay-radial} which identify the conformal weldings of such objects to themselves. We note that the arguments are substantially different from those of \cite[Theorem 1.8]{shef-zipper}. Indeed, \cite{shef-zipper} ``zooms in'' on a boundary point of a $\gamma$-LQG surface to produce an infinite-volume $\gamma$-LQG surface;  the $\gamma$-LQG surfaces we study have finite volume and so cannot be obtained by this kind of limiting argument. 

The stationarity property of the capacity quantum zipper only identifies the law of the field up to additive constant. By contrast, our quantum natural time quantum zipper views the field as a distribution (not just as a distribution modulo additive constant), so we can study the evolution of geometric notions such as $\gamma$-LQG area or length that depend on the additive constant of the field. This enables us to establish the phases of QLE in Theorem~\ref{thm-phase}.

\subsection{Comparison with discrete models}\label{subsec-intro-discrete}
A main novelty of the present work in comparison to \cite{ms-qle} is the derivation of continuum Markov properties which should correspond to scaling limits of Markov properties of the discrete models. In this section, the random planar maps we consider are  Boltzmann-weighted.

First, when a uniform-spanning-tree (UST) weighted random planar map has a DLA ($1$-DBM) process on it grown until it reaches its target point, cutting the edges comprising the DLA gives a random planar map with the disk topology such that, conditioned on its boundary length, it has the law of a UST-weighted random planar map \cite[Proposition 2.3 (iii)]{ms-qle}. Since the scaling limit of UST-weighted random planar maps with the sphere and disk topologies is the quantum sphere and quantum disk with the LQG parameter $\gamma = \sqrt2$ \cite{mullin-maps, shef-burger, DMS14}, the scaling limit of DLA on UST-weighted planar maps should yield a quantum sphere with a QLE$(2,1)$ process, such that the complement of the DLA process should be a quantum disk given its boundary length. 
This continuum Markov property is precisely what we show in Theorem~\ref{thm-simple-markov} for QLE$(2,1)$, and more generally the analogous statement holds for all parameters in the dilute phase.

Second, when a uniform triangulation has an Eden model ($0$-DBM) grown on it for finitely many steps, with positive probability some regions are swallowed (separated from the target point) by the growth process; the joint law of these regions conditioned on their boundary lengths is that of conditionally independent uniform triangulations with the disk topology with specified boundary lengths. Results of this type are now standard for uniform random planar maps explored by peeling processes \cite{as-uipt, curien-legall-peeling}, but see
\cite[Proposition 2.2 (iii) and Figure 2.2]{ms-qle} for a justification for this particular setup. Arguing similarly as before, in the scaling limit this Markov property should correspond to Theorem~\ref{thm-swallowing-markov} for QLE$(8/3,0)$. We note that for $(\gamma^2, \eta) = (8/3,0)$ Theorem~\ref{thm-swallowing-markov}  was already shown in \cite{ms-equivalence}, but it is novel for other parameters.

In both settings of DLA on UST-weighted random planar maps and the Eden model on uniform triangulations, as explained in \cite[Proposition 2.3 (ii) and Proposition 2.2 (i)]{ms-qle}, as the growth process evolves, the unexplored region containing the target point has a stationarity property analogous to that of Theorem~\ref{thm-stationarity}.

In the rest of this section, we use the discrete models to motivate our construction of $\delta$-QLE via SLE, following the discussion of \cite[Section 2]{ms-qle}. 
For the values $\eta = 1$ and $\eta = 0$, the DBM model (in a compatible random environment) enjoys a simple relationship with a random curve. For example, the $0$-DBM process defined on a variant of random triangulation of the sphere can be viewed as a ``reshuffling'' of the so-called \textit{percolation interface} curve. Both a curve and growth process on a graph define an increasing sequence of vertices (or dual vertices) $C_t$ such that $|C_{t + 1} \setminus C_t| = 1$. In the $0$-DBM model, one lets $C_{t + 1} \setminus C_t$ be a dual vertex sampled according to the uniform measure on $\partial C_t$. For the percolation interface, an initial percolation configuration is chosen on the map, and $C_{t + 1} \setminus C_t$ is chosen on $\partial C_t$ according to the interface dictated by that configuration. Remarkably, forgetting the location of the tip of $C_t$, the law of the unexplored region is identical for the $0$-DBM model and percolation interface. Stronger yet, the laws of the full processes of complementary connected components coincide, so the Markov properties are identical. 

A key insight of \cite{ms-qle} is that one can also define a growth process in the continuum by reshuffling a curve. Recall that SLE is a family of random fractal curves which arise as the scaling limit of discrete random curves arising in critical statistical mechanical systems; see \cite{lawler2008conformally} for an overview. LQG with independent SLE is believed (and in some cases proved) to be the scaling limit for various discrete random curves in discrete random environments. The quantum zipper (Section~\ref{subsec-intro-continuum}) gives the continuum analog of the stationarity property discussed earlier. Thus, the continuum analogue of an $\eta$-DBM model on a planar map should heuristically be an object defined by a random curve (SLE) on a random environment (LQG) through a \textit{continuum} reshuffling procedure, and this object should share the Markov properties of the associated curve. The reshuffling relationship between an $\eta$-DBM model and a random curve is currently limited to $\eta = 0, 1$; see \cite[Question 9.13]{ms-qle}.

\medskip 
    \noindent 
    \textbf{Outline.} In Section~\ref{sec-prelim} we review the relevant random curves and surfaces. In Section~\ref{sec-zippers} we state and prove several radial quantum natural time quantum zippers. In Section~\ref{sec-qle-approximates} we use the quantum zippers to define the $\delta$-QLE$(\gamma^2, \eta)$ processes. In Section~\ref{sec-qle-properties} we establish tightness of these processes, and study their subsequential limits  as $\delta \to 0$; in particular we establish stationarity in law of the  subsequential limits (Theorem~\ref{thm-stationarity}). In Section~\ref{sec-phases} we define QLE via such subsequential limits, and establish its phases (Theorem~\ref{thm-phase}) and some Markov properties (Theorems~\ref{thm-simple-markov} and~\ref{thm-swallowing-markov}). In Section~\ref{sec-open} we list some open problems. 
    In Appendix~\ref{sec-other-QLEs} we explain the equivalence of the growth process constructed in \cite{ms-equivalence} and our QLE$(8/3,0)$. 

\medskip 
    \noindent \textbf{Acknowledgments.} We thank Ewain Gwynne for helpful discussions. M.A. was partially supported by NSF grant DMS-2348201 and a start-up grant from the University of California San Diego. D.M. was partially supported by NSF grant DMS-2245832.
\section{Preliminaries}\label{sec-prelim}

\subsection{Notation}
We will typically use $s$ to represent quantum natural time, and $t$ for capacity (negative log conformal radius) time. Given a finite measure $M$, we denote the total mass by $|M|$, and let $M^{\#}$ denote the probability measure $M^\# = M/|M|$. Throughout this paper, we denote by $(f, g)$ the integral $\int_{\D}(f \cdot g) dx$ provided that the integral is finite, but regardless of whether or not $f, g \in L^2(\D)$.

\subsection{Schramm-Loewner evolution and measure-driven Loewner evolution}\label{subsec-sle}
Schramm-Loewner evolution is a family of random fractal curves that arise as scaling limits of interfaces in critical lattice models \cite{schramm0, smirnov-cardy, lsw-lerw-ust, smirnov-ising}. It is indexed by a parameter $\kappa >0$ that describes the roughness of the curve. We give a brief introduction; see \cite{lawler2008conformally} for more details.  

In this paper we will be exclusively concerned with the radial variant of SLE$_\kappa$. This is a random non-self-crossing curve $\eta:[0,\infty) \to \clD$ with $\eta(0) = 1$ and $\lim_{t\to\infty} \eta(t) = 0$. Let $K_t$ be the compact subset of $\clD$ such that $\clD \backslash K_t$ is the connected component of $\clD \backslash \eta([0,t])$ which contains the origin, and let $g_t: \D \backslash K_t \to \D$ be the conformal map with $g_t(0) =0$ and $g_t'(0) > 0$. The curve $\eta$ is parametrized by capacity, meaning that $g_t'(0) = e^{t}$ for all $t \geq 0$. 

Let
\[\Phi(u, z) = z\frac{u + z}{u - z}.\]
By the construction of radial SLE$_\kappa$, the family of conformal maps $(g_t)_{t \geq0}$ solves the  radial Loewner equation
\[
    \dot g_t(z) = \Phi(U_t, g_t(z)) \quad \text{ for all }t > 0 \text{ and }z \in \D \backslash K_t,
    \]
where $\dot g_t(z)$ indicates the time-derivative at time $t$ for fixed $z$, and the random function $U_t$ is such that $U_t = \exp(i\sqrt{\kappa}B_t)$ with $B_t$ a standard Brownian motion. 
The function $U_t$ taking values in $\partial \D$ is called the {driving function}. It is the tip of the curve under the map $g_t$, and describes the point from which the process will grow. 
For more details on radial SLE and other variants, see \cite{lawler2008conformally}.

The Loewner equation describes growth processes which, roughly speaking, grow from a single point prescribed by the driving function. We next discuss how to generalize this to the setting where the growth process may grow from more than one point at a time. The growth process will consist of a sequence of (random) sets $(K_t)_{t \geq 0}$ such that $K_t$ is a compact subset of $\ol \D$ and $\D \setminus K_t$ is simply connected, for all $t \geq 0$. Such sets are called \textit{compact hulls}, but we refer to them interchangeably as \textit{hulls} in this paper.
For $T>0$, let $\cN_T$ be the collection of measures $\nu$ on $[0,T] \times \partial \D$ with Lebesgue marginal in the first coordinate, and let $d_T$ be the L\'evy-Prokhorov metric. Let $\mathcal{N}$ be the collection of measures on $[0, \infty) \times \partial \D$ having Lebesgue marginal in the first coordinate, and equip $\mathcal N$ with the metric $d(\nu, \nu') = \sum_{T \in \mathbb N} 2^{-T} d_T(\nu|_{[0,T] \times \partial \D}, \nu'|_{[0,T] \times \partial \D})$. Then $(\cN, d)$ is a compact metric space, and convergence in $d$ is equivalent to convergence in the L\'evy-Prokhorov metric on $[0,T] \times \partial \D$ for every $T>0$.

Given a driving measure $\nu \in \cN$, the \textit{Loewner-Kufarev evolution} (or \emph{measure-driven Loewner evolution}) reads \eqb \label{eq-loewner-kufarev}    \dot{g_t}(z) = \int_{\partial \D}\Phi(u, g_t(z))d\nu_t(u)\eqe
where $(\nu_t)_{t \geq 0}$ is the family of probability measures obtained by disintegrating $\nu$ according to the first coordinate. 
Its solution $(g_t)_{t \geq 0}$ is a family of conformal maps $\D \setminus K_t \to \D$ where $K_t$ is a compact hull,  $g_t(0) = 0$ and $g_t'(0) = e^{t}$.

Proposition~\ref{prop-loewner-bijection} below states that one can pass bijectively between families of growing hulls and conformal maps, so the driving measure is an appropriate encoding for our growth processes. A version of this result where $\cN$ is replaced by $\cN_T$ is given by \cite[Theorem 1.1]{ms-qle}, and the version for $\cN$ is \cite[Propositions 6.1--6.3]{ms-qle}.

\begin{proposition} %[Hull growth encoding bijection]
\label{prop-loewner-bijection}
There is a bijection between elements $\nu \in \mathcal N$ and families of hulls $(K_t)_{t \geq 0}$ parametrized by negative log-conformal radius. Additionally, the measures $(\nu^n)_{n \in \N}$ converge weakly to $\nu$ if and only if for each time $t$, the functions $((g_t^n)^{-1})_{n \in \mathbb{N}}$ converge locally uniformly to $(g_t)^{-1}$,  where $(g_t^n)_{t \geq 0}$ is the family of mapping-out functions encoded by $\nu^n$ for each $n$, and similarly $(g_t)_{t \geq 0}$ corresponds to $\nu$.
\end{proposition}

We say that a family of compact hulls $(K_n)_{n \in \N}$ converge in the \textit{Carath\'eodory topology} to $K$ if $(g^n)^{-1}$ converges locally uniformly to $g^{-1}$, where $g^n: \D \backslash K_n \to \D$ is the conformal map such that $g^n(0) = 0$ and $(g^n)'(0)>0$, and $g$ is the corresponding map for $K$.
Proposition~\ref{prop-loewner-bijection} therefore equates convergence in $\cN$ with Carath\'eodory convergence of the sequence of hulls at each time $t$.

 The following stronger variant of Proposition~\ref{prop-loewner-bijection} was proven in \cite[Proposition 6.1]{ms-qle}\footnote{The space $\mathcal G$ in that proposition statement is defined in their Section 6.1.}.
 
\begin{proposition}\label{prop-conv-loc-spacetime}
Consider a convergent sequence of driving measures $(\nu^n)_{n \in \N}$ with limit $\nu$, and let $(g^n_t)_{t \geq 0}$ and $(g_t)_{t \geq 0}$ be the corresponding mapping out functions. Define $(g^n_\cdot)^{-1}: \D \times [0,\infty) \to \R$ via $(g^n_\cdot)^{-1}(z, t) := (g^n_t)^{-1}(z)$, and likewise define  $(g_\cdot)^{-1}$. As $n \to \infty$, we have $(g^n_\cdot)^{-1}\to (g_\cdot)^{-1}$ in the topology of uniform convergence on compact subsets of $\D \times [0,\infty)$. 
\end{proposition}

Finally, we will need the following elementary relationship between Carath\'eodory and Hausdorff convergence of compact hulls. The result is the radial analogue of Lemma 2.3 in \cite{peltola2023large}, and its proof is identical.

\begin{lemma}\label{lem-hausdorff-cara-conv}
Let $(K_n)_{n \in \N}$ be a sequence of compact hulls such that $K_n \to C$ in the Carath\'eodory topology and $K_n \to H$ in the Hausdorff topology. Then $\D \setminus C$ is the connected component of $\D \setminus H$ containing $0$.
\end{lemma}

\subsection{The Gaussian free field and Liouville quantum gravity}\label{subsec-lqg-gff}
We now introduce the Gaussian free field (GFF) and Liouville quantum gravity; see \cite{berestycki-lqg-notes} for a more detailed introduction.

Letting $m_{\D}$ be the uniform measure on $\partial \D$ and $\lambda_{\D}$ the Lebesgue measure restricted to $\D$, we define the \textit{Dirichlet inner product} $(f, g )_{\nabla} := \int_{\D}\nabla f \cdot \nabla g \, \lambda_{\D}$ on the function space $\{f \in C^{\infty}(\D): (f, f)_{\nabla} < \infty,\, \int_{\D}f\, m_{\D} = 0\}$. Let $H(\D)$ denote the Hilbert space completion of this function space with respect to the Dirichlet inner product. Fix an orthonormal basis $(f_n)_{n \in \N}$ of $H(\D)$ and let $(\alpha_n)_{n \in \N}$ be independent standard Gaussians. The random Fourier series
\begin{equation}\label{eq-gff}
    h := \sum_{n = 1}^{\infty}\alpha_n f_n
\end{equation} converges a.s.\ in the space of distributions on $\D$ \cite{shef-gff}. We call $h$ the free boundary GFF normalized such that $\int_{\D}h \, m_{\D} = 0$, and denote its law by $P_\D$. The GFF $h$ is the centered Gaussian field on $\D$ with covariance structure $\E[h(z)h(w)] = G_\D(z,w)$, where 
\[G_\D(z,w):= -\log|z-w| - \log|1-z\ol w| \qquad \text{ for }z,w \in \D.\]
Note that since $h$ is too rough to admit pointwise values, the expression $\E[h(z)h(w)]$ does not make literal sense, and should be interpreted via  
\begin{equation*}
    \E[(h, f_1) (h, f_2)] = \int_\D \int_\D f_1(z_1) G_\D(z_1, z_2) f_2(z_2) \, \lambda_\D(dz_1) \, \lambda_\D(dz_2) \quad \text{ for smooth compactly supported }f_1, f_2.
\end{equation*}

Next, given a conformal map $g: D \to \wt D$ and a distribution $h$ on $D$, with $\gamma \in (0, 2)$ and $Q = 2/\gamma + \gamma/2$ we define the \textit{$\gamma$-LQG coordinate change} of $h$ as
\begin{equation}\label{eq-coord-change}
    g\bullet_\gamma h:=h\circ g^{-1}+Q\log|(g^{-1})'|
\end{equation}
We then define an equivalence relation $\sim_{\gamma}$ between domain-field pairs $(D, h)$ in which $(D, h) \sim_{\gamma} (\wt D, \wt h)$ if there exists a conformal map $g: D \to \wt D$ such that $\wt h = g \bullet_{\gamma} h$. Such an equivalence class is called a \textit{$\gamma$-LQG surface} or \textit{quantum surface}, and a choice of representative $(D,h)$ is called an \textit{embedding} of the $\gamma$-LQG surface \cite{shef-kpz}.

We now introduce the Gaussian multiplicative chaos (GMC) measures \cite{kahane} associated to the GFF; see for instance \cite{shef-kpz, berestycki-gmt-elementary} for more details.
The \emph{$\gamma$-LQG  area measure} $\cA_h^{\gamma}$ is the measure on $\D$ defined as the a.s.\ weak limit of $\epsilon^{\gamma^2/2}e^{\gamma h_{\epsilon}}(z)dz$ as $\epsilon \to 0$ where $h_{\epsilon}(z)$ is the average of $h$ on $\partial B(z, \epsilon) \cap \D$ and $dz$ is Lebesgue measure on $\D$. The \emph{$\gamma$-LQG boundary length measure} $\cL_h^{\gamma}$ is the a.s.\ weak limit of the measure $ \epsilon^{\gamma^2/4}e^{\gamma h_{\epsilon}/2}(x)dx$ on $\partial \D$ as $\epsilon \to 0$ where $h_{\epsilon}(x)$ is the average of $h$ on $\partial B(x, \epsilon) \cap \D$ and $dx$ is the length measure on $\partial \D$. Note that these measures also make sense if $h$ is replaced by $h + f$ where $f$ is a bounded continuous function, and satisfy $\cA^\gamma_{h+f} = e^{\gamma f} \cA^\gamma_h$ and $\cL^\gamma_{h+f} = e^{\gamma f/2} \cL^\gamma_h$. 
If $g$ is a conformal automorphism of $\D$, then $g_* \cA^\gamma_h = \cA^\gamma_{g \bullet_\gamma h}$ and $g_* \cL^\gamma_h = \cL^\gamma_{g \bullet_\gamma h}$ \cite{shef-kpz}. In this sense, the $\gamma$-LQG area and boundary length measures are intrinsic to $\gamma$-LQG surfaces (do not depend on the choice of embedding), and in particular this allows us to define $\gamma$-LQG boundary lengths for all simply-connected domains by conformally mapping to $\D$. Note that for $\beta \neq \gamma$ the GMC measure $\cL^\beta_h$ is not intrinsic to the $\gamma$-LQG surface; we will always embed in $\D$ when working with $\cL^\beta_h$.

\subsection{The Liouville field and quantum disk}\label{subsec-liouville-field}
Recall that $P_\D$ is the law of the free boundary GFF on $\D$ normalized to have mean zero on $\partial \D$. 

\begin{definition}[Liouville field on $\mathbb{D}$]\label{def-lf}
Let $\alpha,\beta\in\bbR$ and $s\in\partial\bbD$. 
Let $(h,\mathbf{c})$ be sampled from $P_\bbD\times [e^{(\alpha+\frac{\beta}{2}-Q)c}dc] $ and set 
\[\phi(z) = h(z)+\alpha G_\bbD(z,0)+\frac{\beta}{2}G_\bbD(z,s)+\mathbf c = h(z) - \alpha \log|z| - \beta \log |z - s| + \mathbf c.\] We call $\phi$ the Liouville field on $\bbD$ with insertions $(\alpha,0),(\beta,s)$ and denote the law of $\phi$ by  $\LF_\bbD^{(\alpha, 0), (\beta, s)}$. 
\end{definition}
If $\beta = 0$, we will simply write $\LF_\D^{(\alpha, 0)}$.
We next define a fixed boundary length version of the Liouville field in Definition~\ref{def-lf}.

\begin{definition}[Liouville field on $\mathbb{D}$ with fixed boundary length]\label{def-lf-conditioned}
Let $\alpha \in \R$, $\beta < Q$, and $s \in \partial \D$. Sample $h$ from $P_\D$ and set
\begin{equation*}
    \tilde h(z) = h(z) + \alpha G_\bbD(z,0)+\frac{\beta}{2}G_\bbD(z,s)
\end{equation*}
Fix $\ell > 0$ and let $L = \cL^\gamma_{\tilde h}(\partial \D)$, and define the law $\LF_{\D, \ell}^{(\alpha, 0), (\beta, s)}$ to be that of $\tilde h + \frac{2}{\gamma}\log \frac{\ell}{L}$ under the reweighted law $\frac{2}{\gamma} \ell^{\frac1\gamma(2\alpha + \beta - 2Q) - 1} L^{- \frac1\gamma(2\alpha + \beta - 2Q)} \P_{\D}(dh)$.
\end{definition}

The Liouville field possesses the following explicit disintegration \cite[Lemma 2.7]{ay-reversibility}.

\begin{lemma}[Liouville field disintegration]\label{lem-disint-explicit}
    Given $\alpha \in \R$ and $\beta < Q$, the collection $(\LF_{\D, \ell}^{(\alpha, 0), (\beta, 1)})_{\ell > 0}$ is a disintegration of $\LF_{\D}^{(\alpha, 0), (\beta, 1)}$. In particular, a sample $\phi \sim \LF_{\D, \ell}^{(\alpha, 0), (\beta, 1)}$ satisfies $\cL_{\phi}^{\gamma}(\partial \D) = \ell$ almost surely, and the laws satisfy
    \begin{equation*}\label{eq-disint}
        \LF_{\D}^{(\alpha, 0), (\beta, 1)} = \int_{0}^{\infty}\LF_{\D, \ell}^{(\alpha, 0), (\beta, 1)}d\ell
    \end{equation*}
\end{lemma}

In the following sense, sampling a point from a boundary GMC measure corresponds to adding a log singularity to the Liouville field. Let $\mathrm{Leb}_{\partial \D}$ denote the length measure on $\partial \D$, so $|\mathrm{Leb}_{\partial \D}| = 2\pi$.
\begin{lemma}\label{lem-sample-gmc}
    Let $\alpha \in \R$ and $\beta \in (-2,2)$, then 
    \[ \cL_\phi^\beta(du)\, \LF_\D^{(\alpha,0)}(d\phi) = \LF_\D^{(\alpha, 0),(\beta,u)} (d\phi) \, \mathrm{Leb}_{\partial \D} (du).\]
\end{lemma}
\begin{proof}
    The proof of Lemma~\ref{lem-sample-gmc} is essentially identical to that of \cite[Lemma 2.31]{ahs-integrability}, so we omit it.
\end{proof}

We now introduce a variant of the GFF. 

\begin{definition}\label{def-P-weighted}
Let $\alpha \in \R, \beta \in (-2,2)$ and $\ell > 0$. 
    Let $P_{\alpha, \beta, \ell}$ be the probability measure defined by
    \[\frac{dP_{\alpha, \beta, \ell}}{d \LF^{(\alpha,0)}_{\D, \ell}} (\phi) = Z_{\alpha,\beta,\ell}^{-1} \cL^\beta_\phi(\partial \D), \qquad Z_{\alpha,\beta,\ell} := \LF^{(\alpha,0)}_{\D, \ell}[\cL^\beta_\phi(\partial \D)].\]
    In other words, $P_{\alpha, \beta, \ell}$ is obtained from $\LF^{(\alpha, 0)}_{\D, \ell}$ by weighting by the boundary $\beta/2$-GMC measure and    normalizing to be a probability measure. 
\end{definition}

\begin{lemma}\label{lem-sample-gmc-2}
    We have 
    \[(\cL^\beta_\phi)^\#(du) \,  P_{\alpha, \beta, \ell}(d\phi) = (\LF_{\D, \ell}^{(\alpha, 0), (\beta, u)})^\#(d\phi) \, \mathrm{Leb}_{\partial \D}^\#(du). \]
\end{lemma}
\begin{proof}
    This is immediate from Definition~\ref{def-P-weighted} and Lemma~\ref{lem-sample-gmc}. 
\end{proof}

Finally, let $\gamma \in (0,2)$. The \emph{$\gamma$-LQG disk} or  \emph{quantum disk} is a $\gamma$-LQG surface with the disk topology first introduced in \cite{DMS14}. It is the conjectural scaling limit of random planar maps with the disk topology in the $\gamma$-LQG universality class. As shown in \cite{cercle-disk} and subsequently \cite{ahs-integrability, ars-fzz}, this object can be equivalently defined via  Liouville CFT on the disk. We will not need the explicit definition of this object, and only provide it for concreteness. To simplify the presentation we use an equivalent definition that follows from~\cite[Theorem 3.4]{ars-fzz}. 

\begin{definition}\label{def-QD}
    Let $\ell > 0$. Let $\QD(\ell)$ be the law of $(\D, \phi)/{\sim_\gamma}$ where the field is sampled from the weighted measure $\frac{\gamma}{2\pi(Q-\gamma)^2}\cA_\phi(\D)^{-1}\LF_{\D, \ell}^{(\gamma, 0)}(d\phi)$. A sample from $\QD(\ell)$ is called a $\gamma$-LQG disk with boundary length $\ell$.
\end{definition}
Note that $\QD(\ell)$ is a finite measure for all $\ell> 0$, but for generic $\ell$ it is a non-probability measure.

\subsection{The quantum natural time parametrization of SLE}\label{subsec-quantum-parametrization-SLE}
In Section~\ref{subsec-sle} we encountered the capacity (log conformal radius) parametrization. In this section we introduce the \emph{quantum natural time} parametrization of SLE$_\kappa$, which describes the length of an SLE$_\kappa$ curve with respect to an underlying $\gamma$-LQG geometry. We will always couple the parameters via $\gamma = \min(\sqrt\kappa, 4/\sqrt\kappa)$.

Recall that SLE$_\kappa$ has three phases: when $\kappa \in (0,4]$ the curve is simple, in the swallowing phase $\kappa \in (4,8)$ the curve bounces off of itself and the domain boundary but has Lebesgue measure zero, and when $\kappa \geq 8$ the curve is space-filling. We treat each of these cases separately, starting with the simple phase. Recall the free boundary GFF $h$ on $\D$ normalized to have average zero on $\partial \D$, as defined in Section~\ref{subsec-lqg-gff}.

\begin{definition}\label{def-quantum-time-simple}
    Let $\kappa < 4$ and $\gamma = \sqrt \kappa$. Let $h$ be a free boundary GFF on $\D$ and let $\eta: [0,\infty) \to \clD$ be an independent radial $\SLE_\kappa$ curve. For any $0 < t_1 < t_2$,
    the curve segment $\eta([t_1, t_2])$ corresponds to two boundary arcs of the simply-connected domain $\D \backslash \eta([0,\infty))$. 
    The \emph{quantum natural time} of $\eta$ elapsed between $t_1$ and $t_2$ is the $\gamma$-LQG boundary length in $(\D \backslash \eta([0,\infty)), h)/{\sim_\gamma}$ of either of these boundary arcs.
\end{definition}
Implicit in the above definition is the fact that the $\gamma$-LQG boundary lengths in $\D \backslash \eta([0,\infty))$ of the two boundary arcs corresponding to $\eta([t_1, t_2])$ exist and agree a.s.; this was proved in \cite{shef-zipper} for a specific variant of the GFF $h$ and for chordal SLE, and it holds in our setting by local absolute continuity of the field and curve. A similar result holds for $\kappa = 4$ and $\gamma = 2$ \cite{hp-welding}, but we do not treat this case in our paper. 
See also \cite{ps-quantum-length} for an elementary and unified proof for all $\kappa \in (0,4]$. 

Next, we address the swallowing phase. Consider a radial SLE$_\kappa$ curve with the capacity time-parametrization. Let $U$ be one of the connected components of $\D \backslash \eta([0,\infty))$. We say $U$ is \emph{swallowed} in the time interval $(t_1, t_2]$ if $\partial U \not \subset \eta([0,t_1])$ but $\partial U \subset \eta([0,t_2])$. 
\begin{definition}\label{def-quantum-time-swallowing}
    Let $\kappa \in (4,8)$ and $\gamma = 4/\sqrt \kappa$. Let $h$ be a free boundary GFF on $\D$ and let $\eta: [0,\infty) \to \clD$ be an independent radial $\SLE_\kappa$ curve. There is a constant $\mathfrak c_0$ such that, for any $0 < t_1 < t_2$,
     the \emph{quantum natural time} of $\eta$ elapsed between $t_1$ and $t_2$ is $\mathfrak c_0 \lim_{\eps \to 0}  \eps^{4/\gamma^2} N_\eps(t_1, t_2)$ where $N_\eps(t_1, t_2)$ is the number of regions $U$ swallowed in the time interval $(t_1, t_2]$ having $\gamma$-LQG boundary length between $\eps$ and $2\eps$. 
\end{definition}
This definition via $N_\eps(t_1, t_2)$ is briefly mentioned after \cite[Proposition 1.4.5]{DMS14}, but is not the standard definition. In Remark~\ref{remark-quantum-time-swallowing} right below we explain the equivalence.
\begin{remark}\label{remark-quantum-time-swallowing}
    \cite[Theorem 6.5.7]{DMS14} shows that for a specific variant of the GFF $h$ and for chordal SLE, the countable collection of lengths of swallowed regions, ordered by the time each region is swallowed, agrees in law with the ordered collection of jumps of a totally asymmetric stable L\'evy process with index $\frac4{\gamma^2}$. Note that the totally asymmetric stable L\'evy processes with index $\frac4{\gamma^2}$ form a one-parameter family, which differ only by linear reparametrization of time. Fixing a choice of the L\'evy process,
    \cite[Definition 6.5.8]{DMS14} defines the quantum natural time to be the time-parametrization of this L\'evy process. The equivalence with Definition~\ref{def-quantum-time-swallowing} via $N_\eps$ is immediate from the fact that the L\'evy process's  collection of (jump, time) pairs is Poissonian with intensity $[1_{x > 0} c x^{-4/\gamma^2-1} \, dx] \times [1_{t > 0} \, dt]$, where the constant $c>0$ depends on the choice of L\'evy process. Our constant $\mathfrak c_0$ is an explicit function of $c$. 
\end{remark}

Finally, we address the space-filling phase. 

\begin{definition}\label{def-quantum-time-sf}
    Let $\kappa >8$ and $\gamma = 4/\sqrt \kappa$. Let $h$ be a free boundary GFF on $\D$ and let $\eta: [0,\infty) \to \clD$ be an independent radial $\SLE_\kappa$ curve. For any $0 < t_1 < t_2$,
     the \emph{quantum natural time} of $\eta$ elapsed between $t_1$ and $t_2$ is $\cA^\gamma_h(\eta([t_1, t_2]))$.
\end{definition}

Note that by local absolute continuity, Definitions~\ref{def-quantum-time-simple}-\ref{def-quantum-time-sf} also make sense when the GFF $h$ is replaced by a sample from $\LF^{(\alpha,0), (\beta,1)}_{\D, \ell}$ or $P_{\alpha, \beta, \ell}$ for any parameters $\alpha, \beta, \ell$.

\subsection{Topologies and metric spaces}\label{subsec-topologies} 
In our argument, we will take subsequential limits of random variables taking values in the following metric spaces. 
\begin{enumerate}
    \item ($\mathcal{D}$) \textbf{A function space on $\mathbb{D}$:} 
    Let $a>0$, and consider the Hilbert space $(-\Delta)^a L^2(\D)$ which is the completion of $C_0^{\infty}(\D)$ with respect to the inner product $(f, g)_{a} :=  ((-\Delta)^{-a}f, (-\Delta)^{-a}g)$ where $(\cdot, \cdot)$ is the $L^2$ inner product and $(-\Delta)^{a} f := \sum_{n} (f, f_n)  (-\lambda_n)^a f_n$ where $\lambda_n$ and $f_n$ are eigenvalues and eigenvectors of the Laplacian on $\D$ with Dirichlet boundary conditions.  The Gaussian free field with Dirichlet boundary conditions is an element of $(-\Delta)^a L^2(\D)$ for any $a>0$; see \cite{shef-gff} for details.

    In this paper we consider the free boundary GFF, which we denote here by $\wt h$. By the domain Markov property, $\wt h$ on $\D$ restricted to $(1 - \epsilon)\D$ can be written as $\wt h|_{(1 - \eps)\D} = h + \mathfrak{h}$ where $h$ is a Dirichlet GFF on $(1 - \epsilon)\D$ and $\mathfrak{h}$ is a function harmonic on $(1 - \eps)\D$. The distribution $\wt h|_{(1 - \eps)\D}$ is therefore an element of $(-\Delta)^a L^2((1 - \epsilon)\D)$ which is defined in analogy to the preceding paragraph, and which we denote by $\mathcal{D}_a^{\epsilon}$. The product space $\prod_{n \in \N}\mathcal{D}_a^{1/n}$ equipped with the metric $d_a(\cdot, \cdot) := \sum_{n = 1}^{\infty}2^{-n}\min(d_{a, 1/n}(\cdot, \cdot), 1)$ is separable and complete, and the subset $\mathcal{D}_a := \{ (h_n)_{n \in \N} \in \prod_{n \in \N}\mathcal{D}_a^{1/n}: \, h_m|_{(1 - 1/n)\D} = h_n ,\, \forall m > n \}$ is closed, and is therefore separable and complete in the subspace metric.
    
    The free boundary GFF takes values in $\mathcal{D}_a$ almost surely. The constant $a > 0$ is inconsequential for our arguments, and so we fix some arbitrary $a > 0$ and write $\mathcal{D}_a$ as simply $\mathcal{D}$.

    \item ($\mathcal{N}$) \textbf{Infinite-time driving measures on $\partial \D$:} Recall that for each finite time $T >0$, we let $\mathcal{N}_T$ be the space of measures on $[0, T] \times \partial \D$ with Lebesgue marginal in the first coordinate, equipped with the L\'evy-Prokhorov metric. Each $\mathcal{N}_T$ thus defined is a compact metric space with metric denoted by $d_T$. Let $\mathcal{N}$ be the collection of measures on $[0, \infty) \times \partial \D$ having Lebesgue marginal in the first coordinate, and equip $\mathcal N$ with the metric $d(\nu, \nu') = \sum_{T \in \mathbb N} 2^{-T} d_T(\nu|_{[0,T] \times \partial \D}, \nu'|_{[0,T] \times \partial \D})$. Then $(\cN, d)$ is a compact metric space, and convergence in $d$ is equivalent to convergence in the L\'evy-Prokhorov metric on $[0,T] \times \partial \D$ for every $T>0$. 

    % https://encyclopediaofmath.org/wiki/Lévy-Prokhorov_metric
    \item ($\mathcal{B}$) \textbf{Finite Borel measures on $\overline{\D}$:} We equip the space of finite Borel measures on $\overline{\D}$ with the L\'evy-Prokhorov metric. This is a complete separable metric space, and we equip it with the Borel $\sigma$-algebra.

    \item ($\mathcal{J}$) \textbf{C\`adl\`ag functions from $[0,1]$ to $\R$:} Let $\mathcal J$ be the collection of functions from $[0,1]$ to $\R$ which are right-continuous and have left limits. For an increasing homeomorphism $\lambda$ from $[0,1]$ to itself, let $\|\lambda\|^\circ = \sup_{s < t} | \log \frac{\lambda(t) - \lambda(s)}{t-s}|$. When the space $\mathcal J$ is equipped with the  metric $d^\circ(f,g) = \inf_\lambda \max(\|\lambda\|^\circ, \| f - g\circ \lambda \|_\infty)$, 
    it is a complete separable metric space having the Skorokhod J$_1$ topology; see \cite[Section 12]{billingsley2013convergence} for details.

    \item ($\mathcal I$)  \textbf{Nondecreasing c\`adl\`ag functions from $[0,\infty) \to [0,\infty]$:} Let $(q_j)_{j \in \N}$ be an arbitrary ordering of the nonnegative rational numbers. We equip $\mathcal I$ with the metric $d(f,g) = \sum_{j=1}^\infty 2^{-j}\left| e^{-f(q_j)} - e^{-g(q_j)}\right|$.
\end{enumerate}

\section{Radial quantum natural time quantum zipper}\label{sec-zippers}
The key ingredients in our subsequent construction of QLE are stationarity results for the exploration of an LQG surface by an SLE curve for a fixed amount of quantum natural time. We state these as Propositions~\ref{prop-radial-mot-sf-finite}-\ref{prop-radial-mot-swallow-finite}, with one proposition for each phase of SLE. We also show that in the simple and swallowing phases, the regions cut out by the radial SLE curves are quantum disks (Lemmas~\ref{lem-simple-unexplored} and~\ref{lem-swallowed-sle}). Note that the results in this section cover all $\kappa \in (0,\infty) \backslash \{4, 8\}$ and $\gamma = \min (\sqrt\kappa, 4/\sqrt\kappa)$; the only missing pairs are $(\gamma, \kappa) = (2, 4)$ and $(\gamma, \kappa) = (\sqrt2, 8)$ (corresponding to the critical values of $\kappa$ at the boundaries of the $\SLE_\kappa$ phase transitions).

Recall the Liouville field defined in Section~\ref{subsec-liouville-field}, and let $\mathrm{raSLE}_\kappa$ denote the law of radial SLE$_\kappa$ in $\D$ from $1$ to $0$. The first result, for the space-filling regime, was essentially proved in \cite{ay-reversibility}. See Figure~\ref{fig-cut} (left). 

\begin{figure}[ht]
	\centering
\includegraphics[scale=0.43]{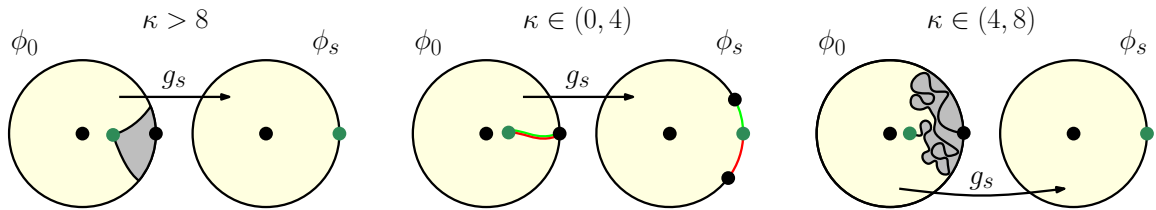}
	\caption{\textbf{Left:} Let $\kappa > 8$ and $\gamma = 4/\sqrt\kappa$. For a certain random field $\phi_0$ and independent space-filling radial $\SLE_\kappa$ $\eta$ (with time re-parametrized according to $\cA^\gamma_{\phi_0}$-area covered), for each $s< \cA^\gamma_{\phi_0}(\D)$ let $\phi_s$ be the field obtained from $\phi_0$ by mapping out the complement of $\eta([0,s])$. Proposition~\ref{prop-radial-mot-sf-finite} identifies the boundary length process $(\cL^\gamma_{\phi_s}(\partial \D))_{s \in [0, \cA^\gamma_{\phi_0}(\D))}$ as Brownian motion, and moreover gives a stationarity property for the conditional law of $\phi_s$ given the boundary length process up until time $s$. \textbf{Middle:} Let $\kappa \in (0,4)$ and $\gamma = \sqrt\kappa$. In this setting, the boundary length process is affine with gradient 2 (as the green and red boundary segments each have quantum length $s$), and the law of the field at time $s$ satisfies a stationarity property (Proposition~\ref{prop-radial-mot-simple-finite}).
    \textbf{Right:} Figure for $\kappa \in (4,8)$ and $\gamma = 4/\sqrt\kappa$  (Proposition~\ref{prop-radial-mot-swallow-finite}).
    } \label{fig-cut}
\end{figure}

\begin{proposition}\label{prop-radial-mot-sf-finite}
Suppose $\kappa > 8$ and $\gamma = 4/\sqrt\kappa \in (0, \sqrt2)$. Let $\ell > 0$. Sample $(\phi_0,\eta)$ from the measure 
    \[
    (\LF_{\bbD,\ell}^{(Q-\frac{\gamma}{4},0),(\frac{3\gamma}{2},1)})^\#\times \mathrm{raSLE}_\kappa
    \]
    and parametrize $\eta$ by its $\cA^\gamma_{\phi_0}$-quantum area. Let $\qduration = \cA^\gamma_{\phi_0}(\D)$. For $s \in [0, \qduration)$, let $g_s:\bbD\backslash \eta([0,s])\to\bbD$ be the conformal map such that $g_s(0) = 0$ and $g_s(\eta(s))=1$. Let $\phi_s = g_s \bullet_\gamma \phi_0$ and $L_s = \cL_{\phi_s}^\gamma(\partial \D)$. The process $(L_s)_{[0,\qduration]}$ has the law of Brownian motion with initial value $\ell$ and variance $4 \tan (2\pi/\kappa) s$ run until the time it hits 0; in particular $\qduration < \infty$ a.s. Moreover, for fixed $s$, conditioned on $\{s < \qduration\}$ and on  $(L_\cdot)_{[0,s]}$, the conditional law of $\phi_s$ is  $(\LF_{\bbD, L_s}^{(Q-\frac{\gamma}{4},0),(\frac{3\gamma}{2},1)})^\#$.
\end{proposition}
\begin{proof}
For the first claim, the law of $(L_s)_{s \geq 0}$ was shown to be that of the aforementioned Brownian motion in \cite[Theorem 3.1]{ay-reversibility}.  
    Next, for $s < \qduration$ define the curve $\eta_s := g_s \circ \eta(\cdot - s)$ in $\D$ from $1$ to $0$. We will show that for fixed $s>0$, the conditional law of $(\phi_s, \eta_s)$ conditioned on $\{s < \cA^\gamma_{\phi_0}(\D)\}$ and on  $(L_\cdot)_{[0,s]}$ is
    \eqb\label{eq-law-curve-field}
(\LF_{\bbD, L_s}^{(Q-\frac{\gamma}{4},0),(\frac{3\gamma}{2},1)})^\# \times \mathrm{raSLE}_\kappa.
    \eqe
    By \cite[Corollary 3.13]{ay-reversibility}, if we further condition on the event $E_1$ that $\eta([0,s])$ is simply connected (i.e., $\eta|_{[0,s]}$ has not wrapped around), then the conditional law of $(\phi_s, \eta_s)$ is~\eqref{eq-law-curve-field}. Iteratively applying this $n>1$ times on time intervals of length $s/n$ rather than $s$, if we instead further condition on the event $E_n$ that $\eta([is/n, (i+1)s/n])$ is simply connected for $i = 0, \dots, n-1$, the conditional law of $(\phi_s, \eta_s)$ is~\eqref{eq-law-curve-field}. Since $\eta$ is continuous we have $\lim_{n\to\infty}\P[E_n] = 1$. Thus, we conclude that the conditional law of $(\phi_s, \eta_s)$ conditioned on $\{s < \cA^\gamma_{\phi_0}(\D)\}$ and on  $(L_\cdot)_{[0,s]}$ is~\eqref{eq-law-curve-field}; this implies the second claim. 
\end{proof}

We will prove the following two variants of Proposition~\ref{prop-radial-mot-sf-finite}, see Figure~\ref{fig-cut} (middle, right). Their proofs are completely different from that of Proposition~\ref{prop-radial-mot-sf-finite}, and depend fundamentally on the radial conformal weldings of LQG developed in \cite{ay-radial}.

\begin{proposition}\label{prop-radial-mot-simple-finite}
Suppose $\kappa \in (0,4)$ and $\gamma = \sqrt\kappa \in (0,2)$. Let $\ell > 0$. Sample $(\phi_0, \eta)$ from the measure 
\[
    (\LF_{\bbD,\ell}^{(Q-\frac{1}{\gamma},0),(\gamma - \frac2\gamma,1)})^\#\times \mathrm{raSLE}_\kappa\]
and parametrize $\eta$ by its quantum length  measure; let $\qduration$ be its duration. For $s \in [0,\qduration)$, let $g_s: \D \backslash \eta([0,s]) \to \D$ be the conformal map such that $g_s(0) = 0$ and $g_s(\eta(s)) = 1$. Let $\phi_s = g_s \bullet_\gamma \phi_0$ and $L_s = \cL_{\phi_s}^\gamma(\partial \D)$. Almost surely $\qduration < \infty$ and $L_s = \ell+2s$ for all $s \in [0,\qduration)$. Moreover, for fixed $s>0$,  conditioned on $\{s < \qduration\}$ the conditional law of $\phi_s$ is $(\LF_{\D, L_s}^{(Q - \frac1\gamma, 0), (\gamma - \frac2\gamma,1)})^\#$.
\end{proposition}
Note that the final claim of Proposition~\ref{prop-radial-mot-simple-finite} also holds if we further condition on the process $(L_\cdot)_{[0,s]}$, since this process is deterministically given by $u \mapsto \ell + 2u$. 

\begin{proposition}\label{prop-radial-mot-swallow-finite}
Suppose $\kappa \in (4, 8)$ and $\gamma = 4/\sqrt\kappa \in (\sqrt2,2)$. Let $\ell > 0$. Sample $(\phi_0,\eta)$
    from the measure 
    \[
    (\LF_{\bbD, \ell}^{(Q-\frac{\gamma}{4},0),(\frac4\gamma - \frac\gamma2,1)})^\#\times \mathrm{raSLE}_\kappa
    \]
    and parametrize {$\eta$} by its quantum natural time; let $\qduration$ be its duration. For $s \in [0, \qduration)$, let $D_s$ be the connected component of $\D \backslash \eta([0,s])$ containing $0$, and let $g_s: D_s \to \D$ be the conformal map such that $g_s(0) = 0$ and $g_s(\eta(s)) = 1$. Let $\phi_s = g_s \bullet_\gamma \phi_0$ and $L_s = \cL_{\phi_s}^\gamma(\partial \D)$. Almost surely we have $\qduration < \infty$. Moreover, for fixed $s$, conditioned on $\{s < \cS\}$, let $\mathcal U$ be the set of quantum surfaces $(U, \phi_0)/{\sim_\gamma}$ where $U$ ranges over the connected components of $\D\backslash \eta([0,s])$ not containing 0. 
    If we further condition on $\mathcal U$ and $(L_\cdot)_{[0,s]}$, the conditional law of $\phi_s$ is $(\LF_{\D, L_s}^{(Q - \frac\gamma4, 0), (\frac4\gamma - \frac\gamma2, 1)})^\#$.
\end{proposition}
\begin{remark}[{Regularity of boundary length processes}]\label{rem-bdy-reg}
For Proposition~\ref{prop-radial-mot-sf-finite}, the boundary length process $(L_\cdot)_{[0,\qduration]}$ is a.s.\ continuous. For Proposition~\ref{prop-radial-mot-simple-finite}, we have $(L_\cdot)_{[0,\qduration)} = (\ell + 2s)_{0 \leq s < \qduration}$ a.s., thus can continuously extend it to $L_\qduration = \ell + 2 \qduration$.
In the setting of Proposition~\ref{prop-radial-mot-swallow-finite}, by local absolute continuity with respect to the non-simple quantum zipper \cite[Theorem 1.18, Corollary 1.19]{DMS14}, the process $(L_\cdot)_{[0,\qduration)}$ is c\'adl\'ag (right-continuous with left limits) and has jumps at a dense collection of times. Moreover, since the log-singularity at the origin has size $Q - \frac\gamma4 < Q$ we have $\lim_{s \to \qduration} L_s = 0$ a.s. The process should admit an exact description in terms of totally asymmetric stable L\'evy processes similar to that of \cite[Theorem 1.2]{ms-finite-mating}, but this has not been proved (Question~\ref{question-length-process}). However, see Proposition~\ref{prop-radial-forested-zipper} below for a variant in terms of \emph{forested quantum surfaces} for which the (generalized) boundary length process is simpler. 
\end{remark}

The special case of Proposition~\ref{prop-radial-mot-swallow-finite} for $(\gamma, \kappa) = (\sqrt{8/3}, 6)$ was already shown in 
\cite[Theorem 1.2]{ms-finite-mating}, except that the fields were described in terms of quantum disks rather than Liouville fields; the equivalence of the descriptions follows from \cite[Theorem 3.4]{ars-fzz}.

Finally, we describe the quantum surfaces cut out by the SLE processes in Propositions~\ref{prop-radial-mot-simple-finite} and~\ref{prop-radial-mot-swallow-finite}. Recall that $\QD(\ell)^\#$ (Definition~\ref{def-QD}) is the law of the quantum disk conditioned to have boundary length $\ell$.
\begin{lemma}\label{lem-simple-unexplored}
In the setting of Proposition~\ref{prop-radial-mot-simple-finite}, the conditional law of $(\D \backslash \eta([0,\qduration]), \phi_0)/{\sim_\gamma}$ given $\qduration$ is $\QD(\ell+2\qduration)^\#$.
\end{lemma}
Recall that a domain $U \subset \D$ is swallowed by $(K_\cdot)_{[0,\qduration]}$ if $U$ is one of the connected components of the interior of $K_s \backslash \bigcup_{u < s} K_u$ for some time $s < \qduration$; in this case, we also say the quantum surface $(U, \phi_0)/{\sim_\gamma}$ is swallowed at time $s$. 
\begin{lemma}\label{lem-swallowed-sle}
In the setting of Proposition~\ref{prop-radial-mot-swallow-finite}, the jumps of $(L_\cdot)_{[0,\qduration]}$ are in bijection with the quantum surfaces swallowed by the radial $\SLE_\kappa$ process, wherein at each jump time the process swallows a quantum surface whose $\gamma$-LQG boundary length is the jump size. 
Moreover, conditioned on $(L_{\cdot})_{[0,\qduration]}$, the swallowed quantum surfaces are conditionally independent, and the conditional law of the quantum surface corresponding to a jump of size $\ell$ is $\QD(\ell)^\#$.
\end{lemma}

In Section~\ref{sec-zip-simple} we prove Proposition~\ref{prop-radial-mot-simple-finite} and Lemma~\ref{lem-simple-unexplored}, and in Section~\ref{sec-zip-nonsimple} we prove Proposition~\ref{prop-radial-mot-swallow-finite} and Lemma~\ref{lem-swallowed-sle}.  

\subsection{The $\kappa < 4$ case}\label{sec-zip-simple}

The key input in the proofs of Proposition~\ref{prop-radial-mot-simple-finite} and Lemma~\ref{lem-simple-unexplored} is a conformal welding result (Proposition~\ref{prop-radial-weld-simple}) for the quantum disk with three boundary points, whose definition is via adding marked points to $\QD$ according to $\gamma$-LQG boundary length:

\begin{definition}\label{def-qd3}
    Let $(\D, \phi)$ be an embedding of a sample from $\int_0^\infty \frac12 \ell^3 \QD(\ell) \, d\ell$. Independently sample three marked points $x_1, x_2, x_3 \in \partial \D$ from $(\cL^\gamma_\phi)^\#$ conditioned on the event that the three points lie in counterclockwise order. Let $\QD_3$ denote the law of $(\D, \phi, x_1, x_2, x_3)/{\sim_\gamma}$.
\end{definition}
Let $\{ \QD_3(a,b,c)\}_{a,b,c >0}$ be the disintegration of $ \QD_3$ with respect to the generalized boundary lengths of the three boundary arcs in counterclockwise order from the first marked point. That is, $ \QD_3 = \iiint_0^\infty  \QD_3(a,b,c) \, da \, db\, dc$, and for a sample from $\QD_3(a,b,c)$ the boundary arcs counterclockwise from the first marked point have quantum lengths $a,b,c$. 

It was first shown in \cite{shef-zipper} that $\gamma$-LQG surfaces can be \emph{conformally welded} according to the $\gamma$-LQG boundary length measure to obtain an SLE curve. In Proposition~\ref{prop-radial-weld-simple} we state a variant due to \cite{ay-radial}. For $a,\ell > 0$, consider a sample from $ \QD_3(a,\ell,a)$. The two boundary arcs adjacent to the first marked point both have quantum length $a$; conformally welding them according to quantum length gives a quantum surface with the disk topology having two marked points (one in the bulk and one on the boundary) and a curve joining the two marked points. Denote its law by $\mathrm{Weld}( \QD_3(a,\ell, a))$. 

\begin{proposition}\label{prop-radial-weld-simple}
Let $\gamma \in (0,2)$ and $\kappa = \gamma^2$. We have 
\eqb\label{eq-radial-weld-simple}
\left(\int_0^\infty \mathrm{Weld}(\QD_{3}(a, \ell, a)) \, da \right)^\# = (\LF_{\mathbb D, \ell}^{(Q-\frac1\gamma, 0), (\gamma - \frac2\gamma, 1)})^\# \times \mathrm{raSLE}_\kappa
\eqe
    in the sense that, for a  curve-decorated quantum surface sampled from the left hand side of~\eqref{eq-radial-weld-simple}, embedding it in $(\bbD, 0, 1)$ yields a field and curve whose law is the right hand side of~\eqref{eq-radial-weld-simple}.  
\end{proposition}
\begin{proof}
    By \cite[Proposition 2.18, Remark 2.19]{ahs-integrability} $\QD_3$ is, up to multiplicative constant, the law of the three-pointed quantum surface arising from the Liouville field with three $\gamma$-insertions on the boundary. This exactly matches the definition of the \emph{quantum triangle} with weights $W_1 = W_2 = W_3 = 2$ introduced in \cite[Definition 2.17]{ASY22}. Thus, denoting the law of the latter by $\QT(2,2,2)$, there is  a constant $c$ such that $\QD_3 = c \QT(2,2,2)$. 
    On the other hand, \cite[Theorem 1.1]{ay-radial} with $W_1 = W_2 = W_3 = 2$ gives the variant of~\eqref{eq-radial-weld-simple} where $\QD_3$ is replaced by $\QT(2,2,2)$. These two results combined give~\eqref{eq-radial-weld-simple}.
\end{proof}

\begin{proof}[{Proof of Proposition~\ref{prop-radial-mot-simple-finite}}]
The first claim that $\cL_{\phi_s}^\gamma(\partial \D) = \ell + 2s$ holds tautologically, since the quantum length of a curve is defined via the quantum boundary length when the complement of the curve is uniformized into $\D$ (Definition~\ref{def-quantum-time-simple}). 

    For the second claim, by Proposition~\ref{prop-radial-weld-simple} the conditional law of $(\bbD, \phi_0, \eta, 0, 1)/{\sim_\gamma}$ given $\{ s < S\}$ is $(\int_s^\infty \mathrm{Weld}( {\QD}_3(a, \ell, a)) \, da )^\#$. By Definition~\ref{def-qd3}, for $a > s$, if we sample from $ \QD_3(a,\ell, a)$ and replace the second and third marked points by the points $s$ units of quantum length closer to the first marked point, the resulting quantum surface has law $ \QD_3(a-s, \ell+2s, a-s)$.
    Therefore, if we cut along the initial segment of $\eta$ of quantum length $s$, the resulting curve-decorated quantum surface has law $(\int_s^\infty \mathrm{Weld}( \QD_3(a - s, \ell+2s, a - s)) \, da )^\#$, hence by Proposition~\ref{prop-radial-weld-simple} the law of $\phi_s$ is $(\LF_{\mathbb D, \ell+2s}^{(Q-\frac1\gamma, 0), (\gamma - \frac2\gamma, 1)})^\#$.
\end{proof}

\begin{proof}[{Proof of Lemma~\ref{lem-simple-unexplored}}]
Consider the three-pointed quantum surface $\cD$ given by $(\D \backslash \eta([0,\qduration]), \phi_0)/{\sim_\gamma}$ with the three boundary points being $0$ and the two prime ends of $\D \backslash \eta([0,\qduration])$ corresponding to 1, with the boundary points given in counterclockwise order starting from 0. 
    By Proposition~\ref{prop-radial-weld-simple}, the law of $\cD$ is $\left(\int_0^\infty \mathrm{Weld}(\QD_{3}(a, \ell, a)) \, da \right)^\#$, so 
    the conditional law of $\cD$ given $\qduration$ is $\QD_3(\cS, \ell, \cS)$. By Definition~\ref{def-qd3}, the quantum surface obtained by forgetting the three points is $\QD(\ell + 2\cS)^\#$, as needed. 
\end{proof}

\subsection{The $\kappa \in (4,8)$ case}\label{sec-zip-nonsimple}
In this section we fix $\kappa \in (4,8)$ and $\gamma = 4/\sqrt\kappa \in (\sqrt2,2)$.
We will prove Proposition~\ref{prop-radial-mot-swallow-finite} and Lemma~\ref{lem-swallowed-sle}. 

We will work with the \emph{forested quantum surfaces} first constructed by \cite{DMS14}; our presentation will mostly follow \cite{ahsy-nonsimple}. We first prove Proposition~\ref{prop-radial-forested-zipper} (Figure~\ref{fig-nonsimple}, right), which is a forested analog of Proposition~\ref{prop-radial-mot-simple-finite} having essentially the same proof, and then use it to deduce Proposition~\ref{prop-radial-mot-swallow-finite} and Lemma~\ref{lem-swallowed-sle}.

A \emph{beaded quantum surface} is an equivalence class of pairs $(D,h)$, where $D \subset \mathbb C$ is a closed set such that each component of its interior together with its prime-end boundary is homeomorphic to the closed disk, and $h$ is defined as a distribution on each of these components. The equivalence relation is $(D, h) \sim_\gamma (\wt D, \wt h)$ if $\wt h = g \bullet_\gamma h$ for some homeomorphism $g: D \to \wt D$ that is conformal on each component of the interior of $D$. As in the case of quantum surfaces, we can decorate beaded quantum surfaces with marked points and curves.

\begin{figure}[ht]
	\centering
\includegraphics[scale=0.55]{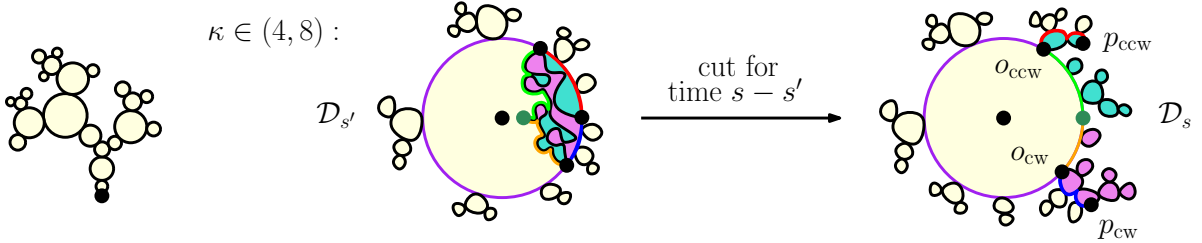}
	\caption{\textbf{Left:} A sample from $\GQD_1$ is a rooted looptree of countably many quantum disks. No two of the quantum disks have a common point; this is similar to the excursion decomposition of standard Brownian motion, wherein no two excursions share an endpoint. Similarly, the root does not lie on the boundary of any of the quantum disks, but rather at the end of an infinite chain of small quantum disks. \textbf{Right:} Illustration for the proof of Proposition~\ref{prop-radial-mot-swallow-finite} at the end of this section. In Proposition~\ref{prop-radial-forested-zipper}, cutting $\cD_0 \sim (\LF_{\bbD, f, \ell_0}^{(Q-\frac{\gamma}{4},0),(\frac4\gamma - \frac\gamma2,1)})^\#$ by an independent radial $\SLE_\kappa$ $\eta$ for time $s$ gives $\cD_s \sim (\LF_{\bbD, f, \ell_0+2s}^{(Q-\frac{\gamma}{4},0),(\frac4\gamma - \frac\gamma2,1)})^\#$. Let $\phi_s$ be the field obtained from $\cD_s$ by forgetting the boundary foresting and embedding in $(\D, 0, 1)$, let $L_s = \cL^\gamma_{\phi_{s}}(\partial \D)$, and let $\mathfrak F_s$ be the boundary foresting of $\cD_s$. For $s' < s$, let $\mathcal U_{[s', s)}$ be the set of quantum surfaces swallowed by $\eta$ in the time interval $[s', s)$. 
    Suppose $s'$ is such that the curve $\eta|_{[s',s]}$ does not close a loop around $0$.      
    The set $\mathcal U_{[s', s)}$ is depicted by blue and pink disks,  $L_{s'}$ is the sum of the red, purple and blue lengths (which are respectively $\ell_\mathrm{ccw}, x_\mathrm{cw} - x_\mathrm{ccw}$ and $\ell_\mathrm{cw}$), and $\mathfrak F_{s'}$ is the collection of yellow looptrees, so $(L_{s'},\mathcal U_{[s',s)}, \mathfrak F_{s'})$ is measurable with respect to $\mathfrak F_s$.} \label{fig-nonsimple}
\end{figure}

There is an infinite measure called $\GQD_1$ on the space of beaded quantum surfaces with one boundary point, a sample of which is a rooted looptree of quantum disks; see Figure~\ref{fig-nonsimple} (left). Roughly speaking, one constructs a looptree from an excursion of a stable L\'evy process with index $4/\gamma^2$ \cite{curien-kortchemski-looptree-def}, then independently samples for each loop of length $\ell$ a quantum disk with boundary length $\ell$. 
We call a sample $\mathcal D$ from $\GQD_1$ a \emph{generalized quantum disk}. This object was first defined in \cite{DMS14}; we will follow the presentation of \cite[Sections 3.1, 3.2]{ahsy-nonsimple}, but see \cite{msw-nonsimple,HL22CLE} for alternative presentations.

The boundary of $\mathcal D$ is equipped with a \emph{generalized boundary length} measure which we denote by $\mathcal{GL}_{\mathcal D}$. The $\mathcal{GL}_{\mathcal D}$-length of a boundary arc $\mathcal C$ is $\lim_{\eps \to 0} \mathfrak c_0 \eps^{4/\gamma^2} N_\eps(\mathcal C)$ where $N_\eps(\mathcal C)$ is the number of quantum disks with boundary length between $\eps$ and $2\eps$ whose boundary is fully traced by $\mathcal C$, where $\mathfrak c_0$ is the constant from Definition~\ref{def-quantum-time-swallowing}. The measure $\GQD_1$ admits a disintegration $\int_0^\infty \GQD_1(\ell)\,d\ell$ according to generalized boundary length (so a sample from $\GQD_1(\ell)$ a.s.\ has generalized boundary length $\ell$), and $|\GQD_1(\ell)|<\infty$ for all $\ell>0$. 

The root of the generalized quantum disk is a typical point with respect to the generalized boundary length measure, in the sense that for a sample $\cD$ from $\GQD_1$, the conditional law of the marked point given the (unmarked) beaded quantum surface is $(\mathcal{GL}_{\mathcal D})^\#$. This allows us to define an analog of Definition~\ref{def-qd3} as follows.

\begin{definition}\label{def-forested-LF}
    Given a sample $\mathcal D$ from $\int_0^\infty \frac12 \ell^2 \GQD_1(\ell) \, d\ell$, independently sample second and third marked points from $(\mathcal{GL}_{\mathcal D})^\#$ conditioned on the event that the three points lie in counterclockwise order on the boundary. Let $\GQD_3$ be the law of the resulting beaded quantum surface with three marked points. 
\end{definition}
Let $\{\GQD_3(a,b,c)\}_{a,b,c >0}$ be the disintegration of $\GQD_3$ with respect to the generalized boundary lengths of the three boundary arcs in counterclockwise order from the first marked point. 

The other kind of beaded quantum surface we need comes from ``foresting'' the boundary of a Liouville field by attaching a L\'evy collection of generalized quantum disks to its boundary. Let $\mathfrak c > 0$ be the constant from \cite[Proposition 3.11]{ahsy-nonsimple}; we will not need its value.

\begin{definition}\label{def-LFf}
    Let $\alpha, \beta \in \R$. Sample $\phi \sim \LF_\bbD^{(\alpha, 0), (\beta, 1)}$, and conditioned on $\phi$ sample a Poisson point process with intensity measure $\mathfrak c \GQD_1 \times \cL_\phi^\gamma$. For each element $(\cD_u, u)$ of the point process, 
    root $\cD_u$ at the boundary point $u\in \partial \D$ of $(\D, \phi, 0, 1)$; this gives a beaded quantum surface $\cD$ with a bulk marked point 0 and a boundary marked point 1.  
    Let $\LF_{\bbD,f}^{(\alpha, 0),(\beta, 1)}$ denote the law of $\cD$.
\end{definition}
The boundary of $\cD$ inherits a  generalized boundary length measure $\mathcal{GL}_{\cD}$  from the corresponding generalized boundary length measures of the attached generalized quantum disks. The total generalized boundary length $|\mathcal{GL}_{\cD}|$ is finite when $|\cL_\phi^\gamma(\partial \D)| < \infty$ (more strongly, if we condition on $\cL_\phi^\gamma(\partial \D)$ then  $|\mathcal{GL}_{\cD}|$ has conditional moments of order $p < \frac{\gamma^2}4$ \cite[Lemma 3.2]{ahsy-nonsimple}). Thus, when $\beta < Q$ the generalized boundary length is finite a.s., and we denote the disintegration of $\LF_{\bbD,f}^{(\alpha, 0),(\beta, 1)}$ with respect to generalized boundary length by $\{\LF_{\bbD,f, \ell}^{(\alpha, 0),(\beta, 1)}\}_{\ell > 0}$.

Finally, as explained in \cite[Theorem 1.4.8]{DMS14}, there is a way to conformally weld certain beaded quantum surfaces called \emph{forested quantum wedges} with respect to generalized boundary length, to obtain a beaded quantum surface decorated by an independent $\SLE_{\kappa}$ curve. Moreover, the welded curve-decorated beaded quantum surface is a measurable function of the two original forested quantum wedges. This gives a solution for the conformal welding problem for forested quantum surfaces; this solution is unique among all curves satisfying certain regularity conditions \cite{mmq-welding}, but it is known that for $\kappa$ close to $8$ there exist other solutions \cite{lz-sierpinski}. In what follows, whenever we refer to the conformal welding of forested quantum surfaces, we mean the conformal welding that comes from \cite{DMS14}; see the discussion immediately above \cite[Theorem 1.4]{ahsy-nonsimple} for more details. 

For $a, \ell > 0$, consider a sample from $\GQD_3(a,\ell,a)$. Conformally welding the two boundary arcs adjacent to the first marked point according to generalized boundary length gives a beaded quantum surface having a marked bulk point and a marked boundary point, and a curve joining the two points. Denote its law by $\mathrm{Weld}(\GQD_{3}(a, \ell, a))$.

\begin{proposition}\label{prop-radial-weld-ns}
Let $\kappa \in (4,8)$ and $\gamma = 4/\sqrt\kappa \in (\sqrt2,2)$. We have 
\eqb\label{eq-radial-weld-ns}
\left(\int_0^\infty \mathrm{Weld}(\GQD_{3}(a, \ell, a)) \, da \right)^\# = (\LF_{\mathbb D, f, \ell}^{(Q-\frac\gamma4, 0), (\frac4\gamma - \frac\gamma2, 1)})^\# \times \mathrm{raSLE}_\kappa.
\eqe 
\end{proposition}
\begin{proof}
Let $\mathrm{QT}^f(\gamma^2-2, \gamma^2-2, \gamma^2-2)$ denote the law of the \emph{forested quantum triangle} with weights $W_1 = W_2 = W_3 = \gamma^2-2$; this corresponds to the law of a quantum  triangle with forested boundary (similarly as in Definition~\ref{def-LFf}). By \cite[Lemma 6.4]{ay-radial} and \cite[Definition 3.9]{ahsy-nonsimple}, there is a constant $C$ such that $\GQD_3 = C \QT^f(\gamma^2-2, \gamma^2-2, \gamma^2-2)$. On the other hand, \cite[Theorem 1.3]{ay-radial} states that~\eqref{eq-radial-weld-ns} holds when $\GQD_3$ is replaced by $\QT^f(\gamma^2-2, \gamma^2-2, \gamma^2-2)$. Combining these two results gives the claim. 
\end{proof}

The following is the forested analog of Proposition~\ref{prop-radial-mot-simple-finite}, see Figure~\ref{fig-nonsimple} (right). For a sample $(\Phi_0,\eta)$ from $(\LF_{\bbD, f, \ell}^{(Q-\frac{\gamma}{4},0),(\frac4\gamma - \frac\gamma2,1)})^\#\times \mathrm{raSLE}_\kappa$, we view $\eta$ as a curve on $\Phi_0$ by embedding the connected component of $\Phi_0$ containing the bulk and boundary marked point in $(\D, 0, 1)$. 

\begin{proposition}\label{prop-radial-forested-zipper}
Suppose $\kappa \in (4, 8)$ and $\gamma = \frac4{\sqrt\kappa} \in (\sqrt2,2)$. Let $\ell_0 > 0$. Sample $(\cD_0,\eta)$
    from the measure 
    \[
    (\LF_{\bbD, f, \ell_0}^{(Q-\frac{\gamma}{4},0),(\frac4\gamma - \frac\gamma2,1)})^\#\times \mathrm{raSLE}_\kappa
    \]
    and parametrize {$\eta$} by its quantum natural time; let $\qduration$ be its duration. For $s \in [0, \qduration)$, let $\cD_s$ be the forested quantum surface obtained from $\cD_0$ by cutting along $\eta$ for quantum natural time $s$. Almost surely $\qduration < \infty$. Moreover, for all $s \in [0,\qduration)$ the generalized boundary length of $\cD_s$ is $\ell_0 + 2s$, and for fixed $s > 0$, conditioned on $\{s < \qduration\}$ the conditional law of $\cD_s$ is $(\LF_{\bbD, f, \ell_0+2s}^{(Q-\frac{\gamma}{4},0),(\frac4\gamma - \frac\gamma2,1)})^\#$.
\end{proposition}

\begin{proof}
The proof is identical to that of Proposition~\ref{prop-radial-mot-simple-finite}, except we use Proposition~\ref{prop-radial-weld-ns} in place of Proposition~\ref{prop-radial-weld-simple}.  
\end{proof}

\begin{proof}[{Proof of Proposition~\ref{prop-radial-mot-swallow-finite}}]
Fix $s >0$. Sample $(\cD_0, \eta)$ from $ (\LF_{\bbD, f, \ell_0}^{(Q-\frac{\gamma}{4},0),(\frac4\gamma - \frac\gamma2,1)})^\#\times \mathrm{raSLE}_\kappa$.
By Proposition~\ref{prop-radial-forested-zipper}, conditioned on $s < \qduration$, cutting along $\eta$ for quantum natural time $s$ yields a forested quantum surface $\cD_s$ with law $(\LF_{\bbD, f, \ell_0+2s}^{(Q-\frac{\gamma}{4},0),(\frac4\gamma - \frac\gamma2,1)})^\#$.

For $s \in [0,\qduration)$, let $\phi_s$ be the field in $\D$ such that $(\D, \phi_s, 0, 1)/{\sim_\gamma}$ is the connected component of $\cD_s$ containing the marked bulk and boundary points (i.e., the field obtained by discarding the foresting of $\cD_s$ and embedding in $(\D, 0, 1)$). By Definition~\ref{def-forested-LF} $\cD_s$ is obtained from $(\D, \phi_s, 0, 1)$ by attaching generalized quantum disks to $\partial \D$; let $\mathfrak F_s$ be the set of pairs $(\cD, x)$ where $\cD$ is an attached generalized quantum disk and $x$ is the quantum length of the counterclockwise boundary arc of $\partial \D$ from $1$ to the attachment point of $\cD$. 

Fix $s$. For any $0 \leq s' \leq s$, let $\mathcal U_{[s',s)}$ be the collection of quantum surfaces swallowed by $\eta$ in the time interval $[s', s)$. 
We now show that if $s'$ is a time such that $\eta|_{[s', s]}$ does not close a loop around the origin (i.e., there is a path from $0$ to $\partial \D$ that is disjoint from $\eta([s', s])$), then $((L_\cdot)_{[s', s]}, \mathcal U_{[s', s)}, \mathfrak F_{s'})$ is measurable with respect to $\mathfrak F_s$; see Figure~\ref{fig-nonsimple} (right). First, $\mathcal U_{[s', s)}$ is the set of quantum disks of the forested boundary of $\cD_s$ which touch the clockwise and counterclockwise boundary arcs from 1 with generalized boundary length $(s-s')$, hence is measurable with respect to $\mathfrak F_s$. Second, let $p_\mathrm{ccw}$ be the boundary point at generalized boundary length $s - s'$ counterclockwise from $1$. Let $(\cD_\mathrm{ccw}, x_\mathrm{ccw}) \in \mathfrak F_s$ correspond to the generalized quantum disk attached to $\partial \D$ containing $p_\mathrm{ccw}$, and for the chain of quantum disks in $\cD_\mathrm{ccw}$ between $p_\mathrm{ccw}$ and the root $o_\mathrm{ccw}$ of $\cD_\mathrm{ccw}$, let $\ell_\mathrm{ccw}$ be the quantum length of the counterclockwise boundary arc $\cC_\mathrm{ccw}$ from $p_\mathrm{ccw}$ to $o_\mathrm{ccw}$. Similarly, define $p_\mathrm{cw}$, $(\cD_\mathrm{cw}, x_\mathrm{cw}), \cC_\mathrm{cw}$ and $\ell_\mathrm{cw}$ by replacing every instance of ``counterclockwise'' with ``clockwise'' in the preceding. Then $L_{s'} = \ell_\mathrm{ccw} + (x_\mathrm{cw} - x_\mathrm{ccw}) + \ell_\mathrm{cw}$, so $L_{s'}$ is also measurable with respect to $\mathfrak F_s$; the same argument with $s'$ replaced by any time in $(s', s)$ yields that $(L_\cdot)_{[s', s]}$ is measurable with respect to $\mathfrak F_s$. Third, consider the curve $\cC_{s'}$ from $p_\mathrm{ccw}$ to $p_\mathrm{cw}$ obtained by concatenating $\cC_\mathrm{ccw}$, the counterclockwise arc of $\partial \D$ from $o_\mathrm{ccw}$ to $o_\mathrm{cw}$, and the time-reversal of $\cC_\mathrm{cw}$. Then $\mathfrak F_{s'}$ is the collection of pairs $(\cD, u)$ where $\cD$ is a looptree of quantum disks attached to $\cC_{s'}$ and $u$ is the quantum length along $\cC_{s'}$ from $p_\mathrm{ccw}$ to the attachment point of $\cD$. 

Thus, on the event that $\eta|_{[s',s]}$ does not close a loop around the origin, the tuple $((L_\cdot)_{[s', s]}, \mathcal U_{[s', s)}, \mathfrak F_{s'})$ is measurable with respect to $\mathfrak F_s$. As $\eta$ is a continuous curve that is bounded away from $0$, it makes at most finitely many loops around $0$, so inductively applying the above measurability claim a finite number of times, we conclude that $((L_\cdot)_{[0,s]}, \mathcal U)$ is measurable with respect to $\mathfrak F_s$, where $\mathcal U = \mathcal U_{[0,s)}$ is the collection of all quantum surfaces swallowed.

By Definition~\ref{def-forested-LF}, the conditional law of $\phi_0$ given $L_0$ is $(\LF_{\D, L_0}^{(Q-\frac{\gamma}{4},0),(\frac4\gamma - \frac\gamma2,1)})^\#$, so conditioning on $L_0 = \ell$ puts us in the setting of Proposition~\ref{prop-radial-mot-swallow-finite}. Thus, we must show that conditioned on $((L_\cdot)_{[0,s]}, \mathcal U)$, the conditional law of $\phi_s$ is $(\LF_{\D, L_s}^{(Q-\frac{\gamma}{4},0),(\frac4\gamma - \frac\gamma2,1)})^\#$. 
Since $(L_\cdot)_{[0,s]}$ and $\mathcal U$ are measurable with respect to $\mathfrak F_s$, it suffices to show that conditioned on $(L_s, \mathfrak F_s)$, the conditional law of $\phi_s$ is $(\LF_{\D, L_s}^{(Q-\frac{\gamma}{4},0),(\frac4\gamma - \frac\gamma2,1)})^\#$. 

By Proposition~\ref{prop-radial-forested-zipper} the law of $\cD_s$ is  $(\LF_{\bbD, f, \ell_0+2s}^{(Q-\frac{\gamma}{4},0),(\frac4\gamma - \frac\gamma2,1)})^\#$, so by Definition~\ref{def-forested-LF}, conditioned on $L_s$, the field $\phi_s$ and boundary foresting $\mathfrak F_s$ are conditionally independent, with $\phi_s \sim (\LF_{\D, L_s}^{(Q-\frac{\gamma}{4},0),(\frac4\gamma - \frac\gamma2,1)})^\#$, and $\mathfrak F_s$ a Poisson point process with intensity measure $\mathfrak c \GQD \times \mathrm{Leb}_{[0,L_s]}$ conditioned to have generalized boundary length $\ell_0 + 2s$. Thus, conditioned on $(L_s, \mathfrak F_s)$ the conditional law of $\phi_s$ is $(\LF_{\D, L_s}^{(Q-\frac{\gamma}{4},0),(\frac4\gamma - \frac\gamma2,1)})^\#$. This completes the proof. 
\end{proof}

\begin{proof}[{Proof of Lemma~\ref{lem-swallowed-sle}}]
The first claim that the jumps of $(L_\cdot)_{[0,\qduration]}$ are in bijection with the quantum surfaces swallowed follows from local absolute continuity of the field and curve with respect to \cite[Theorem 1.4.7]{DMS14}. For the second claim, similarly as in the proof of Proposition~\ref{prop-radial-mot-swallow-finite}, it suffices to prove the claim when $(\phi_0, \eta)$ is replaced by a sample $(\cD_0, \eta)$ from the right hand side of $\left(\int_0^\infty \mathrm{Weld}(\GQD_{3}(a, \ell, a)) \, da \right)^\# = (\LF_{\mathbb D, f, \ell}^{(Q-\frac\gamma4, 0), (\frac4\gamma - \frac\gamma2, 1)})^\# \times \mathrm{raSLE}_\kappa$. By Proposition~\ref{prop-radial-weld-ns}, cutting by $\eta$ yields a generalized quantum disk $\cD$ with generalized boundary length $\ell + 2\qduration$. To sample a generalized quantum disk, one first samples a random looptree, then independently samples for each loop of length $\ell$ a quantum disk of boundary length $\ell$. Since $(L_\cdot)_{[0,\qduration]}$ is measurable with respect to the looptree of $\cD$, when we condition on $(L_\cdot)_{[0,\qduration]}$ the swallowed surfaces are conditionally independent quantum disks with specified boundary lengths, as needed. 
\end{proof}

\section{The $\delta$-$\QLE$ process}\label{sec-qle-approximates}
In this section we define a $\delta$-approximation of $\QLE$, where we iteratively explore an LQG surface by a series of SLE curves, reshuffling the point from which the SLE curve grows every $\delta$ units of quantum natural time. 

Recall that the probability measure $P_{\alpha, \beta, \ell}$ from Definition~\ref{def-P-weighted} is the law of a GFF variant with $\gamma$-LQG boundary length $\ell$. For any hull $K$ let $g_K: \D \backslash K \to \D$ be the conformal map satisfying $g_K(0) = 0$ and $g_K'(0) > 0$. If $D\subset \D$ is a simply-connected domain containing 0, we also define $g_D : D \to \D$ by $g_D = g_K$ where $K = \D \backslash D$.

\begin{definition}[$\delta$-QLE$(\gamma^2, \eta)$ process] \label{def-delta-QLE-finite}
Consider either $\gamma \in (\sqrt3 -1,2)$ and $\eta = \frac3{\gamma^2} - \frac12$, or $\gamma \in (0, 4/3) \cup (2\sqrt3 - 2,2)$ and $\eta = \frac{3\gamma^2}{16} - \frac12$. In the former case set $\kappa = \gamma^2$, whereas in the latter case set $\kappa = 16/\gamma^2$.
Let $(\alpha,\beta)$ be defined as follows:
\begin{equation*} (\alpha, \beta) = 
        \begin{cases} 
          (Q-\frac{1}{\gamma},\gamma - \frac2\gamma) & \gamma \in (\sqrt3 - 1,2), \, \, \eta = \frac3{\gamma^2} - \frac12 \\
            (Q-\frac{\gamma}{4},\frac4\gamma - \frac\gamma2) & \gamma \in (2\sqrt3 - 2, 2), \, \, \eta = \frac{3\gamma^2}{16} - \frac12 \\
          (Q-\frac{\gamma}{4},\frac{3\gamma}{2}) & \gamma \in (0, 4/3), \, \,  \eta = \frac{3\gamma^2}{16} - \frac12 
       \end{cases}
    \end{equation*}
Sample $\phi^\delta \sim P_{\alpha, \beta, 1}$. We will define a shrinking process of origin-containing domains $D_s^\delta \subset \D$ beginning with $D_0^\delta = \D$. 
Iteratively for $n \in \mathbb Z_{\geq 0}$ we carry out the following:
    \begin{itemize}
        \item (\textbf{Resample the curve tip}): 
        Let $\phi^\delta_{n\delta} = g_{D^\delta_{n\delta}} \bullet_\gamma \phi^\delta$. 
        Sample a point from the probability measure $(\cL_{\phi_{n\delta}^\delta}^{\beta})^\#$ on $\partial \D$ and let $p^\delta_n \in \partial D_{n\delta}^\delta$ be its image under $(g_{D_{n\delta}^\delta})^{-1}$. Note that $p^\delta_n$ is understood as a prime end\footnote{If $\kappa \leq 4$, each point on the interior of an SLE$_\kappa$ segment corresponds to two prime ends in $\partial D^\delta_{n\delta}$, namely the boundary points on the left and right of the curve segment.} of $D_{n\delta}^\delta$. 
        
        \item (\textbf{Grow an $\SLE_\kappa$ curve}): Let $\eta^\delta_n$ be radial SLE$_\kappa$ in $D^\delta_{n\delta}$ from $p^\delta_n$ to $0$. 
        Parametrize $\eta_n^\delta$ by quantum natural time as in Definitions~
        \ref{def-quantum-time-simple}-\ref{def-quantum-time-sf}, and let $\qduration^\delta_n$ be its duration. 
        For $s \in (n \delta, n\delta + \min(\delta, \qduration^\delta_n)]$ let $D^\delta_s$ be the connected component of $D^\delta_{n\delta} \backslash \eta^\delta_n([0,s-n\delta])$ which contains 0. If  $\qduration^\delta_n \leq \delta$, set $\qduration^\delta = n\delta + \qduration^\delta_n$ and terminate the process at time  $\qduration^\delta$. 
    \end{itemize}
    If the process does not terminate, set $\qduration^\delta = \infty$. Let $K^\delta_s = \ol\D \backslash D^\delta_s$, so $(K^\delta_\cdot)_{[0,\qduration^\delta]}$ is a growing family of hulls. 
    We call $(\phi^\delta, (K^\delta_\cdot)_{[0,  \qduration^\delta]})$ a $\delta$-$\mathrm{QLE}(\gamma^2, \eta)$ process.    
\end{definition}

\begin{remark}
For $\eta = \frac{3\gamma^2}{16} - \frac12$, the above definition only applies to the restricted parameter range $\gamma \in (0, 4/3) \cup (2\sqrt3 - 2,2)$. The reason for this is that when $\gamma \in [4/3, 2\sqrt3-2]$ we have $\beta \geq 2$, so the GMC measure $\cL^\beta_{\phi^\delta_{n\delta}}$ is the zero measure. Consequently, we cannot resample the curve tip from $(\cL^\beta_{\phi^\delta_{n\delta}})^\#$. For the same reason, for $\eta = \frac{3}{\gamma^2} - \frac12$ we only consider $\gamma \in (\sqrt3 - 1, 2)$, since $\gamma \in (0,\sqrt3-1]$ gives $\beta \leq -2$.
\end{remark}

In Section~\ref{subsec-properties-delta-qle}, we discuss some basic properties of $\delta$-QLE which can be derived from the stationarity properties of the corresponding LQG/SLE couplings from Section~\ref{sec-zippers}. In Section~\ref{sec-scaling-rels} we explain how Definition~\ref{def-delta-QLE-finite} describes, at least heuristically, a version of $\eta$-DBM on $\gamma$-LQG where the particles are $\delta$-length segments of SLE; in particular, it explains the relationship between $\gamma, \eta$ and $\beta$.

\subsection{Properties of $\delta$-QLE}\label{subsec-properties-delta-qle}
In Lemmas~\ref{lem-duration-indep}--\ref{lem-swallowed-delta-qle} we will show that many key observables of $\delta$-QLE agree in law with the corresponding explorations of LQG via SLE described in Propositions~\ref{prop-radial-mot-sf-finite}-\ref{prop-radial-mot-swallow-finite}, and in Lemma~\ref{lem-tight-t} we obtain a tightness result for the capacity time at a fixed quantum natural time $s$. 
If $\kappa \in (0,4)$, let $(\phi_0^\SLE, \eta^\SLE)$, $\qduration^\SLE$, $(\phi_s^\SLE)_{[0,\qduration^\SLE)}$ and $(L_s^\SLE)_{[0,\qduration^\SLE)}$ be the quantities defined in Proposition~\ref{prop-radial-mot-simple-finite} with initial boundary length $\ell = 1$ (we add the ``$\SLE$'' superscript to emphasize the SLE setting and avoid later notational conflict). Similarly, if $\kappa \in (4,8)$ or $\kappa \in (8,\infty)$, define these quantities as in Proposition~\ref{prop-radial-mot-swallow-finite} or Proposition~\ref{prop-radial-mot-sf-finite} respectively with $\ell = 1$. 

The first claim is that the $\delta$-QLE duration and boundary length process agree in law with those of the corresponding SLE exploration of LQG; moreover, the field $\phi^\delta_s$ satisfies a stationarity property inherited from the SLE exploration of LQG. 

\begin{lemma}\label{lem-duration-indep} 
Defining $\phi_s^\delta := g_{K^\delta_s} \bullet_\gamma \phi^\delta$ and  $L_s^\delta := \cL^\gamma_{\phi_s^\delta}(\partial \D)$ for $s <\qduration^{\delta}$, we have $(\qduration^\delta,(L_\cdot^\delta)_{[0,\qduration^\delta)}) \stackrel d= (\qduration^\SLE, (L_\cdot^\SLE)_{[0,\qduration^\SLE)})$. In particular $\qduration^\delta$ is a.s.\ finite. Moreover, for any $s>0$, conditioned on $\{s < \qduration^{\delta}\}$  and on $(L^\delta_\cdot)_{[0,s]}$, the conditional law of $\phi_s^\delta$ is $P_{\alpha, \beta, L_s^\delta}$.
\end{lemma}
\begin{proof}
    Recall from Definition~\ref{def-delta-QLE-finite} that $p^\delta_0 \sim ( \cL^\beta_{\phi^\delta_0})^\#$, and let $R_{p^\delta_0}: \D \to \D$ denote the rotation that sends $p^\delta_0$ to $1$. Let $\tilde \phi^\delta_0 = R_{p^\delta_0} \bullet_\gamma \phi^\delta$, for each $s < \qduration^\delta$ let $\tilde D^\delta_s = R_{p^\delta_0}( D^\delta_s)$, and for each $n$ let $\tilde p^\delta_{n} = R_{p^\delta_0}(p^\delta_n)$ and $\tilde \eta^\delta_n = R_{p^\delta_0} \circ \eta^\delta_n$. In other words, under the rotation $R_{p^\delta_0}$, the objects $\phi^\delta,  D^\delta_s, p^\delta_{n}, \eta^\delta_n$  become $\tilde \phi^\delta_0, \tilde D^\delta_s, \tilde p^\delta_{n}, \tilde \eta^\delta_n$. 

    For each $s < \qduration^\delta$, let $g^\delta_s: \tilde D^\delta_s \to \D$ be the conformal map sending $0$ to $0$ and sending the curve tip $\tilde \eta^\delta_{\lfloor s/\delta\rfloor}(s - \lfloor s/\delta \rfloor \delta)$ to $1$, and let $\tilde \phi^\delta_s = g^\delta_s \bullet_\gamma \tilde \phi^\delta_0
    $. Note that the fields $\tilde \phi^\delta_s$ and $R_{p^\delta_0} \bullet_\gamma \phi^\delta_s$ differ by a random rotation.

    By Lemma~\ref{lem-sample-gmc-2} the joint law of $(\tilde\phi^\delta_0, \tilde \eta^\delta_0, p^\delta_0)$ is $(\LF_{\D, 1}^{(\alpha, 0), (\beta, 1)})^\#\times \mathrm{raSLE}_\kappa \times \mathrm{Leb}_{\partial \D}^\#$. 
    Since $(\tilde \phi^\delta_0, \tilde \eta^\delta_0) \stackrel d= (\phi^\SLE_0, \eta^\SLE)$ (where this latter pair is defined immediately above the lemma statement), we have 
    \eqb\label{eq-cond-field}
    \begin{gathered}
    (\qduration^\delta \wedge \delta, (L_\cdot^\delta)_{[0, \qduration^\delta \wedge \delta)}) \stackrel d= (\qduration^\SLE \wedge \delta,(L_\cdot^\SLE)_{[0,\qduration^\SLE \wedge \delta)});
        \\
        \text{For }s \in [0,\delta), \text{ conditioned on } s< \qduration^\delta \text{ and } (L_\cdot^\delta)_{[0, s]}, \text{ the conditional law of }\tilde \phi^\delta_s \text{ is } (\LF_{\D, L_s^\delta}^{(\alpha,0), (\beta, 1)})^\#.
    \end{gathered}
    \eqe
    Let $\tilde \phi^\delta_{\delta-} = g^\delta_{\delta-} \bullet_\gamma \tilde \phi^\delta_0$ where $g^\delta_{\delta-}: \tilde D^\delta_s \to \D$ is the conformal map sending $(0, \tilde \eta^\delta_0(\delta))$ to $(0,1)$; that is, $\tilde \phi^\delta_{\delta-}$ is the field at time $\delta$ before reshuffling the marked boundary point. By the argument of~\eqref{eq-cond-field}, conditioned on $\delta < \qduration^\delta$ and $(L_\cdot^\delta)_{[0, \delta]}$, the conditional law of $\tilde \phi^\delta_{\delta-}$ is $(\LF_{\D, L_\delta^\delta}^{(\alpha,0), (\beta, 1)})^\#$. Thus, by Lemma~\ref{lem-sample-gmc-2}, conditioned on $\delta < \qduration^\delta$ and $(L_\cdot^\delta)_{[0,\delta]}$, the conditional law of $(\tilde \phi^\delta_\delta, g^\delta_\delta \circ \tilde \eta^\delta_1)$ is $(\LF_{\D, L^\delta_\delta}^{(\alpha,0), (\beta, 1)})^\# \times \mathrm{raSLE}_\kappa$. Again, we are in the setting of Propositions~\ref{prop-radial-mot-sf-finite}-\ref{prop-radial-mot-swallow-finite} (this time with initial boundary length $L_\delta^\delta$), hence iterating the argument yields for $k=2$ 
    \eqb\label{eq-cond-field-full}
       \begin{gathered}
            (\qduration^\delta \wedge k\delta,(L_\cdot^\delta)_{[0, \qduration^\delta \wedge k\delta)}) \stackrel d= (\qduration^\SLE \wedge k\delta,(L_\cdot^\SLE)_{[0,\qduration^\SLE \wedge k\delta)});\\
           \text{For }s \in [0,k\delta), \text{ conditioned on } s< \qduration^\delta \text{ and } (L_\cdot^\delta)_{[0, s]}, \text{ the conditional law of }\tilde \phi^\delta_s \text{ is } (\LF_{\D, L_s^\delta}^{(\alpha,0), (\beta, 1)})^\#.
       \end{gathered}
    \eqe
    Proceeding iteratively as before yields~\eqref{eq-cond-field-full} for all $k$. We conclude that $(\qduration^\delta,(L^\delta_\cdot)_{[0,\qduration^\delta)}) \stackrel d= (\qduration^\SLE,(L_\cdot^\SLE)_{[0,\qduration^\SLE)})$. 
    
    Finally, fix $s > 0$, and condition on $s< \qduration^\delta$ and $(L_\cdot^\delta)_{[0, s]}$. We will identify the law of $\phi^\delta_s$ by showing it agrees with $\tilde \phi_s^\delta$ after applying a uniform random rotation of $\D$ whose angle is independent of $\tilde \phi_s^\delta$. Let $\theta = \arg((g^\delta_s)'(0)) \in [0,2\pi)$ and let $R_{e^{i\theta}}: \D \to \D$ be the rotation sending $e^{i\theta}$ to 1, so $g = R_{e^{i\theta}} \circ g^\delta_s$ is the  conformal map from $\tilde D^\delta_s$ to $\D$ with $g(0) = 0$ and $g'(0) > 0$. Then $R_{e^{i\theta}} \circ g^\delta_s \circ R_{p^\delta_0} =  R_{p^\delta_0} \circ g_{D^\delta_s}$; indeed, the left and right hand sides are conformal maps from $D^\delta_s$ to $\D$ fixing $0$ and having the same argument of derivative at 0. Since $\tilde \phi^\delta_s = (g^\delta_s \circ R_{p^\delta_0}) \bullet_\gamma \phi^\delta_0$
 and $\phi^\delta_s = g_{D^\delta_s} \bullet_\gamma \phi^\delta$, this yields  $ R_{e^{i\theta}} \bullet_\gamma \tilde \phi^\delta_s = R_{p^\delta_0} \bullet_\gamma \phi^\delta_s$, hence $\phi^\delta_s = (R_{p_0^\delta}^{-1} \circ R_{e^{i\theta}}) \bullet_\gamma \tilde \phi^\delta_s$. Since $p^\delta_0$ is independent of $(\tilde \phi^\delta_s, \theta)$ and is uniformly distributed on $\partial \D$, we conclude that $R_{p_0^\delta}^{-1} \circ R_{e^{i\theta}}$ is a uniformly random rotation of $\partial \D$ which is independent of $\tilde \phi^\delta_s$. By~\eqref{eq-cond-field-full} the law of $\tilde \phi^\delta_s$ is $(\LF_{\D, L_s^\delta}^{(\alpha,0), (\beta, 1)})^\#$, hence by Lemma~\ref{lem-sample-gmc-2} the law of $\phi^\delta_s$ is $P_{\alpha, \beta, L^\delta_s}$.
\end{proof}

The same proof generalizes in the swallowing regime of $\delta$-QLE$(\gamma^2,\eta)$ to give the same Markov property when we further condition on the swallowed quantum surfaces.

\begin{lemma}\label{lem-swallowing-markov}
Let $s>0$. In the setting of Definition~\ref{def-delta-QLE-finite} and Lemma~\ref{lem-duration-indep}  with $\gamma \in (2\sqrt3 - 2, 2)$ and $\eta = \frac{3\gamma^2}{16} - \frac12$, 
conditioned on $\{s < \qduration^\delta\}$, let $\mathcal U$ be the set of quantum surfaces swallowed by time $s$. If we further condition on $\mathcal U$ and on $(L^\delta_\cdot)_{[0,s]}$, the conditional law of $\phi_s^\delta$ is $P_{\alpha, \beta, L_s^\delta}$.
\end{lemma}
\begin{proof}
    The proof is identical to that of Lemma~\ref{lem-duration-indep}, except at each step we use the last claim of Proposition~\ref{prop-radial-mot-swallow-finite} to get the conditional independence of $\phi_s$ and the quantum surfaces swallowed up until time $s$.
\end{proof}

Next, we give an explicit description of the ``unexplored'' regions of the growth process: in the dilute phase the complement of the QLE hull at the final time $\qduration^\delta$ is a quantum disk (Lemma~\ref{lem-simple-unexplored-qle}), and in the swallowing phase, given $(L^\delta_\cdot)_{[0,\qduration^\delta)}$ the quantum surfaces swallowed are conditionally independent quantum disks with boundary lengths equal to the sizes of the jumps of $(L^\delta_\cdot)_{[0,\qduration^\delta)}$ (Lemma~\ref{lem-swallowed-delta-qle}). 

\begin{lemma}\label{lem-simple-unexplored-qle}
In the setting of Definition~\ref{def-delta-QLE-finite} and Lemma~\ref{lem-duration-indep}  with $\gamma \in (\sqrt3 - 1, 2)$ and $\eta = \frac{3}{\gamma^2} - \frac12$, the conditional law of $(\D \backslash  K^\delta_\qduration, \phi^\delta)/{\sim_\gamma}$ given $\qduration^\delta$ is $\QD(1 + 2 \qduration^\delta)^\#$.
\end{lemma}

\begin{lemma}\label{lem-swallowed-delta-qle}
In the setting of Definition~\ref{def-delta-QLE-finite} and Lemma~\ref{lem-duration-indep}  with $\gamma \in (2\sqrt3 - 2, 2)$ and $\eta = \frac{3\gamma^2}{16} - \frac12$, the jumps of $(L^\delta_\cdot)_{[0,\qduration^\delta)}$ are in bijection with the quantum surfaces swallowed by $(\phi^\delta, (K^\delta_\cdot)_{[0,\qduration^\delta]})$, wherein at each jump time the process swallows a quantum surface whose $\gamma$-LQG boundary length is the jump size. Moreover, conditioned on $(L^\delta_{\cdot})_{[0,\qduration^\delta)}$, the swallowed quantum surfaces are conditionally independent, and the conditional law of the quantum surface corresponding to a jump of size $\ell$ is $\QD(\ell)^\#$.
\end{lemma}

The proofs of Lemmas~\ref{lem-simple-unexplored-qle} and~\ref{lem-swallowed-delta-qle} follow the same inductive argument as in the proof of the first claim of Lemma~\ref{lem-duration-indep}, using Lemmas~\ref{lem-simple-unexplored} and~\ref{lem-swallowed-sle} respectively as inputs. We omit the details.

Finally, we define the $\delta$-QLE capacity time parametrization  $t^\delta: [0,\infty) \to [0,\infty]$ as follows. 
\eqb \label{eq-t-delta}
t^\delta(s) = 
\left\{
	\begin{array}{ll}
		\log g'_{K_s^\delta}(0)  & \mbox{if } 0 \leq s < \qduration^\delta  \\
	   \infty & \mbox{if } s \geq \qduration^\delta
	\end{array}
\right.
\eqe
Here, as before, for a $\D$-hull $K$ the function $g_K: \D \backslash K \to \D$ is the conformal map with $g_K(0) = 0$ and $g_K'(0) > 0$.
Thus, $t^\delta|_{[0,\qduration^\delta)}$ is an increasing homeomorphism from $[0,\qduration^\delta)$ to $[0,\infty)$.  Unlike the other pieces of $\delta$-QLE data discussed, $t^\delta$ does not agree in law with the capacity time parametrization of the associated SLE exploration of LQG. We now show the law of $t^\delta$ is tight in the following sense, so we can take a subsequential limit in the next section.

\begin{lemma}\label{lem-tight-t}
    Fix $s>0$. As $N \to \infty$, the conditional probability $\P[t^\delta(s) > N \mid s < \qduration^\delta]$ converges to 0 at a rate that does not depend on $\delta$. 
    In particular, the conditional laws of $t^\delta(s)$ conditioned on $\{ s < \qduration^\delta\}$ are tight as $\delta \to 0$.
\end{lemma}
\begin{proof}
    The conditional law of $\phi^{\delta}_s$ given $\{ s < \qduration^{\delta}\}$ does not depend on $\delta$ (Lemma~\ref{lem-duration-indep}). Thus, for any $\eps > 0$, we can choose $m(\eps) > 0$ sufficiently small such that 
   $\P[E^\delta_\eps \mid s < \qduration^{\delta}] > 1 - \eps$ for all $\delta$, where $E^\delta_\eps =\{\cA^\gamma_{\phi_s^{\delta}}(B_{1/2}(0)) > m(\eps)\}$. 
   
   Next, let $f =  g_{K_s^\delta}^{-1}$. On the event $\{t^\delta(s) > N\}$ we have $f'(0) < e^{-N}$, so the Koebe distortion theorem gives $|f(z)| \leq |z|(1-|z|)^{-2} f'(0) < 2 e^{-N}$ whenever $|z| \leq \frac12$. We conclude $\{ t^\delta(s) > N \} \subset \{f(B_{1/2}(0)) \subset B_{2e^{-N}}(0)\}$. On this latter event, we have $\cA^\gamma_{\phi_s^\delta}(B_{1/2}(0)) = \cA^\gamma_{\phi^\delta}(f(B_{1/2}(0))) \leq \cA^\gamma_{\phi^\delta}(B_{2e^{-N}}(0))$. Therefore 
   \begin{align*}
       \P[t^{\delta}(s) > N \mid s < \qduration^{\delta}] &\leq \P[(E^\delta_\eps)^c \mid s < \qduration^{\delta}] + \P[\{t^{\delta}(s) > N\} \cap E^\delta_\eps \mid s < \qduration^{\delta}] \\
       &\leq \eps + \mathbb{P}[\{\cA^\gamma_{\phi_s^{\delta}}(B_{1/2}(0)) \leq \cA^\gamma_{\phi^{\delta}}(B_{2e^{-N}}(0)) \} \cap E^\delta_\eps \mid s < \qduration^{\delta}] \\
       &\leq \eps + \P[m (\eps) < \cA^\gamma_{\phi^{\delta}}(B_{2e^{-N}}(0)) \mid s < \qduration^{\delta}] \\
       &\leq \eps + \frac{\P[m (\eps) < \cA^\gamma_{\phi^{\delta}}(B_{2e^{-N}}(0))]}{\P[ s < \qduration^{\delta}]}.
   \end{align*}
   For fixed $\eps$, as $N \to \infty$ the numerator shrinks to zero at a rate uniform in $\delta$ (indeed, the law of $\phi^{\delta}$ does not depend on $\delta$), and the denominator does not depend on $\delta$ (Lemma~\ref{lem-duration-indep}). We conclude that as $N \to \infty$, we have $\P[t^{\delta}(s) > N \mid s < \qduration^{\delta}] \to 0$ at a rate uniform in $\delta$. 
\end{proof}

\subsection{$\delta$-QLE$(\gamma^2, \eta)$ as an $\eta$-DBM growth process on $\gamma$-LQG }\label{sec-scaling-rels} 
In this section we explain that heuristically, the $\delta$-QLE$(\gamma^2, \eta)$ process constructed in Definition~\ref{def-delta-QLE-finite} can be viewed as an $\eta$-DBM growth process on a $\gamma$-LQG surface at discretization scale $\delta$. 

Recall that $\eta$-DBM is a discrete growth model where, at each step, a particle attaches to the boundary at a point sampled from the boundary measure $(\frac{d\nu}{d\mu})^{\eta} \mu$, where $\mu$ and $\nu$ are the uniform (counting) and harmonic measures respectively. At the continuum level, for a $\gamma$-LQG surface, the analog of the uniform measure is the $\gamma$-LQG boundary length measure $\cL^\gamma_\phi$. Thus, we would like to define a $\delta$-discretized version of QLE$(\gamma^2, \eta)$ where, at each step, we sample a boundary point from ``$(\frac{d\nu}{d\cL^\gamma_\phi})^\eta \cL^\gamma_\phi$'' (where $\nu$ is harmonic measure) and grow an SLE curve at that point for quantum natural time $\delta$ (corresponding to attaching a $\delta$-sized particle). The issue is that $(\frac{d\nu}{d\cL^\gamma_\phi})^\eta \cL^\gamma_\phi$ does not make literal sense since $\nu$ and $\cL^\gamma_\phi$ are mutually singular. In this section, we will explain that, at least at the level of discrete approximations,
we can interpret 
\eqb\label{eq-eta-beta}
\text{``\quad $((\frac{d\nu}{d\cL^\gamma_\phi})^\eta \cL^\gamma_\phi)^\# = (\cL^\beta_\phi)^\#$\quad ''}\qquad \text{ if }\beta \text{ satisfies } 
 \eta =  \frac14 \beta^2 - \frac12 \beta Q + 1 \text{ and } \beta < Q.
\eqe
For $(\gamma, \eta)$ in~\eqref{eq-gamma-eta} this corresponds to $\beta = \beta(\gamma, \eta)$ as in~\eqref{eq-alpha-beta}. In this way, the $\delta$-QLE$(\gamma^2, \eta)$ process introduced in Definition~\ref{def-delta-QLE-finite} can be heuristically understood as a $\eta$-DBM process on $\gamma$-LQG. 

\begin{remark}
    Using the formal relation~\eqref{eq-eta-beta} we can write down a definition for a $\delta$-QLE$(\gamma^2, \eta)$ process for \emph{all} $(\gamma, \eta)$ for which there exists $\beta$ satisfying~\eqref{eq-eta-beta}, simply by changing the value of $\beta$ in Definition~\ref{def-delta-QLE-finite}. In this paper, we restrict our attention to the values of $(\gamma, \eta)$ in~\eqref{eq-gamma-eta} because the stationarity properties inherited from the LQG/SLE couplings (see Section~\ref{subsec-properties-delta-qle}) make them more tractable. 
\end{remark}

We now give a formal derivation of~\eqref{eq-eta-beta}.
Fix some large integer $n$ and divide $\partial \D$ into $n$ intervals of equal $\cL^\gamma_h$-length; let $I$ be one such interval. Denoting by $|I|$ the $\mathrm{Leb}_{\partial \D}$ length of $I$, up to multiplicative constants we have
\[\cL^\gamma_h(I) \approx |I|^{\gamma Q/2} e^{\frac\gamma2 h_{I}}, \quad \cL^\beta_h(I) \approx |I|^{\beta \hat Q / 2} e^{\frac\beta2 h_{I}}, \]
where $h_{I}$ denotes the average of $h$ on $I$ and $\hat Q = \frac\beta2 + \frac2\beta$. 
Combining these gives
\[ \cL^\beta_h(I) \approx |I|^{\frac\beta2(\hat Q - Q)} \cL^\gamma_h(I)^{\beta/\gamma} \propto \nu(I)^{\frac\beta2(\hat Q - Q)},\]
where in the last step we use that $\cL^\gamma_h(I)$ does not depend on $I$, and $|I|$ is $2\pi$-times the harmonic measure $\nu(I)$ of $I$ viewed from $0$. 

Sampling one of these intervals $I$ uniformly at random is a discrete analog of sampling a point from $(\cL^\gamma_h)^\#$, since each of the intervals has the same $\cL^\gamma_h$ length. 
By contrast, sampling one of these intervals $I$ with probability proportional to $\cL^\beta_h(I)$ is analogous to sampling from $(\cL^\beta_h)^\#$. Thus, at the level of discrete approximations, sampling from $(\cL^\beta_h)^\#$ can be viewed as sampling from $(\cL^\gamma_h)^\#$ and reweighting by the $\eta$-th power of the harmonic measure, where $\eta = \frac\beta2(\hat Q - Q)$;  substituting $\hat Q = \frac\beta2 + \frac2\beta$ yields~\eqref{eq-eta-beta}.

\section{Subsequential limits of $\delta$-QLE}\label{sec-qle-properties}
Fix a pair of parameters $(\gamma, \eta)$ in~\eqref{eq-gamma-eta}. We will construct a  subsequential limit $(\phi, (K_\cdot)_{[0,\qduration)})$ of $\delta$-QLE$(\gamma^2, \eta)$ as $\delta \to 0$ (Lemma~\ref{lem-coupling} and following text), and in the rest of this section establish properties of $(\phi, (K_\cdot)_{[0,\qduration)})$. 
In Section~\ref{sec-phases} we will define QLE$(\gamma^2, \eta)$ by passing to a further subsequence.

Consider the $\delta$-QLE$(\gamma^2, \eta)$ process $(\phi^\delta, (K^\delta_\cdot)_{[0,\qduration^\delta)})$ from 
Definition~\ref{def-delta-QLE-finite}. Recall the boundary length process $(L^\delta_\cdot)_{[0,\qduration^\delta)}$ from Lemma~\ref{lem-duration-indep} and capacity parametrization $t^\delta$ from~\eqref{eq-t-delta}. Set $L^\delta_{\qduration^\delta} = \lim_{s \to \qduration^\delta} L^\delta_s$; this limit exists by Remark~\ref{rem-bdy-reg} and the first claim of Lemma~\ref{lem-duration-indep}.
Define the function $\tilde L^\delta:[0,1] \to [0,\infty)$ by $\tilde L^\delta(x) = L^\delta_{\qduration^\delta x}$. Let $(\tilde K^\delta_t)_{t \geq 0}$ be the process $(K_s^\delta)_{0\leq s < \qduration^\delta}$ reparametrized by capacity, i.e., $\tilde K^\delta_t := K^\delta_{s^\delta(t)}$ where $s^\delta = (t^\delta)^{-1}$. Let $\nu^\delta$ be the driving measure encoding $(\tilde K^\delta_t)_{0 \leq t < \infty}$ as in Proposition~\ref{prop-loewner-bijection}.

The tuple $(\phi^\delta, \nu^\delta, \qduration^\delta, \cA^\gamma_{\phi_0^\delta}, \tilde L^\delta,  t^\delta)$ is a random variable taking values in the complete separable metric space $\mathcal{D} \times \mathcal{N}\times \R \times \mathcal{B} \times \mathcal J \times \mathcal I$, where $\cD, \mathcal N$, $\cB$, $\mathcal J$ and $\mathcal I$ are defined in Section~\ref{subsec-topologies}. Since $\cD, \mathcal N$, $\R$,  $\cB$, $\mathcal J$ and $\mathcal I$ are complete separable metric spaces, so is $\mathcal{D} \times \mathcal{N}\times \R \times \mathcal{B} \times \mathcal J \times \mathcal I$.

\begin{lemma}\label{lem-tight-sf}
The family of random variables $(\phi^\delta, \nu^\delta, \qduration^\delta, \cA^\gamma_{\phi^\delta}, \tilde L^\delta,  t^\delta)_{\delta > 0}$ is tight.
\end{lemma}
\begin{proof}
    It suffices to establish tightness of each of the marginals. By Lemma~\ref{lem-duration-indep}, the law of each of the random variables $\phi^\delta$, $\qduration^\delta$, $\cA^\gamma_{\phi^\delta}$ and $\wt L^\delta$ does not depend on $\delta$. Since these random variables take values in separable complete metric spaces, each of these marginals is tight as $\delta$ varies. Finally, since $\mathcal N$ and $\mathcal I$ are compact metric spaces, the families $(\nu^\delta)_{\delta > 0}$ and $(t^\delta)_{\delta>0}$ are also tight.
\end{proof}

\begin{lemma}\label{lem-coupling}
    There exists a sequence $(\delta_k)$ decreasing to 0 such that there is a coupling of the random variables
    $$(\phi^{\delta_k},  \nu^{\delta_k}, \qduration^{\delta_k}, \cA^\gamma_{\phi^{\delta_k}}, \tilde L^{\delta_k}, t^{\delta_k})_{k > 0}$$
    which converges almost surely as $k \to \infty$. Moreover, if $t = \lim_{k \to \infty} t^{\delta_k}$ and $\qduration = \lim_{k \to \infty} \qduration^{\delta_k}$, then $t(s) < \infty$ if and only if $s < \qduration$.
\end{lemma}
\begin{proof}
More strongly, for any sequence decreasing to $0$, we can find a subsequence $(\delta_k)$ such that the first claim holds. Indeed, denote the law of $(\phi^\delta, \nu^\delta, \qduration^\delta, \cA^\gamma_{\phi^\delta}, \tilde L^\delta, t^\delta)$ by $\mu_\delta$. Since each $\mu_\delta$ is a probability measure on the complete separable metric space $\mathcal{D} \times \mathcal{N}\times \R \times \mathcal{B} \times \mathcal J \times \mathcal I$, for any sequence decreasing to $0$ we can find a subsequence $(\delta_k)$ such that $\mu_{\delta_k}$ converges weakly as $k \to \infty$ by Prokhorov's theorem, and then the Skorokhod representation theorem gives the desired coupling. For the second claim, by definition $t(q) = \lim_{k \to \infty} t^{\delta_k}(q)$ for rational $q$. Thus, for rational $q_1 < \qduration$ we have $t(q_1) < \infty$ due to Lemma~\ref{lem-tight-t}, and for rational $q_2 > \qduration$ we have $t(q_2) = \infty$ since $q_2 > \qduration$ implies $q_2 > \qduration^{\delta_k}$ and hence $t^{\delta_k}(q_2) = \infty$ for sufficiently large $k$. The right-continuity of $t$ then yields the second claim.
\end{proof}

In the setting of Lemma~\ref{lem-coupling}, let 
\[(\phi, \qduration, \nu, \tilde L, t) := \lim_{k \to \infty}(\phi^{\delta_k}, \qduration^{\delta_k}, \nu^{\delta_k},  \tilde L^{\delta_k}, t^{\delta_k}),\]
and let $L_s = \tilde L_{s/\qduration}$ for $s \in [0,\qduration]$ (so in particular $L_\qduration = \lim_{s \to \qduration} L_s$). Let $(\tilde K_t, g_t)_{t >0}$ be the process of hulls and mapping-out functions corresponding to $\nu$ as in Proposition~\ref{prop-loewner-bijection} and define $K_s = \tilde K_{t(s)}$ for $0 \leq s < \qduration$. Note that we have not yet defined $K_\qduration$; this will be done in Section~\ref{sec-phases}.

The pair $(\phi, (K_\cdot)_{[0,\qduration)})$ induces the process of fields $\phi_s := g_{t(s)} \bullet_\gamma \phi$ for $0 \leq s < \qduration$ (so $\phi_s$ is obtained from $\phi$ by mapping out from $\D \backslash K_s$ to $\D$). 
We call $(L_\cdot)_{[0,\qduration]}$ the boundary length process of $(\phi, (K_\cdot)_{[0,\qduration)})$; we will see that it is a measurable function of $(\phi, (K_\cdot)_{[0,\qduration)})$ (Corollary~\ref{cor-length-process}) which describes the boundary length associated to $\phi_s$ (Proposition~\ref{prop-bd-length-consistent-reals}).

\begin{remark}
In Section~\ref{sec-phases} we will define QLE$(\gamma^2, \eta)$ to be the pair $(\phi, (K_\cdot)_{[0,\qduration]})$, where we extend the process to time $\qduration$, and take a further subsequence and apply the Skorokhod representation theorem to endow the process with desirable properties\footnote{We expect that these desirable properties should hold even without taking a further subsequence. This is related to the conjectural uniqueness of the subsequential limit (Question~\ref{question-unique}).}.
    Thus, all results obtained in this section also apply for $(\phi, (K_\cdot)_{[0,\qduration]})$ sampled from  QLE$(\gamma^2, \eta)$. 
\end{remark}

We now identify the laws of the field and the boundary length process of $(\phi, (K_\cdot)_{[0,\qduration)})$. 
\begin{lemma}\label{lem-qle-field-bdy-procress}
    The law of $\phi$ is $P_{\alpha, \beta, 1}$  where $(\alpha, \beta)$ is as in Definition~\ref{def-delta-QLE-finite} and $P_{\alpha, \beta, 1}$ is defined in Definition~\ref{def-P-weighted}. Also, $(L_\cdot)_{[0,\qduration)}$  agrees in law with the boundary length process of Proposition~\ref{prop-radial-mot-simple-finite} if $\gamma \in (\sqrt3 - 1,2)$ and $\eta = \frac3{\gamma^2} - \frac12$, Proposition~\ref{prop-radial-mot-swallow-finite} if $\gamma \in (2\sqrt3 - 2, 2)$ and $\eta = \frac{3\gamma^2}{16} - \frac12$, or  
    Proposition~\ref{prop-radial-mot-sf-finite} if $\gamma \in (0, 4/3)$ and $\eta = \frac{3\gamma^2}{16} - \frac12$ (where in all three cases, the initial boundary length is $\ell = 1$).
\end{lemma}
\begin{proof}
The first claim follows from the fact that for all $k$ the law of $\phi^{\delta_k}$ is $P_{\alpha, \beta, 1}$, and $\phi = \lim_{k \to \infty} \phi^{\delta_k}$. For the second claim, by Lemma~\ref{lem-duration-indep} the law of $(\tilde L^{\delta_k},  \qduration^{\delta_k})$ does not depend on $k$, hence $(\tilde L, \qduration) \stackrel d= (\tilde L^{\delta_k},  \qduration^{\delta_k}) $ for any $k$; we conclude that $(L_\cdot)_{[0,\qduration)} \stackrel d= (L_\cdot^{\delta_k})_{[0,\qduration^{\delta_k})}$ for any $k$. Lemma~\ref{lem-duration-indep} then identifies the law of  $(L_\cdot^{\delta_k})_{[0,\qduration^{\delta_k})}$ for each of the three cases, as needed. 
\end{proof}

We next establish two types of results for $(\phi, (K_\cdot)_{[0,\qduration)})$. First, the subsequential limits of certain observables of $\delta$-QLE$(\gamma^2, \eta)$ a.s.\ agree with the corresponding observables for $(\phi, (K_\cdot)_{[0,\qduration)})$. This is shown for the LQG area measure in Lemma~\ref{lem-area-measure-consistent} and the boundary length process in Corollary~\ref{cor-length-process}. Second, $(\phi, (K_\cdot)_{[0,\qduration)})$ enjoys a stationarity property (Proposition~\ref{prop-bd-length-consistent-reals}) analogous to that of $\delta$-QLE (Lemma~\ref{lem-duration-indep}). 

\begin{lemma}\label{lem-area-measure-consistent}
    Almost surely $\cA^\gamma_{\phi} = \lim_{k \to \infty} \cA^\gamma_{\phi^{\delta_k}}$ as elements of $\mathcal B$.
\end{lemma}
\begin{proof}
Let $\wt \cA := \lim_{k \to \infty} \cA^\gamma_{\phi^{\delta_k}}$. 
By definition the law of $\phi^{\delta_k}$ is $P_{\alpha, \beta, 1}$ for any $k$, so the law of  $(\phi^{\delta_k}, \cA^\gamma_{\phi^{\delta_k}})$ does not depend on $k$. Since $\lim_{k \to \infty}(\phi^{\delta_k}, \cA^\gamma_{\phi^{\delta_k}}) = (\phi, \wt \cA)$ a.s., we have $(\phi, \wt \cA) \stackrel d= (\phi^{\delta_k}, \cA^\gamma_{\phi^{\delta_k}})$ for any $k$. Since $\cA^\gamma_{\phi^{\delta_k}}$ is the $\gamma$-LQG area measure of $\phi^{\delta_k}$ a.s., we conclude that $\wt \cA$ is the $\gamma$-LQG area measure of $\phi$ a.s., completing the proof.
\end{proof}

Next, we give a technical lemma that we will use to show stationarity of $(\phi, (K_\cdot)_{[0,\qduration)})$. Namely, $\phi^{\delta_k}_{s_n}$ converges in a certain sense to $\phi_s$ where first $k \to \infty$ then $s_n \downarrow s$ along a sequence of rational times $s_n$; we work with rational times since  $t^{\delta_k}$ and $t$ lie in the metric space $\mathcal I$ of c\'adl\'ag nondecreasing functions, whose metric is defined in terms of pointwise convergence on the rationals (Section~\ref{subsec-topologies}). 

\begin{lemma}\label{lem-field-estimate}
Let $q>0$ be rational and let $f$ be a smooth compactly-supported function on $\D$. Then 
\[ \lim_{k \to \infty} (\phi_{q}^{\delta_k}, f) = (\phi_q, f)\quad \text{a.s.\ on the event }\{q < \qduration\}.\]
\end{lemma}
Note that the above limit makes sense since $q < \qduration$ implies $q < \qduration^{\delta_k}$ for all large $k$ so $\phi_{q}^{\delta_k}$ is defined.
\begin{proof}
Let $(g^{\delta_k}_t)_{t \geq 0}$ be the family of mapping out functions for $\nu^{\delta_k}$ and let $(g_t)_{t \geq 0}$ be the family of  mapping out functions for $\nu$ as in Proposition~\ref{prop-loewner-bijection}. 
To lighten notation, in this argument we write $g^k:= g^{\delta_k}_{t^{\delta_k}(q)}$ and $g:=g_{t(q)}$. 
    Since $\phi_q = g \bullet_\gamma \phi = \phi \circ g^{-1} + Q \log |(g^{-1})'|$, we have 
\eqb \nonumber
(\phi_q, f) = (\phi \circ g^{-1} + Q \log |(g^{-1})'|, f) = (\phi, |g'|^2 f\circ g) + Q (\log |(g^{-1})'|, f).
\eqe
where $|g'|^2 f\circ g$ is extended smoothly from $\D\backslash K_q$ to $\D$ by setting it equal to zero on $K_q$. Similarly,
\eqb \nonumber
(\phi_{q}^{\delta_k}, f) = (\phi^{\delta_k}, |(g^k)'|^2 f\circ g^k) + Q (\log |((g^k)^{-1})'|, f).
\eqe
Thus, to prove the claim, it suffices to show the almost sure limits
\begin{align} \label{eq-map-estimate}
    &\lim_{k \to \infty} (\log |((g^k)^{-1})'|, f) = (\log |(g^{-1})'|, f),\\
\label{eq-field-estimate}
&\lim_{k \to \infty} (\phi^{\delta_k}, |(g^k)'|^2 f\circ g^k) = (\phi, |g'|^2 f\circ g).
\end{align}
To prove these, we crucially need that $(g^k)^{-1} - g^{-1} \to 0$ uniformly on compact sets as $k \to \infty$; this follows from Proposition~\ref{prop-conv-loc-spacetime} since $\lim_{k \to \infty} \nu^{\delta_k} = \nu$ and $\lim_{k \to \infty} t^{\delta_k}(q) = t(q)$.
The first limit~\eqref{eq-map-estimate} is then immediate since the Cauchy integral formula gives $((g^k)^{-1})' - (g^{-1})' \to 0$ uniformly on compact sets.

We now turn to showing~\eqref{eq-field-estimate}. Fix some $\eps$ such that the support of $f\circ g$ lies in $(1-2\eps)\D$, then the support of $f \circ g^k$ lies in $(1-\eps) \D$ for all large $k$.  Let $\| \cdot \|_0$ denote the $L^2$-norm on $(1-\eps) \D$, and let $\|\cdot \|_a = \|(-\Delta)^{-a} \cdot\|_0$ where $a>0$ is the arbitrary constant in the definition of the function space $\cD$ in Section~\ref{subsec-topologies}. Since $(-\Delta)$ is self-adjoint, the Cauchy-Schwarz inequality yields $(f_1, f_2) = (((-\Delta)^{-a}f_1) ,((-\Delta)^a f_2)) \leq \| f_1\|_a \cdot \|(-\Delta)^a f_2\|_0$.  We thus have
\[
\begin{gathered}
|(\phi^{\delta_k}, |(g^k)'|^2 f \circ g^k) - (\phi, |g'|^2 f \circ g)| \leq |(\phi^{\delta_k}, |(g^k)'|^2 f \circ g^k - |g'|^2 f \circ g)| + |(\phi^{\delta_k} - \phi, |g'|^2 f\circ g)|  \\
\leq \| \phi^{\delta_k}\|_a \cdot \|(-\Delta)^a(|(g^k)'|^2 f \circ g^k - |g'|^2 f \circ g)\|_0 + \|\phi^{\delta_k} - \phi\|_a \cdot \|(-\Delta)^a (|g'|^2 f\circ g)\|_0.
\end{gathered}
\]
Since $\lim_{k \to \infty} \phi^{\delta_k} = \phi$ in $\cD$, we have $\lim_{k \to \infty} \|\phi^{\delta_k}\|_a = \|\phi\|_a<\infty$. Moreover, since the $g^k$ are holomorphic functions which converge uniformly on a neighborhood of $g^{-1} (\mathrm{supp}(f))$ to $g$, their derivatives converge uniformly on $g^{-1} (\mathrm{supp}(f))$ to the derivatives of $g$. As $f$ is smooth, we conclude that all partial derivatives of $|(g^k)'|^2 f \circ g^k$ converge uniformly on $\ol \D$ to those of $|g'|^2 f \circ g$. Thus $\lim_{k \to \infty} \|(-\Delta)^a (|(g^k)'|^2 f\circ g^k - |g'|^2 f \circ g)\|_0 = 0$. This implies the first term converges to zero. For the second term, $\lim_{k \to \infty} \phi^{\delta_k} = \phi$ in $\cD$ implies $\lim_{k \to \infty}\| \phi^{\delta_k} - \phi\|_a = 0$. Thus~\eqref{eq-field-estimate} holds, completing the proof. 
\end{proof}

We now employ Lemma~\ref{lem-field-estimate} to show that for each time $s$, the field at time $s$ and boundary length process up until time $s$ of $(\phi, (K_\cdot)_{[0,\qduration)})$ agree in law with those of $\delta$-QLE$(\gamma^2, \eta)$.

\begin{proposition}\label{prop-bd-length-consistent-reals}
For $s>0$ and $k \in \mathbb N$, the law of $(\phi_s, (L_\cdot)_{[0,s]})$ restricted to $\{ s < \qduration\}$ agrees with the law of $(\phi_s^{\delta_k}, (L_\cdot^{\delta_k})_{[0,s]})$ restricted to $\{ s < \qduration^{\delta_k}\}$. In particular, conditioned on $\{ s < \qduration\}$ and on $(L_\cdot)_{[0,s]}$, the conditional law of $\phi_s$ is $P_{\alpha, \beta, L_s}$ where $(\alpha, \beta)$ is as in Definition~\ref{def-delta-QLE-finite}.
\end{proposition}
\begin{proof}
By Lemma~\ref{lem-qle-field-bdy-procress} we have $(L_\cdot)_{[0,\qduration)} \stackrel d= (L^{\delta_k})_{[0,\qduration^{\delta_k})}$ for all $k$. It suffices to show that for any $s>0$, conditioned on $\{s < \qduration\}$ and on $(L_\cdot)_{[0,s]}$, the conditional law of $\phi_s$ is $P_{\alpha, \beta, L_s}$. Indeed, this gives the second claim, and hence the first claim by Lemma~\ref{lem-duration-indep}. We first prove this for rational $s$. 

Let $q > 0$ be rational and let $f$ be a smooth compactly-supported function on $\D$. By Lemma~\ref{lem-field-estimate}, 
\eqb\label{eq-phi-limit}
\lim_{k \to \infty}  (\phi_{q}^{\delta_k}, f) = (\phi_q, f)\quad \text{a.s.\ on the event }\{q < \qduration\}.
\eqe
Let $(f_{j})_{j\in \N}$ be a countable collection of test functions such that $(f_{j})_{j \in \N} \cap ((-\Delta)^a L^2((1 - 1/n)\D))$ is dense in $(-\Delta)^a L^2((1 - 1/n)\D)$ for all integers $n > 1$. Such a collection exists as $(-\Delta)^a L^2((1 - 1/n)\D)$ is separable for each $n$, and %the space of smooth compactly supported functions $H_s((1 - 1/n)\D)$ 
$C_c^\infty((1-1/n)\D)$
is dense in $((-\Delta)^a L^2((1 - 1/n)\D))$ (see Remark 2.8 of \cite{shef-gff}). %Indeed, endowed with the subspace metric $H_s((1 - 1/n)\D)$ is also separable, and admits a countable dense set $A$. As $A$ is dense in $H_s((1 - 1/n)\D)$, and $H_s((1 - 1/n)\D)$ is dense in $(-\Delta)^a L^2((1 - 1/n)\D)$, we have $A$ dense in $(-\Delta)^a L^2((1 - 1/n)\D)$ and the existence of the collection $(f_j)_{j \in \N}$ follows.

By Lemma~\ref{lem-duration-indep} the law of $(\phi_q^{\delta_k}, (L_\cdot^{\delta_k})_{[0,q]})$ restricted to $\{q < \qduration^{\delta_k}\}$ does not depend on $k$, hence the law of $(((\phi_q^{\delta_k}, f_j))_{j \in \N},  (L_\cdot^{\delta_k})_{[0,q]})$ restricted to $\{q < \qduration^{\delta_k}\}$ does not depend on $k$. Since $\lim_{k\to\infty} (L_\cdot^{\delta_k})_{[0,q]} =(L_\cdot)_{[0,q]}$ and by~\eqref{eq-phi-limit} $\lim_{k \to \infty} (\phi_q^{\delta_k}, f_j) = (\phi_q, f_j)$ for all $j$, we conclude that the law of $(((\phi_q, f_j))_{j \in \N},  (L_\cdot)_{[0,q]})$ restricted to $\{q < \qduration\}$ agrees with the law of $(((\phi_q^{\delta_k}, f_j))_{j \in \N},  (L_\cdot^{\delta_k})_{[0,q]})$ restricted to $\{q < \qduration^{\delta_k}\}$ for any $k$. 
Since any element of $\cD$ is determined by its pairings with $(f_{j})_{j\in \N}$, we conclude that the law of $(\phi_q, (L_\cdot)_{[0,q]})$ restricted to $\{q < \qduration\}$ agrees with the law of $(\phi_q^{\delta_k}, (L_\cdot^{\delta_k})_{[0,q]})$ restricted to $\{ q < \qduration^{\delta_k}\}$. Lemma~\ref{lem-duration-indep} thus implies that conditioned on $\{q < \qduration\}$ and on $(L_\cdot)_{[0,q]}$, the conditional law of $\phi_q$ is $P_{\alpha, \beta, L_q}$.

Now, consider arbitrary $s>0$, and let $s_n$ be a decreasing sequence of rationals with limit $s$. Let $E_s = \{ s < \qduration\}$ and $E_{s_n} = \{ s_n < \qduration\}$. On the event $E_s$ define $\psi_s = \phi_s - \frac2\gamma \log L_s$, and similarly, on $E_{s_n}$ define $\psi_{s_n} = \phi_{s_n} - \frac2\gamma \log L_{s_n}$. We claim that for any smooth compactly-supported function $f$ on $\D$, we have 
\eqb\label{eq-psi-limit}
\lim_{n \to \infty} (\psi_{s_n}, f) = (\psi_s, f) \quad \text{a.s.\ on }E_s.
\eqe
Indeed, on $E_s$ the right-continuity of $t$ yields $\lim_{n \to \infty} t(s_n) = t(s)$, hence if $(g_t)_{t \geq 0}$ is the family of  mapping out functions for $\nu$ as in Proposition~\ref{prop-loewner-bijection}, then $g_{t(s_n)}^{-1}$ converges uniformly to $g_{t(s)}^{-1}$ on compact subsets of $\D$; arguing exactly as in Lemma~\ref{lem-field-estimate} gives $\lim_{n \to \infty} (\phi_{s_n}, f) = (\phi_s, f)$. Since $\lim_{n \to \infty} L_{s_n} = L_s$ by the right-continuity of $L_\cdot$, we obtain~\eqref{eq-psi-limit}.

By the rational-time case just proved, conditioned on $E_{s_n}$ and $(L_\cdot)_{[0,s]}$, the conditional law of $\psi_{s_n}$ is $P_{\alpha, \beta, 1}$. 
Since $E_{s_n} \downarrow E_s$, and using~\eqref{eq-psi-limit}, we conclude that conditioned on $E_s$ and $(L_\cdot)_{[0,s]}$, the conditional law of $\psi_{s}$ is $P_{\alpha, \beta, 1}$. Thus, conditioned on $E_s$ and on $(L_\cdot)_{[0,s]}$, the conditional law of $\phi_s$ is $P_{\alpha, \beta, L_s}$ as needed. 
\end{proof}

As immediate consequences, the boundary length process of $(\phi, (K_\cdot)_{[0, \qduration)})$ is measurable with respect to $(\phi, (K_\cdot)_{[0, \qduration)})$, and the time-parametrization $t(s)$ is a.s.\ strictly increasing.  

\begin{corollary}\label{cor-length-process}
For each $s \geq 0$, on the event $\{s < \qduration\}$ we have $L_s = \cL^\gamma_{\phi_s}(\partial \D)$ a.s. In particular, $(L_\cdot)_{[0, \qduration]}$ can be equivalently defined using $(\phi, (K_\cdot)_{[0, \qduration)})$ by setting $L_s = \cL^\gamma_{\phi_s}(\partial \D)$ for rational $s \in [0,\qduration)$,  extending to all $s \in [0,\qduration)$ by right-continuity and setting $L_\qduration = \lim_{s\to \qduration} L_s$. 
\end{corollary}
\begin{proof}
    By Proposition~\ref{prop-bd-length-consistent-reals}, for each $s > 0$ the conditional law of $\phi_s$ given $\{s < \qduration\}$ and $L_s$ is $P_{\alpha, \beta, L_s}$, so $L_s = \cL^\gamma_{\phi_s}(\partial \D)$. 
    Thus we have $L_s = \cL^\gamma_{\phi_s}(\partial \D)$ for all rational $s \in[0, \qduration)$ a.s., and the extension to all $s \in [0,\qduration]$ follows from the right-continuity of $(L_\cdot)_{[0,\qduration]}$ and the definition of $L_\qduration$.
\end{proof}

\begin{corollary}\label{cor-t-increasing}
The function $t: [0,\infty) \to [0,\infty]$ is strictly increasing on $[0,\qduration]$. 
\end{corollary}
\begin{proof}
    If $t$ is constant on some subinterval of $[0,\qduration]$, then by Corollary~\ref{cor-length-process} the boundary length process $L$ is also constant on that subinterval. On the other hand, Lemma~\ref{lem-qle-field-bdy-procress} shows that the boundary length process is nonconstant on any interval of positive length. Thus $t$ is strictly increasing on $[0,\qduration]$.
\end{proof}

We conclude by noting that we have already proved Theorem~\ref{thm-stationarity}. 
\begin{proof}[{Proof of Theorem~\ref{thm-stationarity}}]
    In Section~\ref{sec-phases} we will define QLE$(\gamma^2, \eta)$ by taking a further subsequential limit in our definition of the process $(\phi, (K_\cdot)_{[0,\qduration)})$, so  Proposition~\ref{prop-bd-length-consistent-reals} also holds for QLE$(\gamma^2, \eta)$.
\end{proof}

\section{Construction and phases of quantum natural time QLE}
\label{sec-phases}
In this section we will prove Theorems~\ref{thm-phase},~\ref{thm-simple-markov} and~\ref{thm-swallowing-markov}. Let $(\gamma, \eta)$ be parameters coupled as in~\eqref{eq-gamma-eta}. We first recall the limiting objects constructed in Section~\ref{sec-qle-properties}.

By Lemma~\ref{lem-coupling}, we can pick some sequence $(\delta_k)$ decreasing to 0 for which we can couple the random variables    $$(\phi^{\delta_k},  \nu^{\delta_k}, \qduration^{\delta_k}, \cA^\gamma_{\phi^{\delta_k}}, \tilde L^{\delta_k}, t^{\delta_k})_{k > 0}$$
    to converge almost surely as $k \to \infty$. 
    Let $(\phi, \nu, \qduration, \tilde L, t) = \lim_{k \to \infty}(\phi^{\delta_k}, \nu^{\delta_k}, \qduration^{\delta_k}, \tilde L^{\delta_k}, t^{\delta_k})$, let $L_s = \tilde L_{s/\qduration}$ for $s \in [0,\qduration]$, and let $(\tilde K_t, g_t)_{t >0}$ be the process of hulls and mapping-out functions corresponding to $\nu$ as in Proposition~\ref{prop-loewner-bijection}. Define $K_s = \tilde K_{t(s)}$ for $0 \leq s < \qduration$. This gives $(\phi, (K_\cdot)_{[0,\qduration)})$ and its boundary length process $(L_\cdot)_{[0,\qduration]}$.

From each driving measure $\nu^{\delta_k}$ we define the corresponding process of hulls and mapping-out functions $(\tilde K_t^{\delta_k}, g_t^{\delta_k})_{t \geq 0}$ as in Proposition~\ref{prop-loewner-bijection}, and set $\phi^{\delta_k}_s = g^{\delta_k}_{t^{\delta_k}(s)} \bullet_\gamma \phi^{\delta_k}$ for each $s < \qduration^{\delta_k}$. Similarly, we write $\phi_s = g_{t(s)} \bullet_\gamma \phi$ for $s < \qduration$. 

In the following subsections, we will define $K_\qduration$ (possibly by passing to a further subsequence) and prove our main theorems. Each subsection will address a different phase. 

\subsection{Space-filling phase}

\begin{definition}\label{def-qle-sf}
    Let $\gamma \in (0,4/3)$ and $\eta = \frac{3\gamma^2}{16} - \frac12$, let $(\phi, (K_\cdot)_{[0,\qduration)})$ be defined as above, and let $K_\qduration = \ol{\bigcup_{s < \qduration} K_s}$. We call $(\phi, (K_\cdot)_{[0,\qduration]})$ a QLE$(\gamma^2, \eta)$ process. 
\end{definition}

We now show that in this regime QLE$(\gamma^2, \eta)$ is space-filling.

\begin{proof}[{Proof of Theorem~\ref{thm-phase} (space-filling phase)}]
We will prove that 
\eqb\label{eq-area-agreement}
\text{for rational }s, \text{ we have } \cA^\gamma_{\phi}(K_s) = s \text{ a.s.\ on }\{ s < \qduration\}.
\eqe
Given~\eqref{eq-area-agreement}, the monotonicity of $(K_s)_{[0,\qduration)}$ implies $\cA^\gamma_\phi(K_s) = s$ for all $s \in [0, \qduration)$ simultaneously a.s., and $\qduration = \cA^\gamma_\phi(\D) \geq \cA^\gamma_\phi(K_{\qduration}) \geq \lim_{s \uparrow \qduration} \cA^\gamma_{\phi}(K_s) = \qduration$ implies $\cA^\gamma_\phi(K_\qduration) = \qduration$ and $\cA^\gamma_\phi(\D \backslash K_\qduration) = 0$ a.s. Finally, since $\cA^\gamma_\phi$ a.s.\ assigns positive mass to nonempty open sets, $\cA^\gamma_\phi(\D \backslash K_\qduration) = 0$  implies $K_\qduration = \clD$, completing the proof.

We first prove that for rational $s$, on $\{ s < \qduration\}$ we have the a.s.\ lower bound   \begin{equation}\label{eq-area-lower}
     \cA^\gamma_{\phi}(K_s) \geq s.
\end{equation} 
Recall that $K_s = \tilde K_{t(s)}$ by definition. 
For any compact subset $C$ of $\D \backslash \tilde K_{t(s)}$, for sufficiently small $\eps>0$ we have $\mathrm{dist}(C, \partial (\D \backslash \tilde K_{t(s)}))>2\eps$ since $\D \backslash \tilde K_{t(s)}$ is open. Let $C_\eps$ be the set of points within distance $\eps$ of $C$. For sufficiently large $k$, we a.s.\ have 
\[\cA^\gamma_{\phi}(C) \leq \cA^\gamma_{\phi^{\delta_k}}(C_\eps) +\eps \leq \cA^\gamma_{\phi^{\delta_k}}(\D \backslash \tilde K^{\delta_k}_{t^{\delta_k}(s)}) +\eps = \cA^\gamma_{\phi^{\delta_k}}(\D) - s + \eps \leq \cA^\gamma_{\phi}(\D) - s + 2\eps.\]
Indeed, the first inequality follows from $\lim_{k \to \infty} \cA^\gamma_{\phi^{\delta_k}} = \cA^\gamma_{\phi}$ (Lemma~\ref{lem-area-measure-consistent}), the second from the inclusion $C_\eps \subset \D \backslash \tilde K^{\delta_k}_{t^{\delta_k}(s)}$ for large $k$ (since $\nu^{\delta_k} \to \nu$ and $t^{\delta_k}(s) \to t(s)$ and using Proposition~\ref{prop-conv-loc-spacetime}), and the last from $\lim_{k\to\infty}\cA^\gamma_{\phi^{\delta_k}}(\D)=\cA^\gamma_{\phi}(\D)$. Since $\eps>0$ was arbitrary, we conclude that $\cA^\gamma_{\phi}(C) \leq \cA^\gamma_{\phi}(\D) - s$ for all compact $C \subset \D \backslash \tilde K_{t(s)}$. The measure $\cA^\gamma_{\phi}$ is inner regular (more strongly, it is Radon~\cite{rhodes2013gaussianmultiplicativechaosapplications}), giving $\cA^\gamma_{\phi}(\D \setminus \tilde K_{t(s)}) \leq \cA^\gamma_{\phi}(\D) - s$ and hence~\eqref{eq-area-lower}.

    We will now bootstrap~\eqref{eq-area-lower} to obtain~\eqref{eq-area-agreement}.  Fix any $k > 0$. In our subsequent argument, we will not send $k \to \infty$ or use the coupling of $\delta_k$-QLE with QLE; we will only use that certain observables of $\delta_k$-QLE and QLE have the same law. Since by definition $\cA^\gamma_{\phi^{\delta_k}}(K^{\delta_k}_s) = s$, we have 
    \begin{align}\label{eq-area-decomposition-1}
        \cA^\gamma_{\phi}(K_{s}) + \cA^\gamma_{\phi}(\D \setminus K_s) &= \cA^\gamma_{\phi}(\D) \quad \text{ on the event } \{s < \qduration \} \\
        s +  \cA^\gamma_{\phi^{\delta_k}}(\D \setminus K_{s}^{\delta_k}) &= \cA^\gamma_{\phi^{\delta_k}}(\D) \quad \text{ on the event } \{s < \qduration^{\delta_k} \} \label{eq-area-decomposition-2}
    \end{align}
    Let $X,Y$ and $Z$ be the random variables in~\eqref{eq-area-decomposition-1} conditioned on $\{ s < \qduration\}$, and let $Y'$ and $Z'$ be the random variables in~\eqref{eq-area-decomposition-2} conditioned on $\{ s < \qduration^{\delta_k}\}$, so almost surely $X+Y=Z$ and $s + Y' = Z'$. 

    We have $Y = \cA^\gamma_{\phi}(\D \backslash K_{s}) = \cA^\gamma_{\phi_s}(\D)$ and $Y' = \cA^\gamma_{\phi^{\delta_k}}(\D \backslash K_{s}^{\delta_k}) = \cA^\gamma_{\phi_s^{\delta_k}}(\D)$.  
    By Proposition~\ref{prop-bd-length-consistent-reals} the law of $\cA^\gamma_{\phi_s}(\D)$ conditioned on $\{ s < \qduration\}$ agrees with the law of $\cA^\gamma_{\phi_s^{\delta_k}}(\D)$ conditioned on $\{ s < \qduration^{\delta_k}\}$, hence $Y \stackrel d= Y'$. 
    Next, Lemma~\ref{lem-area-measure-consistent} gives $\qduration = \lim_{k \to \infty} \qduration^{\delta_k} = \lim_{ k \to \infty} \cA^\gamma_{\phi^{\delta_k}}(\D) = \cA^\gamma_{\phi}(\D)$, and by definition $\qduration^{\delta_k} = \cA^\gamma_{\phi^{\delta_k}}(\D)$. Since the unconditioned laws of $\cA^\gamma_{\phi_0}(\D)$ and $\cA^\gamma_{\phi_0^{\delta_k}}(\D)$ agree (Lemma~\ref{lem-qle-field-bdy-procress}), we conclude that $Z \stackrel d= Z'$. 

    Finally, $Z \stackrel d= Z'$ implies $\E[e^{-Z}] = \E[e^{-Z'}]$, and hence $\E[e^{-X + s} \cdot e^{-Y}] = \E[e^{-Y'}]$. Since $Y \stackrel d= Y'$, this gives $\E[e^{-X + s} \cdot e^{-Y}] = \E[e^{-Y}]$. On the other hand,  we have $X \geq s$ a.s.\  by~\eqref{eq-area-lower}, so $e^{-X+s} \cdot e^{-Y} \leq e^{-Y}$ a.s.. We conclude that equality holds a.s., so $X = s$ a.s.; in other words,  $\cA^\gamma_{\phi}(K_s) = s$ a.s., so \eqref{eq-area-agreement} holds. 
\end{proof}

\subsection{Dilute phase}
We now prove that when $\gamma \in (\sqrt3 - 1,2)$ and $\eta = \frac{3}{\gamma^2} - \frac12$, QLE$(\gamma^2, \eta)$ is dilute in the sense that the growth process has zero LQG area. The corresponding statement for $\delta$-QLE$(\gamma^2, \eta)$ is immediate:

\begin{lemma}\label{lem-simple-delta-massless}
    For each $k$ we have $\cA^\gamma_{\phi^{\delta_k}}(K^{\delta_k}_{\qduration^{\delta_k}}) = 0$ a.s.
\end{lemma}
\begin{proof}
    Let $\kappa = \gamma^2$. Since an independent $\SLE_\kappa$ curve sampled on $\gamma$-LQG has zero $\gamma$-LQG area, and the $\delta_k$-QLE$(\gamma^2, \eta)$ process is defined by iteratively sampling $\SLE_\kappa$ curves, the result follows. 
\end{proof}

To set up for the limiting argument in the proof of Proposition~\ref{prop-simple-markov} below, we now introduce a probability measure on the function space $\cD$ (defined in Section~\ref{subsec-topologies}). For a quantum surface sampled from $\QD(\ell)^\#$, further sample a marked point from the probability measure proportional to its quantum area measure, and let $\QD'(\ell)^\#$ denote the law of the marked quantum surface\footnote{Note that in the literature a measure $\QD_{1,0}(\ell)^\#$ is also defined on the space of marked quantum surfaces; the relationship is that $\QD_{1,0}(\ell)^\#$ is the probability measure obtained by weighting $\QD'(\ell)$ by quantum area.}. Any sample from $\QD'(\ell)$ can be embedded as $(\D, \phi, 0)$ for some field $\phi$, and this field is only uniquely specified up to rotation. Let $P(\ell)$ be the law of $\phi$ when this rotation is chosen uniformly at random. 

Earlier, we defined $(\phi, (K_\cdot)_{[0,\qduration)})$ via a subsequential limit; we will now define the final state $K_\qduration$.
Since the space of compact subsets of $\ol \D$ equipped with the Hausdorff distance is a compact separable metric space, we can pass to a subsequence and recouple using the Skorokhod representation theorem so that $\lim_{k \to \infty} K^{\delta_k}_{\qduration^{\delta_k}}$ exists with respect to the Hausdorff distance. We define $K_\qduration = \lim_{k \to \infty} K^{\delta_k}_{\qduration^{\delta_k}}$.

\begin{proposition}\label{prop-simple-markov}
Passing to a further subsequential limit and recoupling using the Skorokhod representation theorem, we can guarantee the following. First, $K_s \subset K_\qduration$ for all $s < \qduration$ a.s. Second, $\cA^\gamma_\phi(K_\qduration) = 0$ a.s. Third, the conditional law of $(\D \backslash K_\qduration, \phi)/{\sim_\gamma}$ given $\qduration$ is $\QD(L_\qduration)^\#$.
\end{proposition}
\begin{proof}
\noindent
    For each $k$, sample a point $z^k \in \D$ from the probability measure $(\cA^\gamma_{\phi^{\delta_k}})^\#$ on $\D$; by Lemma~\ref{lem-simple-delta-massless} we have $z^k \in \D \backslash K^{\delta_k}_{\qduration^{\delta_k}}$ a.s. Let $f^k: \D \to \D \backslash K^{\delta_k}_{\qduration^{\delta_k}}$ be the conformal map such that $f^k(0) = z^k$ and $(f^k)'(0) > 0$. Independently sample $\theta$ uniformly from $[0,2\pi)$. 
    Passing to a subsequence and recoupling the laws via the Skorokhod representation theorem, we can ensure the limit $f = \lim_{k \to \infty} f^k$ a.s.\ exists in the topology of local uniform convergence. Since $\cA^\gamma_\phi$ has no atoms, we have $f'(0) \neq 0$. 
    \medskip 

    \noindent 
    \textbf{Step 1: The quantum surface parametrized by $f(\D)$ is a quantum disk.} 
   Let $R_\theta: \D \to \D$ be the rotation map $R_\theta(w) = e^{i\theta}w$. For each $k$ let $\psi^k = (R_{\theta} \circ (f^k)^{-1}) \bullet_\gamma \phi^{\delta_k}$ (so $(\D, \psi^k) \sim_\gamma (\D \backslash K^{\delta_k}_{\qduration^{\delta_k}}, \phi^{\delta_k})$), and let $\psi = (R_{\theta} \circ f^{-1}) \bullet_\gamma \phi$.

    By Lemma~\ref{lem-simple-unexplored-qle}, the conditional law of $(\D \backslash K^{\delta_k}_{\qduration^{\delta_k}}, \phi^{\delta_k})/{\sim_\gamma}$ given $\qduration^{\delta_k}$ is $\QD(1 + 2 \qduration^{\delta_k})^\#$, so since $\theta$ is independent of everything else, the conditional law of $\psi^k$ given $\qduration^{\delta_k}$ is $P(1 + 2 \qduration^{\delta_k})$ (defined immediately above this proof). 
    By Lemma~\ref{lem-duration-indep} the law of $\qduration^{\delta_k}$ does not depend on $k$, hence the joint law of $(\qduration^{\delta_k}, \psi^k)$ does not depend on $k$. Since $\qduration^{\delta_k} \to \qduration$ and $\psi^k \to \psi$ a.s., we conclude that $(\qduration, \psi) \stackrel d= (\qduration^{\delta_k}, \psi^k)$ for any $k$; in particular, the conditional law of $(\D, \psi)/{\sim_\gamma}$ given $\qduration$ is $\QD(1+2\qduration)^\#$. 

\medskip 
\noindent 
\textbf{Step 2: Showing $\cA^\gamma_\phi(\D \backslash f(\D)) = 0$ a.s.\ and $\cA^\gamma_\phi(K_s) = 0$ for rational $s < \qduration$ a.s.} For any fixed $k$, we have
    \[\cA^\gamma_{\phi^{\delta_k}}(\D) = \cA^\gamma_{\phi^{\delta_k}}(\D \backslash K^{\delta_k}_{\qduration^{\delta_k}}) = \cA^\gamma_{\psi^k}(\D) \stackrel d= \cA^\gamma_\psi(\D) = \cA^\gamma_{\phi}(f(\D)),\]
    where the first equality follows from Lemma~\ref{lem-simple-delta-massless}, the second from the definition of $\psi^k$, the distributional equality follows from $\psi^k \stackrel d= \psi$, and the last equality from the definition of $\psi$. Since $\phi^{\delta_k} \stackrel d= \phi$, we conclude that $\cA^\gamma_\phi(\D) \stackrel d= \cA^\gamma_\phi(f(\D))$, hence $\cA^\gamma_\phi(\D \backslash f(\D)) = 0$ a.s. By an analogous argument with Lemma~\ref{lem-simple-unexplored-qle} replaced by Lemma~\ref{lem-duration-indep}, for fixed $s$ the conditional laws of $\cA_\phi^\gamma(\D)$ and $\cA_\phi^\gamma (\D \backslash K_s)$ given $\{ s < \qduration\}$ agree, so the second claim follows. 

\medskip 
\noindent \textbf{Step 3: Showing $\D \backslash K_\qduration = f(\D)$.}
    %Now, note that a.s.\ the pointed domain $(\D \setminus f(\D), z)$ is the Carath\'eodory limit of $(K_{\qduration^k}^{\delta_k}, z)$. Indeed, 
    As $f^k \to f$ locally uniformly, %the point $z = f(0)$ satisfies $z \in f^k(\D)$ for all sufficiently large $k$, and so 
    the pointed domains $(f^k(\D), z)$ converge in the Carath\'eodory topology to $(f(\D), z)$. Since $f^k(\D) = \D \backslash K_{\qduration^{\delta_k}}^{\delta_k}$ and $\lim_{k \to \infty} K_{\qduration^{\delta_k}}^{\delta_k} = K_\qduration$ with respect to the Hausdorff topology, by Lemma~\ref{lem-hausdorff-cara-conv} the connected component of $\D \backslash K_\qduration$ containing $z$ is $f(\D)$. 
    If $\D \backslash K_\qduration$ contained a connected component other than $f(\D)$, then $\D \backslash f(\D)$ would have nonempty interior, contradicting Step 2 since $\cA^\gamma_\phi$ a.s.\ assigns positive mass to open sets. We conclude that $\D \backslash K_\qduration = f(\D)$.

    Combining Steps 2 and 3 gives the second claim. By Step 3, we have $(\D \backslash K_\qduration, \phi)/{\sim_\gamma} = (f(\D), \phi)/{\sim_\gamma} = (\D, \psi)/{\sim_\gamma}$, and so by Step 1 the third claim holds.

  Finally, we turn to the first claim. By monotonicity it suffices to show that $K_s \subset K_\qduration$ for rational $s < \qduration$. Since $f^k \to f$ uniformly on compact sets and $f(\D) = \D \backslash K_\qduration$, for any point $w \in \D \backslash K_\qduration$ there exists a neighborhood $U \subset \D \backslash K_\qduration$ of $w$ such that $U \subset \D \setminus K_{\qduration^{\delta_k}}^{\delta_k}$ for all $k$ sufficiently large; we conclude that $U \subset \D \backslash \tilde K^{\delta_k}_{t^{\delta_k}(s)}$ for sufficiently large $k$.
    Since $\lim_{k\to\infty} t^{\delta_k}(s) = t(s)$, by Proposition~\ref{prop-conv-loc-spacetime} the Carath\'eodory limit of $\tilde K^{\delta_k}_{t^{\delta_k}(s)}$ is $\tilde K_{t(s)}$. Rewriting $\tilde K_{t(s)}$ as $K_s$, the Carath\'eodory kernel theorem yields that either $U \subset \D \backslash K_s$ or $U \subset K_s$; the latter is impossible since $\cA^\gamma_\phi(K_s) = 0$ by Step 2 but $\cA^\gamma_\phi(U) > 0$, hence $U \subset \D \backslash K_s$. Thus, any point $w$ in $\D \backslash K_\qduration$ is also in $\D \backslash K_s$, i.e., $K_s \subset K_\qduration$. This concludes the proof of the first claim. 
\end{proof}

We now define QLE$(\gamma^2, \eta)$ for this parameter range; the relevant theorems are then immediate. 

\begin{definition}\label{def-qle-simple}
Let $\gamma \in (\sqrt3 - 1, 2)$ and $\eta = \frac{3}{\gamma^2} - \frac12$, and let $(\phi, (K_\cdot)_{[0,\qduration]})$ be the process in Proposition~\ref{prop-simple-markov}. We call $(\phi, (K_\cdot)_{[0,\qduration]})$ a QLE$(\gamma^2, \eta)$ process. 
\end{definition}

\begin{proof}[{Proof of Theorem~\ref{thm-simple-markov}}]
    The claim is immediate from Definition~\ref{def-qle-simple} and Proposition~\ref{prop-simple-markov}.
\end{proof}

\begin{proof}[{Proof of Theorem~\ref{thm-phase} (dilute phase)}]
The claim is immediate from Definition~\ref{def-qle-simple} and Proposition~\ref{prop-simple-markov}.
\end{proof}

\subsection{Swallowing phase}
Fix $\gamma \in (2\sqrt3 - 2,2)$ and $\eta = \frac{3\gamma^2}{16} - \frac12$. In Proposition~\ref{prop-swallowed} below, by passing to a further subsequence we will be able to understand the surfaces cut out by $(\phi, (K_\cdot)_{[0,\qduration)})$. The arguments differ from those for quantum natural time QLE$(8/3,0)$ in  \cite[Section 6]{ms-equivalence} because we set up the subsequential limit differently. Indeed,  \cite{ms-equivalence} first takes a subsequential limit for which they identify the laws of all quantum surfaces swallowed, then defines the time parametrization by counting swallowed surfaces in a way analogous to Definition~\ref{def-quantum-time-swallowing}. We instead choose the subsequence so the capacity time parametrizations $t^{\delta_k}$ converge, and the limit $t$ is used to define the quantum natural time parametrization; our approach has the advantage of applying to the dilute and space-filling ranges, for which there is not a clear way to implement the method of \cite{ms-equivalence}.  

We now state a uniform embedding result. Recall the function space $\cD$ from Section~\ref{subsec-topologies}, recall the definitions of $\QD'(\ell)$ and $P(\ell)$ given above Proposition~\ref{prop-simple-markov}, and for $\alpha \in \R$ let  $R_\alpha: \D \to \D$ be the rotation map $R_\alpha(z) = e^{i\alpha} z$. 
\begin{lemma}\label{lem-rotate-all}
    Fix a decreasing sequence of positive reals $\ell_1 > \ell_2 > \cdots$. Suppose $\psi_1, \psi_2, \dots \in \cD$ are random fields such that the quantum surfaces $(\D, \psi_i, 0)/{\sim_\gamma}$ are independent and each have law $\QD(\ell_i)^\#$. Independently of $(\psi_i)_{i=1}^\infty$, sample $\theta_1, \theta_2, \dots$ independently and uniformly from $[0,2\pi)$. Then the fields $(R_{\theta_i} \bullet_\gamma \psi_i)_{i=1}^\infty$ are independent, and for each $i$ the law of $R_{\theta_i} \bullet_\gamma \psi_i$ is $P(\ell_i)$.
\end{lemma}
\begin{proof}
    Let $(\phi_i)_{i=1}^\infty$ be a sequence of fields whose joint law is $\prod_{i=1}^\infty P(\ell_i)$. 
    By definition, we can couple $(\psi_i)_{i=1}^\infty$ with $(\phi_i)_{i=1}^\infty$ such that for each $i$ we have $\psi_i = R_{\alpha_i} \bullet_\gamma \phi_i$ for some random angle $\alpha_i$. We may take $(\theta_i)_{i=1}^\infty$ independent of $(\psi_i, \phi_i, \alpha_i)_{i=1}^\infty$, so $R_{\theta_i} \bullet_\gamma \psi_i  = R_{\theta_i} \bullet_\gamma (R_{\alpha_i} \bullet_\gamma \phi_i) = R_{\theta_i + \alpha_i} \bullet_\gamma \phi_i$ for all $i$. Although $(\alpha_i)_{i=1}^\infty$ is not necessarily independent of $(\phi_i)_{i=1}^\infty$, since the angles $\theta_i$ are uniformly random, the angles $(\theta_i + \alpha_i)_{i=1}^\infty$ (viewed modulo $2\pi$) are uniform and independent, and $(\theta_i + \alpha_i)_{i=1}^\infty$ is independent of $(\phi_i)_{i=1}^\infty$. Since $P(\ell_i)$ is preserved under rotation with any deterministic angle, the law of $(R_{\theta_i + \alpha_i} \bullet_\gamma \phi_i)_{i=1}^\infty$ is $\prod_{i=1}^\infty P(\ell_i)$; we are done since $(R_{\theta_i} \bullet_\gamma \psi_i)_{i=1}^\infty = (R_{\theta_i + \alpha_i} \bullet_\gamma \phi_i)_{i=1}^\infty$.
\end{proof}

\begin{proposition}\label{prop-swallowed}
Passing to a further subsequential limit and recoupling using the Skorokhod representation theorem, we can guarantee the following. The jumps of $(L_\cdot)_{[0,\qduration]}$ are in bijection with the quantum surfaces swallowed by $(\phi, (K_\cdot)_{[0,\qduration)})$, wherein at each jump time the process $(\phi, (K_\cdot)_{[0,\qduration)})$ swallows a quantum surface whose $\gamma$-LQG boundary length is the jump size. Moreover, conditioned on $(L_{\cdot})_{[0,\qduration]}$, the swallowed quantum surfaces are conditionally independent, and the conditional law of the quantum surface corresponding to a jump of size $\ell$ is $\QD(\ell)^\#$. Finally, for fixed $s>0$, conditioned on $\{s < \qduration\}$ the conditional law of $\phi_s$ given $(L_\cdot)_{[0,s]}$ and the quantum surfaces swallowed up to time $s$ is $P_{\alpha, \beta, L_s}$.
\end{proposition}
\begin{proof}

\medskip \noindent
\textbf{Step 1: Coupling the boundary length processes and the swallowed regions.} 
Let $\eps > 0$. 
Let $J_\eps$ be the collection of pairs $(s,x)$ such that $s < \qduration$ and $\lim_{u \uparrow s} L_u - L_s = x > \eps$, so $J_\eps$ describes the jumps of $L$ having size greater than $\eps$. Define $J_\eps^k$ in the same way using $L^{\delta_k}$. Note that $J_\eps$ and $J_\eps^k$ are finite because c\'adl\'ag functions on compact time intervals cannot have infinitely many jumps of size greater than $\eps$. 
%; note that $|J_\eps^k|<\infty$ because $J_\eps^k$ is in bijection with the swallowed quantum surfaces with boundary length greater than $\eps$ (Lemma~\ref{lem-swallowed-delta-qle}), and the infinitude of this collection would contradict that $\cA^\gamma_\phi(\D) < \infty$. 
Since $\lim_{ k \to \infty} \tilde L^{\delta_k} = \tilde L$ in $\mathcal J$ (i.e., in the Skorokhod J$_1$ topology; see Section~\ref{subsec-topologies}) and $\lim_{k \to \infty} \qduration^{\delta_k} = \qduration$, we conclude that for each $(s, x) \in J_\eps$, for sufficiently large $k$ there exists a pair $(s^k, x^k) \in J_\eps^k$ such that $\lim_{k \to \infty} (s^k, x^k) = (s,x)$. In particular $|J_\eps^k| \geq |J_\eps|$ for all large $k$; since $L^{\delta_k} \stackrel d= L$ (Lemma~\ref{lem-qle-field-bdy-procress}) we have $|J_\eps^k| \stackrel d=  |J_\eps|$ for all $k$, so  almost surely $|J_\eps^k| = |J_\eps|$ for sufficiently large $k$, and $J_\eps^k \to J_\eps$ elementwise a.s. 

Let $\mathcal U^k_\eps$ be the set of quantum surfaces swallowed by $(\phi^{\delta_k}, (K^{\delta_k}_\cdot)_{[0,\qduration^{\delta_k}]})$ having $\gamma$-LQG boundary lengths greater than $\eps$. Lemma~\ref{lem-swallowed-delta-qle} yields that $\mathcal U^k_\eps$ is in bijection with $J_\eps^k$, and conditioned on $J_\eps^k$ the elements of $\mathcal U^k_\eps$ are conditionally independent quantum disks having boundary lengths specified by $J_\eps^k$. For the $j$th element of $J_\eps^k$, let $U^k_{\eps,j} \subset \D$ be the region parametrizing the corresponding swallowed quantum surface, and independently sample $z^k_{\eps,j} \sim (\cA^\gamma_{\phi^{\delta_k}}|_{U^k_{\eps,j}})^\#$. Let $f^k_{\eps,j} : \D \to U^k_{\eps,j}$ be the conformal map such that $f^k_{\eps,j}(0) = z^k_{\eps,j}$ and $(f^k_{\eps,j})'(0) > 0$. By passing to a subsequence and recoupling the laws via the Skorokhod representation theorem as needed, we can ensure that the limit $f_{\eps,j} := \lim_{k \to \infty} f^k_{\eps,j}$ a.s.\ exists in the topology of local uniform convergence. Since $\cA^\gamma_\phi$ has no atoms a.s., we have $f_{\eps,j}'(0) \neq 0$ for all $j$. 
Let the $j$th element of $J_\eps$ be $(s_{\eps,j}, x_{\eps,j})$ (ordered by first coordinate), and let $U_{s_{\eps,j}} = f_{\eps,j} (\D)$. Note that the domains $(U_{s_{\eps, j}})_{ j \leq |J_\eps|}$ are disjoint, since for each $k$ the domains $U^k_{\eps, j}$ are disjoint.

We claim that for some $s_{\eps,j}' \leq s_{\eps, j}$ we have $U_{s_{\eps,j}} \subset K_{s_{\eps,j}'} \backslash K_{s_{\eps,j}'-}$, where we write $K_{s-} := \bigcup_{u < s} K_u$. For each $k$, let $(s^k_{\eps,j}, x^k_{\eps,j})$ be the $j$th element of $J_\eps^k$, so $\lim_{k \to \infty} (s^k_{\eps,j}, x^k_{\eps,j}) = (s_{\eps,j}, x_{\eps,j})$ a.s. Recall that for any rational $q$ we have $\lim_{k \to \infty} t^{\delta_k}(q) = t(q)$, so $K^{\delta_k}_q \to K_q$ in the Carath\'eodory topology by Proposition~\ref{prop-conv-loc-spacetime}. For any rational $q > s_{\eps,j}$, we have $q > s^k_{\eps,j}$ for sufficiently large $k$, hence $f^k_{\eps,j}(\D) \subset K^{\delta_k}_{q}$. Since $K^{\delta_k}_q \to K_q$ in the Carath\'eodory topology, we conclude that $U_{s_{\eps,j}} \subset K_q$. Now, for rational $q < s_{\eps,j}$, we have $q < s^k_{\eps,j}$ for sufficiently large $k$, hence $f^k_{\eps,j}(\D) \subset \D\backslash K^{\delta_k}_q$. Since $K^{\delta_k}_q \to K_q$ in the Carath\'eodory topology, by the Carath\'eodory kernel theorem we have either $U_{s_{\eps,j}} \subset K_q$ or $U_{s_{\eps,j}} \subset \D \backslash K_q$. We conclude that there exists a time $s_{\eps,j}' \leq s_{\eps,j}$ such that $U_{s_{\eps,j}} \subset K_{s_{\eps,j}'} \backslash K_{s_{\eps,j}'-}$.

\medskip \noindent
\textbf{Step 2: Identifying the quantum surfaces $(U_s, \phi)/{\sim_\gamma}$.} 
Independently of everything else, sample angles $(\theta_j)_{j=1}^\infty$ independently and uniformly from $[0,2\pi)$. Let $\psi_{\eps,j}^k = (R_{\theta_j} \circ (f_{\eps,j}^k)^{-1}) \bullet_\gamma \phi^{\delta_k}$ for $j \leq |J_\eps^k|$, and let $\psi_{\eps,j} = (R_{\theta_j} \circ f_{\eps,j}^{-1}) \bullet_\gamma \phi$ for $j  \leq |J_\eps|$. By Lemmas~\ref{lem-swallowed-delta-qle} and~\ref{lem-rotate-all}, conditioned on $L^{\delta_k}$, the conditional law of $(\psi_{\eps,j}^k)_{j\leq|J_\eps^k|}$ is $\prod_{j=1}^{|J_\eps^k|}P(x_{\eps,j}^k)$. Since the law of $L^{\delta_k}$ does not depend on $k$, neither does the law of $(L^{\delta_k}, (\psi_{\eps,j}^k)_{j=1}^{|J_\eps^k|})$. Since $(L^{\delta_k}, (\psi_{\eps,j}^k)_{j\leq|J_\eps^k|})$ a.s.\ converges to $(L, (\psi_{\eps,j})_{j\leq|J_\eps|})$, we conclude that $(L^{\delta_k}, (\psi_{\eps,j}^k)_{j\leq|J_\eps^k|}) \stackrel d= (L, (\psi_{\eps,j})_{j\leq|J_\eps|})$. Thus, if $\cD_\eps$ is the collection of disjoint domains $U_{s_{\eps, j}}$, then $\cD_\eps$ is in bijection with the jumps $J_\eps$ of $L$ having size greater than $\eps$; each jump $(s,x) \in J_\eps$ corresponds to a domain $U_s$ such that $U_s \subset K_{s'} \backslash K_{s'-}$ for some $s' \leq s$, and conditioned on $L$, the quantum surfaces $(U_s, \phi)/{\sim_\gamma}$ are conditionally independent quantum disks with boundary lengths given by the sizes of the corresponding jumps of $L$.

By passing to a subsequence and recoupling the laws as needed, we conclude that the above statement holds also when $\eps$ is replaced by zero: There is a collection $\cD$ of disjoint domains in bijection with the jumps $J$ of $L$; each jump $(s, x) \in J$ corresponds to a domain $U_s$ such that $U_s \subset K_{s'} \backslash K_{s'-}$ for some $s' \leq s$, and conditioned on $L$, the quantum surfaces $(U_s, \phi)/{\sim_\gamma}$ are conditionally independent quantum disks with boundary lengths given by the sizes of the corresponding jumps of $L$.

\medskip \noindent \textbf{Step 3: All $\gamma$-LQG area is contained in $\bigcup_{\cD} U_s$.} We claim that 
a.s.
\eqb\label{eq-swallowing-trace-zero}
\cA^\gamma_\phi(\D) = \sum_{U_s \in \cD} \cA^\gamma_\phi(U_s).
\eqe
To see that~\eqref{eq-swallowing-trace-zero} holds, by Lemma~\ref{lem-qle-field-bdy-procress} we have $L \stackrel d= L^{\delta_k}$ for any $k$, so by Lemma~\ref{lem-swallowed-delta-qle} $\sum_{\mathcal D} \cA^\gamma_\phi(U_s)$ agrees in law with the total LQG area of the regions cut out by SLE segments in $\delta$-QLE$(\gamma^2, \eta)$. This equals $\cA^\gamma_{\phi^{\delta_k}}(\D)$ because the SLE segments have zero LQG area. We conclude that $\cA^\gamma_\phi(\D) \stackrel d= \sum_{\mathcal D} \cA^\gamma_\phi(U_s)$. On the other hand %, since $U_s \subset K_s \backslash K_{s-}$ for each $s$,
the domains $U_s$ are disjoint, so $\cA^\gamma_\phi(\D) \geq \sum_{\mathcal D} \cA^\gamma_\phi(U_s)$ a.s. These two sentences imply~\eqref{eq-swallowing-trace-zero}. 

\medskip \noindent
\textbf{Step 4: For rational $q$, the conditional law of $\phi_{q}$ given $\{q < \qduration\}$, $(L_\cdot)_{[0,q]}$, and $\{(U_s, \phi)/{\sim_\gamma} \: : \: s \leq q\}$ is $P_{\alpha, \beta, L_q}$.} By Lemmas~\ref{lem-swallowing-markov} and~\ref{lem-swallowed-delta-qle}, the joint law of $\phi_{q}^{\delta_k}, (L^{\delta_k}_\cdot)_{[0,q]}$ and the quantum surfaces swallowed by time $q$ conditioned on $\{q < \qduration^{\delta_k}\}$ does not depend on $k$. By the construction in Step 1, and the Carath\'eodory convergence of $K^{\delta_k}_q$ to $K_q$, % (via Proposition~\ref{prop-conv-loc-spacetime} and $\lim_{k \to \infty} t^{\delta_k}(q) = t(q)$), 
these converge respectively to the joint law of $\phi_q$, $(L_\cdot)_{[0,q]}$, and $\{(U_s, \phi)/{\sim_\gamma} \: : \: s \leq q\}$ conditioned on $\{q < \qduration\}$, and hence this conditional law agrees with the conditional law for each $k$. Thus, by Lemma~\ref{lem-swallowing-markov} the conditional law of $\phi_{q}$ given $\{q < \qduration\}$, $(L_\cdot)_{[0,q]}$, and $\{(U_s, \phi)/{\sim_\gamma} \: : \: s \leq q\}$ is $P_{\alpha, \beta, L_q}$.

\medskip \noindent \textbf{Step 5: The quantum surface $(U_s, \phi)/{\sim_\gamma}$ is swallowed at time $s$.} 
Let $q$ be rational. On the event $\{ q < \qduration\}$, by~\eqref{eq-swallowing-trace-zero} and Step 3 we have 
\eqb\label{eq-swallowed-too-much}
\E[e^{-\cA^\gamma_{\phi}(\D)} \mid q < \qduration] = \E[ \exp( - \sum_{\substack{(s, x) \in J \\ s' \leq q }} \cA^\gamma_\phi(U_s) - \cA^\gamma_\phi(\D \backslash K_q) ) \mid q < \qduration]
\eqe
where, as before, $s'\leq s$ denotes the time that the region $U_s$ is swallowed. 
Now, let $J^k$ be the set of jumps of $L^{\delta_k}$, and let $\cD^k$ be the set of regions swallowed by the $\delta_k$-QLE$(\gamma^2, \eta)$ process, so $\cD^k$ is in bijection with $J^k$: for each $(s, x) \in J^k$ there is a domain $U_s^k \in \cD^k$ swallowed at time $s$. Then 
\[\E[e^{-\cA^\gamma_{\phi^{\delta_k}}(\D)} \mid q < \qduration^{\delta_k}] = \E[ \exp( - \sum_{\substack{(s, x) \in J^k \\ s \leq q }} \cA^\gamma_{\phi^{\delta_k}}(U_s^k) - \cA^\gamma_{\phi^{\delta_k}}(\D \backslash K^{\delta_k}_q) ) \mid q < \qduration^{\delta_k}].\]
By Step 4, this implies 
\eqb\label{eq-swallowed-right-amount}
\E[e^{-\cA^\gamma_{\phi}(\D)} \mid q < \qduration] = \E[ \exp( - \sum_{\substack{(s, x) \in J \\ s \leq q }} \cA^\gamma_\phi(U_s) - \cA^\gamma_\phi(\D \backslash K_q) ) \mid q < \qduration].
\eqe
The right hand sides of~\eqref{eq-swallowed-too-much} and~\eqref{eq-swallowed-right-amount} agree since their left hand sides agree, but 
\[\sum_{\substack{(s, x) \in J \\ s' \leq q }} \cA^\gamma_\phi(U_s) \geq \sum_{\substack{(s, x) \in J \\ s \leq q }} \cA^\gamma_\phi(U_s)\] since for each jump $(s,x) \in J$ we have $s' \leq s$. We conclude that the random sums agree a.s. In other words, with probability 1 there is no $(s,x) \in J$ such that  $s' \leq q < s$. Since this holds for all rational $q$, we conclude that a.s.\ for all $(s,x) \in J$ we have $s' = s$; that is, $U_s \subset K_s \backslash K_{s-}$ for all $(s,x) \in J$. 

Finally, by~\eqref{eq-swallowing-trace-zero} the interior of $\D \backslash \bigcup_\cD U_s$ has $\cA^\gamma_\phi$-measure zero, hence is empty. Therefore $U_s$ is the interior of $K_s \backslash K_{s-}$. Thus, the quantum surface $(U_s, \phi)/{\sim_\gamma}$ is swallowed at time $s$.

\medskip \noindent \textbf{Conclusion.} 
By Step 3 the interior of $\D \backslash \bigcup_\cD U_s$ is empty (since $\cA^\gamma_\phi$ assigns positive mass to open sets), so the collection of swallowed quantum surfaces is precisely $\{(U_s, \phi)/{\sim_\gamma} \: : \: U_s \in \cD\}$. The first two claims are now immediate consequences of Steps 2 and 5. The final claim for rational times $q$ follows from Steps 4 and 5, and is bootstrapped from rational times to real times by the  argument from the last paragraph of the proof of Proposition~\ref{prop-bd-length-consistent-reals}. 
\end{proof}

We now define QLE$(\gamma^2, \eta)$ for the swallowing parameter range, and conclude. 

\begin{definition}\label{def-qle-swallowing}
Let $\gamma \in (2\sqrt3-2,2)$ and $\eta = \frac{3\gamma^2}{16} - \frac12$, let $(\phi, (K_\cdot)_{[0,\qduration)})$ be as in Proposition~\ref{prop-swallowed}, and let $K_\qduration = \ol{\bigcup_{s < \qduration}K_s}$. We call $(\phi, (K_\cdot)_{[0,\qduration]})$ a QLE$(\gamma^2, \eta)$ process. 
\end{definition}

\begin{proof}[{Proof of Theorem~\ref{thm-swallowing-markov}}]
The claim follows from Definition~\ref{def-qle-swallowing} and Proposition~\ref{prop-swallowed}.
\end{proof}

\begin{proof}[{Proof of Theorem~\ref{thm-phase} (swallowing phase)}]
Consider the collection $J$ of pairs $(s,x)$ describing the jumps of $(L_\cdot)_{[0,\qduration]}$, so $(s, x) \in J$ is equivalent to $L_{s-} - L_s  = x> 0$. For each such pair let $U_s = \mathrm{int}(K_s \backslash K_{s-})$. 
 By Proposition~\ref{prop-swallowed}, conditioned on $(L_\cdot)_{[0,\qduration]}$ the quantum surfaces $\{(U_s, \phi)/{\sim_\gamma} \: : \: (s,x) \in J\}$ are conditionally independent, with the conditional law of $(U_s, \phi)/{\sim_\gamma}$ being $\QD(x)^\#$. Since the jump times are dense and a.s.\ each $(U_s, \phi)/{\sim_\gamma}$ has strictly positive quantum area, we conclude that the function $f: [0, \qduration] \to \R$ given by $f(s) = \cA^\gamma_\phi(K_s)$ is strictly increasing. 
 By~\eqref{eq-swallowing-trace-zero}, $\cA^\gamma_\phi(\D)$ equals the sum of all jump sizes of $f|_{[0,\qduration)}$, so $f' = 0$ almost everywhere, and $\cA^\gamma_\phi(K_\qduration) = f(\qduration) \geq \lim_{s \to \qduration} f(s) = \cA^\gamma_\phi(\D) \geq \cA^\gamma_\phi(K_\qduration)$
 implies $f$ is continuous at $\qduration$. Finally, since $\ol \D \backslash K_\qduration$ is an open set with $\cA^\gamma_\phi(\ol \D \backslash K_\qduration) = 0$, we have $\ol \D \backslash K_\qduration = \emptyset$, or equivalently $K_\qduration = \ol \D$. 
\end{proof}

\section{Open problems}\label{sec-open}
Recall that in this paper, QLE$(\gamma^2, \eta)$ without any further qualifier refers to quantum natural time QLE$(\gamma^2, \eta)$ as constructed in this work, and capacity time QLE$(\gamma^2, \eta)$ refers to the construction of \cite{ms-qle}.

Perhaps the most important (and likely most challenging) question is that of scaling limits. 

\begin{question}\label{question-scaling}
    Prove that diffusion-limited aggregation on a uniform-spanning-tree-weighted random planar map converges in the scaling limit to QLE$(2, 1)$ on $(\gamma = \sqrt2)$-LQG. 
\end{question}
In particular, the DLA process should be parametrized so one particle is attached per unit of time, and its time-parametrization should converge to the quantum natural time parametrization of QLE$(2,1)$. 
More generally, for the parameters $(\gamma, \eta)$ for which we have constructed QLE$(\gamma^2, \eta)$, we conjecture that  $\eta$-DBM on random planar maps in the $\gamma$-LQG universality class should converge in the scaling limit to QLE$(\gamma^2, \eta)$. See \cite{ms-qle} for further discussion.

A fundamental question is whether our construction of QLE$(\gamma^2, \eta)$ agrees with the capacity time construction of \cite{ms-qle}. This is not known even for $(\gamma^2, \eta) = (8/3,0)$.
\begin{question}\label{question-equiv}
Does QLE$(\gamma^2, \eta)$ agree with capacity time QLE$(\gamma^2, \eta)$?
\end{question}
Here is a precise formulation. Consider a QLE$(\gamma^2, \eta)$ process $(\phi, (K_\cdot)_{[0,\qduration]})$. Since $\phi \sim P_{\alpha, \beta, 1}$, the law of $\phi$ is absolutely continuous with respect to that of the field considered in \cite{ms-equivalence} when viewed as distributions modulo additive constant. Thus, their construction of capacity time QLE$(\gamma^2, \eta)$ gives another growth process $(\phi, (K^\mathrm{cap}_\cdot)_{[0,\infty)})$ parametrized by capacity time. The question, then, is whether $(\phi, (K_\cdot)_{[0,\qduration)})$ and $(\phi, (K^\mathrm{cap}_\cdot)_{[0,\infty)})$ agree in law when $(K_\cdot)_{[0,\qduration)}$ and $(K^\mathrm{cap}_\cdot)_{[0,\infty)}$ are viewed modulo monotone reparametrization of time. 

Question~\ref{question-equiv} only makes sense if QLE$(\gamma^2, \eta)$ is uniquely defined:

\begin{question}\label{question-unique}
    The construction of QLE$(\gamma^2, \eta)$ took a limit of $\delta_k$-QLE$(\gamma^2, \eta)$ along some subsequence $(\delta_k)_{k \in \N}$ approaching zero. Does convergence hold without passing to a subsequence? In other words, is the limiting construction unique? 
\end{question}
One can also ask whether additional randomness beyond the random environment is necessary to construct QLE$(\gamma^2, \eta)$. 
\begin{question}\label{question-measurability}
    For the QLE$(\gamma^2, \eta)$ process $(\phi, (K_\cdot)_{[0,\qduration]})$, is $(K_\cdot)_{[0,\qduration]}$ measurable with respect to $\phi$?
\end{question}
For QLE$(8/3, 0)$, Questions~\ref{question-unique} and~\ref{question-measurability} have been solved in the affirmative in \cite{ms-equivalence} in their construction of the $\sqrt{8/3}$-LQG metric. 

In the context of Question~\ref{question-equiv}, the following question is natural. 
\begin{question}\label{question-recover-time}
    Given a sample of QLE$(\gamma^2, \eta)$ modulo monotone reparametrization of time, can one recover its time parametrization? 
\end{question}
This recovery is possible in the space-filling and swallowing parameter ranges. Indeed, in the former case one can measure the $\gamma$-LQG area of the explored region to determine the time (Theorem~\ref{thm-phase}), and in the latter case one can count the number of swallowed regions with boundary length between $\eps$ and $2\eps$, renormalize, and send $\eps \to 0$ (Proposition~\ref{prop-swallowed}). In the dilute parameter range, however, the approach of measuring boundary lengths and comparing against $(L_\cdot)_{[0,\qduration]} = (1 + 2s)_{0 \leq s \leq \qduration}$ runs into the obstruction that we only have the identity $\cL^\gamma_{\phi_s}(\partial \D) = 1+2s$ for rational quantum times $s$ (Corollary~\ref{cor-length-process}), and we have not shown the monotonicity of boundary lengths (or even the existence of the $\gamma$-LQG boundary length for all times). This leads to the next question:

\begin{question}\label{question-lengths}
    Does $\cL^\gamma_{\phi_s}(\partial \D) = L_s$ simultaneously for all $s \in [0,\qduration)$ a.s.?
\end{question}

The next question is about whether QLE$(\gamma^2, \eta)$ grows continuously when viewed from 0. 
\begin{question}\label{question-continuous}
    Does QLE$(\gamma^2, \eta)$ grow continuously with respect to capacity a.s.? In other words, is the capacity reparametrization function $t(s)$ defined below Lemma~\ref{lem-coupling} continuous a.s.? 
\end{question}
For $(\gamma^2, \eta) = (8/3,0)$, the answer to Question~\ref{question-continuous} is yes because QLE$(8/3,0)$ describes metric ball growth with respect to the $\gamma = \sqrt{8/3}$ LQG metric, and this metric induces the Euclidean topology \cite{ms-equiavlence-2}. It remains open for all other values of $(\gamma^2, \eta)$. 

An answer to Question~\ref{question-continuous} would allow us to define a QLE$(\gamma^2, \eta)$ process $(\phi, (K_\cdot)_{[0,\qduration]})$ when $\phi$ is only locally absolutely continuous with respect to a sample from $P_{\alpha, \beta, 1}$. Indeed, if $\phi$ is a field such that the law of $\phi|_{\D \backslash B_\eps(0)}$ is absolutely continuous with respect to that of $\tilde\phi|_{\D \backslash B_\eps(0)}$ where $\tilde\phi \sim P_{\alpha, \beta, 1}$, then for fixed $\eps>0$ the QLE process for $\phi$ can be defined up until the time it hits $\partial B_{\eps}(0)$ by reweighting the corresponding law of the process for $\tilde \phi$ by a suitable Radon-Nikodym derivative; one can then send $\eps \to 0$ to construct the whole process. 

\begin{question}\label{question-length-process}
    Identify the law of the boundary length process $(L_\cdot)_{[0, \qduration]}$ for QLE$(\gamma^2, \eta)$ in the swallowing phase.
\end{question}
One approach to Question~\ref{question-length-process} would be to study the L\'evy process definition of generalized quantum disks, since $(L_\cdot)_{[0,\qduration]}$ agrees in law with the boundary length process in Proposition~\ref{prop-radial-mot-swallow-finite}, which is itself closely related to $\GQD$ by Proposition~\ref{prop-radial-weld-ns}. For the special case $(\gamma^2, \eta) = (8/3,0)$, \cite[Theorem 1.2]{ms-finite-mating} resolves Question~\ref{question-length-process} in the spherical setting, and the result for the disk setting then follows from the discussion in Appendix~\ref{sec-other-QLEs}.

\appendix

\section{Equivalence of QLE$(8/3,0)$ with the construction of \cite{ms-equivalence}}\label{sec-other-QLEs}

Fix $\gamma = \sqrt{8/3}$ and $\eta = 0$.  
In \cite{ms-equivalence} the authors construct quantum natural time QLE$(8/3, 0)$ in the sphere setting as a growth process on a $\sqrt{8/3}$-LQG surface from one marked point to another, by first constructing a $\delta$-QLE$(8/3,0)$ process on the sphere then passing to a subsequential limit. We will explain that their $\delta$-QLE$(8/3,0)$ construction, after any integer number of size $\delta$ steps, proceeds exactly as our $\delta$-QLE$(8/3,0)$ construction in Definition~\ref{def-delta-QLE-finite}, but with random initial boundary length rather than boundary length one. We then discuss why the two notions of QLE$(8/3,0)$ on the disk agree. 

Following \cite[Section 4.1]{ms-equivalence}, let $\mathsf{M}^2_\mathrm{SPH}$ denote the law of the \emph{quantum sphere with two marked points}, so a sample from $\mathsf{M}^2_\mathrm{SPH}$ is a $\gamma$-LQG surface with the sphere topology and two marked points. Alternatively see \cite[Section 4.5]{DMS14} for a definition (where $\mathsf{M}^2_\mathrm{SPH}$ was called $\mathcal N_{4/3}$). 
Let $\widehat {\mathbb C} = \mathbb C  \cup \{ \infty\}$ be the Riemann sphere. 

For fixed $\delta > 0$, \cite{ms-equivalence} constructs $\delta$-QLE$(8/3,0)$ as follows. Start with an embedding $(\hat \bbC, h, \infty, 0)$  of a quantum sphere sampled from $\mathsf M^2_\mathrm{SPH}$. Independently sample a whole-plane SLE$_6$ curve from $\infty$ to $0$, and explore the initial $\delta$-segment of the curve; let $D_1$ be the complementary connected component containing $0$. Iteratively, given $D_n$, sample a point $p_n$ on $\partial D_n$ according to the probability measure proportional to the $\gamma$-LQG boundary length measure, sample a radial SLE$_6$ curve in $D_n$ from $p_n$ to $0$, and explore it for time $\delta$. This cuts $D_n$ into countably many pieces; let $D_{n+1}$ be the connected component containing $0$. This is iterated until the time the SLE$_6$ curve hits 0 within time $\delta$, at which time the $\delta$-QLE$(8/3,0)$ process stops. 

We now explain how the $\delta$-QLE$(8/3,0)$ process on the sphere reduces to our $\delta$-QLE$(8/3,0)$ process on the disk after any fixed number of steps. %The key point is that after any fixed number of size $\delta$ steps, the conditional law of the unexplored region given its boundary length is exactly that of a quantum disk with one marked bulk point and one marked boundary point. This means the future evolution of the sphere process is governed by the same measures as our disk construction.
Let $s = k\delta>0$, and condition on the event that the $\delta$-QLE$(8/3, 0)$ has not terminated by time $s$. Let $U_s$ be the origin-containing complementary connected component of the growth process at time $s$, and let $x_s \in \partial U_s$ be the curve tip at time $s$. 
By \cite[Section 6.1]{ms-equivalence}, 
further conditioning on the $\gamma$-LQG boundary length $L$ of $(U_s, h)/{\sim_\gamma}$, the conditional law of $(U_s, h, 0, x_s)/{\sim_\gamma}$ is $\QD_{1,1}(L)^\#$, where $\QD_{1,1}$ denotes the law of the $\gamma$-LQG disk with one marked bulk point and one marked boundary point. By \cite[Proposition 3.9]{ars-fzz}, for a sample from $\QD_{1,1}(\ell)^\#$ embedded in $(\D, 0, 1)$, the law of the field is $(\LF_{\D, \ell}^{(\gamma, 0), (\gamma, 1)})^\#$. Thus, embedding $(U_s, h, 0, x_s)/{\sim_\gamma}$ in $(\D, 0, 1)$, since $(\alpha, \beta) = (\gamma, \gamma)$ when $\gamma =\sqrt{8/3}$, the field is as in Definition~\ref{def-delta-QLE-finite} but with random boundary length $L$ (rather than boundary length one). 
Combining, this implies that on the event that the $\delta$-QLE$(8/3,0)$ process has not terminated by time $s = k\delta$, it subsequently evolves as $\delta$-QLE$(8/3,0)$ on the disk as in Definition~\ref{def-delta-QLE-finite} with random initial boundary length. 

Thus, conditioning on survival up to time $1$, the sphere construction of \cite{ms-equivalence} yields a disk QLE$(8/3,0)$ process by viewing the evolution from time $1$ onward, and this disk QLE$(8/3,0)$ is a subsequential limit of the same $\delta$-QLE$(8/3,0)$ processes as in this work but with random initial boundary length.

\cite{ms-equivalence} defines  QLE$(8/3,0)$ by taking a subsequential limit of $\delta$-QLE$(8/3,0)$ as $\delta \to 0$ under which certain quantities converge, but the convergence is established via compactness and tightness arguments. The same is true for our subsequential limit. Consequently, whichever subsequence is taken in \cite{ms-equivalence} to construct QLE$(8/3,0)$, we can pass to a further subsequence to obtain our construction of QLE$(8/3,0)$. In this case, the two growth processes agree modulo time parametrization since they have the same driving measure. The time parametrizations also agree, since both time parametrizations can be defined as renormalized counts of small swallowed regions; see  \cite[Remark 6.4]{ms-equivalence}, Definition~\ref{def-quantum-time-swallowing} and Lemma~\ref{lem-qle-field-bdy-procress}. Similarly, given any subsequence used to construct QLE$(8/3,0)$ in our approach, we can pass to a further subsequence to implement the construction from \cite{ms-equivalence}. In this sense, the construction of QLE$(8/3,0)$ in  \cite{ms-equivalence} is equivalent to our construction of QLE$(8/3,0)$.

\bibliographystyle{plain}
\bibliography{citations}

\end{document}